\documentclass{amsart}

\usepackage{amsfonts,amssymb,latexsym}
\usepackage{amsmath,amsthm,amsxtra}
\usepackage{eucal}
\usepackage[latin1]{inputenc}
\usepackage{color}
\usepackage{graphicx}
\usepackage{subfigure}
\usepackage{hyperref}
\usepackage{amssymb}

\newtheorem{theorem}{Theorem}[section]
\newtheorem{lemma}[theorem]{Lemma}
\newtheorem{proposition}[theorem]{Proposition}
\newtheorem{corollary}[theorem]{Corollary}

\newtheorem*{theorem*}{Theorem}

\theoremstyle{definition}
\newtheorem{definition}[theorem]{Definition}

\newtheorem*{merci}{Acknowledgements}

\theoremstyle{remark}
\newtheorem{remark}{Remark}[section]

\newcommand{\R}{\mathbb{R}}
\newcommand{\ve}{\varepsilon}
\newcommand{\be}{\begin{equation}}
\newcommand{\ee}{\end{equation}}
\newcommand{\ba}{\begin{equation*}}
\newcommand{\ea}{\begin{equation*}}
\newcommand{\bea}{\begin{eqnarray}}
\newcommand{\eea}{\end{eqnarray}}
\newcommand{\bee}{\begin{eqnarray*}}
\newcommand{\eee}{\end{eqnarray*}}
\newcommand{\ben}{\begin{enumerate}}
\newcommand{\een}{\end{enumerate}}
\newcommand{\nonu}{\nonumber}

\newcommand{\al}{\alpha}

\newcommand{\ga}{\gamma}

\DeclareMathOperator{\rmax}{r_{\max}}
\DeclareMathOperator{\sech}{sech}

\usepackage{hyperref}
\usepackage{wrapfig}
\usepackage{tikz}
\usetikzlibrary{arrows,calc,decorations.pathreplacing}
\definecolor{light-gray1}{gray}{0.90}
\definecolor{light-gray2}{gray}{0.80}
\definecolor{light-gray3}{gray}{0.60}

\numberwithin{equation}{section}

\begin{document}
\title[Asymptotic Stability of solitons for the ZK equation ]{Asymptotic Stability of high-dimensional Zakharov-Kuznetsov solitons}

%\author[R. C\^ote$^*$, C. Mu\~noz, D. Pilod and G. Simpson]{Rapha\"el C\^ote, Claudio Mu\~noz, Didier Pilod$^{\dagger}$ and Gideon Simpson}

\author[R. C\^ote]{Rapha\"el C\^ote$^*$}
\thanks{$^*$Partially supported by the project ERC 291214 BLOWDISOL}
\address{Rapha\"el C\^ote, CNRS and \'Ecole polytechnique
Centre de Math\'ematiques Laurent Schwartz UMR 7640,
Route de Palaiseau, 91128 Palaiseau cedex, France}
\email{cote@math.polytechnique.fr}

\author[C. Mu\~noz]{Claudio Mu\~noz$^{**}$}
\thanks{$^{**}$ C.M. would like to thank the Laboratoire de Math\'ematiques d'{}Orsay for their kind hospitality}
\address{Claudio Mu\~noz, CNRS and Laboratoire de Math\'ematiques d'{}Orsay UMR 8628, B\^at. 425 Facult\'e des Sciences d'{}Orsay,
Universit\'e Paris-Sud F-91405 Orsay Cedex France}
\email{Claudio.Munoz@math.u-psud.fr}

\author[D. Pilod]{Didier Pilod$^{\dagger}$}
\thanks{$^{\dagger}$Partially supported by CNPq/Brazil, grants 302632/2013-1 and 481715/2012-6.}
\address{Didier Pilod, Instituto de Matem\'atica, Universidade Federal do Rio de Janeiro, Caixa Postal 68530, CEP: 21945-970, Rio de Janeiro, RJ, Brazil}
\email{didier@im.ufrj.br}

\author[G. Simpson]{Gideon Simpson}
\address{Gideon Simpson, Department of Mathematics, Drexel University, 33rd and Market Streets, Philadelphia, PA 19104, USA}
\email{simpson@math.drexel.edu}

\date{September 8th, 2015}

\begin{abstract}
We prove that solitons (or solitary waves) of the Zakharov-Kuz\-net\-sov (ZK) equation, a physically relevant high dimensional generalization of the Korteweg-de Vries (KdV) equation appearing in Plasma Physics, and having mixed KdV and nonlinear Schr\"odinger (NLS) dynamics, are \emph{strongly} asymptotically stable in the energy space. We also prove that the sum of well-arranged solitons is stable in the same space. Orbital stability of ZK solitons is well-known since the work of de Bouard \cite{deBo}.  Our proofs follow the ideas by Martel \cite{Ma} and Martel and Merle \cite{MM2}, applied for generalized KdV equations in one dimension. In particular, we extend to the high dimensional case several monotonicity properties for suitable \emph{half-portions} of mass and energy; we also prove a new Liouville type property that characterizes ZK solitons, and a key Virial identity for the linear and nonlinear part of the ZK dynamics, obtained independently of the mixed KdV-NLS dynamics. This last Virial identity relies on a simple sign condition, which is numerically tested for the two and three dimensional cases, with no additional spectral assumptions required. Possible extensions to higher dimensions and different nonlinearities could be obtained after a suitable local well-posedness theory in the energy space, and the verification of a corresponding sign condition.
\end{abstract}

\maketitle

\tableofcontents

\section{Introduction}
We are interested in the the Zakharov-Kuznetsov (ZK) equation 
\begin{equation}  \label{ZK}
\partial_tu+\partial_{x_1}\big( \Delta u+u^2\big)=0,
\end{equation}
where $u=u(x,t)$ is a real-valued function, $x=(x_1,x_2) \in \mathbb \mathbb \R \times \R^{d-1}$, $ t \in \mathbb R$, $\Delta=\sum_{j=1}^d\partial_{x_j}^2$ denotes the laplacian. The ZK equation is a particular case of the generalized Zakharov-Kuznetsov (gZK) equation 
\begin{equation}  \label{gZK}
\partial_tu+\partial_{x_1}\big( \Delta u+u^p\big)=0,
\end{equation}
where $p \in \mathbb Z_+$ is such that $2 \le p < \infty$ if $d=1,2$ and $2 \le p < 1+\frac4{d-2}$ if $d \ge 3$. We observe that when the spatial dimension $d$ is equal to $1$, equation \eqref{gZK} becomes the well-known generalized Korteweg- de Vries (gKdV) equation.

\medskip

The ZK equation was introduced by Zakharov and Kuznetsov in \cite{ZK} to describe the propagation of ionic-acoustic waves in uniformly magnetized plasma in the two dimensional and three dimensional cases.  The derivation of ZK from the Euler-Poisson system with magnetic field  in the long wave limit was carried out by Lannes, Linares and Saut in \cite{LLS}. The ZK equation was also derived by Han-Kwan \cite{HK} from the Vlasov-Poisson system in a combined cold ions and long wave limit. Moreover, the following quantities are conserved by the flow of ZK, 
\begin{equation} \label{M} 
M(u)=\int u(x,t)^2dx,
\end{equation}
and 
\begin{equation} \label{H}
H(u)=\int \big( \frac12|\nabla u(x,t)|^2-\frac1{p+1}u(x,t)^{p+1}\big)dx.
\end{equation}

The well-posedness theory for ZK and gZK has been extensively studied in the recent years. In the  two dimensional case, Faminskii proved  that the Cauchy problem associated to the ZK equation is globally well-posed in the energy space $H^1(\mathbb R^2)$ \cite{Fa}. The local well-posedness result was pushed down to $H^s(\mathbb R^2)$ for $s>\frac34$ by Linares and Pastor \cite{LP} and to $s>\frac12$ by Gr\"unrock  and Herr \cite{GH} and Molinet and the third author \cite{MP}. The best result for the ZK equation in the three dimensional case was obtained last year by Ribaud and Vento \cite{RV}. They proved local well-posedness in $H^s(\mathbb R^3)$ for $s>1$. Those solutions were extended globally in time in \cite{MP}. Note however that it is still an open problem to obtain well-posedness in $L^2(\mathbb R^2)$ and $H^1(\mathbb R^3)$ for the ZK equation. Finally, we also refer to \cite{LP,LP2,FLP,RV2,Gr} for more well-posedness results for the gZK equation with $p \ge 3$ and to \cite{Pan,BuIsMe1,BuIsMe2} for unique continuation results concerning  ZK.

\medskip

Note that if $u$ solves \eqref{gZK} with initial data $u_0$, then $u_{\lambda}(x,t)=\lambda^{2/(p-1)}u(\lambda x,\lambda^3t)$ is also a solution to \eqref{gZK} with initial data $u_{0,\lambda}(x)=\lambda^{2/(p-1)}u_0(\lambda x)$ for any $\lambda>0$. Hence, $\|u_{0,\lambda}\|_{\dot{H}^s}=\lambda^{2/(p-1)+s-d/2}\|u_{0}\|_{\dot{H}^s}$, so that the scale-invariant Sobolev space for the gZK equation is $H^{s_c(p)}(\R^d)$, where $s_c(p) = \frac{d}2-\frac2{p-1}$. In particular, the gZK equation is $L^2$-critical (or simply critical) if $p=1+\frac4d$. In the sequel, we will say that the problem is \textit{subcritical} if $p<1+\frac4d$ and \textit{supercritical} if $p>1+\frac4d$.

\subsection{The elliptic problem} For $c>0$, equation \eqref{gZK} admits special solutions of the form 
\begin{equation} \label{soliton}
u(x,t)=\mathcal{Q}_c(x_1-ct,x_2,\cdots,x_d) \quad \text{with} \quad \mathcal{Q}_c(x) \underset{|x| \to +\infty}{\longrightarrow 0},
\end{equation} 
where $\mathcal{Q}_c(x)=c^{1/(p-1)}\mathcal{Q}(c^{1/2}x)$ and $\mathcal{Q}$ satisfies 
\begin{equation} \label{E} 
-\Delta \mathcal{Q}+\mathcal{Q}-\mathcal{Q}^p=0.
\end{equation} 
Observe that $\mathcal{Q}=\mathcal{Q}_1$.

We recall the following theorem on the elliptic PDE \eqref{E}, which follows, for example, from the results of Berestycki and Lions \cite{BL} and Kwong \cite{Kw}. 
\begin{theorem*} 
%\begin{thmb}
Assume that $2 \le p < \infty$ if $d=1,2$ and $2 \le p < 1+\frac4{d-2}$ if $d \ge 3$.
Then there exists a unique positive radially symmetric solution $Q$ to \eqref{E} in $H^1(\mathbb R^d)$, which is called a \textit{ground state}. In addition, $Q \in C^{\infty}(\mathbb R^d)$, $\partial_rQ(r) <0$ for all $r>0$ and  there exists $\delta>0$ such that 
\begin{equation} \label{elliptPDE1}
|\partial^{\alpha}Q(x)| \lesssim_{\alpha} e^{-\delta|x|} \quad  \forall \, x \in \mathbb R^d, \ \forall \,\alpha \in \mathbb Z_+^d \, .
\end{equation}
%\end{thmb}
\end{theorem*}

The solutions of \eqref{gZK} of the form \eqref{soliton} with $\mathcal{Q}=Q$ are called \textit{solitary waves} or \textit{solitons}. They were proved by de Bouard in \cite{deBo} to be orbitally stable in $H^1(\mathbb R^d)$ if  $p<1+\frac4d$ and unstable for $p>1+\frac4d$. In other words, the solitary waves associated to \eqref{gZK} are orbitally stable in the subcritical case and unstable in the supercritical case. 

In the following, for any $c>0$, we will denote by $\mathcal{L}_c$ the operator which linearizes \eqref{E} around $Q_c$, \textit{i.e.}, 
\begin{equation} \label{Lc}
\mathcal{L}_c=-\Delta+c-pQ_c^{p-1} \, .
\end{equation}
In the case $c=1$, we also denote $\mathcal{L}=\mathcal{L}_1$.

Next, we gather some well-known facts about the operator $\mathcal{L}$ (see Weinstein \cite{We}).
\begin{theorem*} Assume that $2 \le p < \infty$ if $d=1,2$ and $2 \le p < 1+\frac4{d-2}$ if $d \ge 3$. 
Then, the following assertions are true.  \\

\noindent (i) $\mathcal{L}$ is a self-adjoint operator and 
\begin{equation} \label{spectrum}
\sigma_{ess}(\mathcal{L}) = [\lambda_{ess},+\infty), \quad \text{for some} \ \lambda_{ess} >0.
\end{equation}

\noindent (ii) 
\begin{equation} \label{kernel}
\ker \mathcal{L} = \text{span} \big\{ \partial_{x_j}Q \ : \ j=1,\cdots,d\big\}.
\end{equation}

\noindent (iii) $\mathcal{L}$ has a unique single negative eigenvalue $-\lambda_0$ (with $\lambda_0>0$) associated to a positive radially symmetric eigenfunction $\chi_0$. Without loss of generality, we choose $\chi_0$ such that $\|\chi_0\|_{L^2}=1$. Moreover, there exists $\tilde{\delta}>0$ such that $| \chi_0(x)| \lesssim e^{-\tilde{\delta}|x|}$, for all $x \in \mathbb R^d$. \\

\noindent (iv) Let us define 
\begin{equation} \label{LambdaQ}
\Lambda Q:=\big(\frac{d}{dc}Q_{c}\big)_{c=1}=\frac1{p-1}Q+\frac12x\cdot \nabla Q \, .
\end{equation} 
Then, 
\begin{equation} \label{LambdaQeq}
\mathcal{L} \, \Lambda Q=-Q \, ,
\end{equation}
and 
\begin{equation} \label{LambdaQeq2}
\int Q \Lambda Q=c_{p,d} \, \|Q\|_{L^2}^2 \, \quad \text{where} \quad c_{p,d}=\frac1{p-1}-\frac{d}4 \, .
\end{equation}
\end{theorem*}

\subsection{Statement of the results}

As already mentioned, de Bouard proved in \cite{deBo} that the solitary waves of the gZK equation are stable in the subcritical case in the following sense. 
\begin{theorem*}[Stability]
Assume that $2 \le p <1+\frac4d$ and that the Cauchy problem associated to \eqref{gZK} is well-posed in $H^1(\mathbb R^d)$. Let $c_0>0$. Then, there exists $\epsilon_0>0$ and $K_0>0$ such that if $u_0 \in H^1(\mathbb R^d)$ satisfies $\|u_0-Q_{c_0}\|_{H^1} \le \epsilon \le \epsilon_0$,  the solution $u$ of \eqref{gZK} with $u(\cdot,0)=u_0$ satisfies 
\begin{displaymath}
\sup_{t \in \mathbb R} \inf_{\tau \in \mathbb R^d}\|u(\cdot,t)-Q_{c_0}(\cdot-\tau)\|_{H^1} \le K_0 \epsilon \, .
\end{displaymath}
\end{theorem*}

The main result of this paper is the asymptotic stability of the family of
solitons of \eqref{ZK} in the case $d=2$. Then, we consider the stability of the multi-soliton case (see Theorem \ref{NSoliton} below).

\begin{theorem}[Asymptotic stability] \label{AsymptStab} 
Assume $d=2$. Let $c_0>0$. For any $\beta>0$, 
there exists $\epsilon_0>0$ such that if $0<\epsilon  \le \epsilon_0$ and $u \in C(\mathbb R : H^1(\mathbb R^2))$ is a solution of \eqref{ZK} satisfying 
\begin{equation} \label{AsymptStab.1}
\inf_{\tau \in \mathbb R^2}\|u(\cdot,t)-Q_{c_0}(\cdot-\tau)\|_{H^1} \le \epsilon, \quad \forall \, t \in \mathbb R \, ,
\end{equation}
then the following holds true. 

There exist $c_+>0$ with $|c_+-c_0| \le K_0\epsilon$, for some positive constant $K_0$ independent of $\epsilon_0$, and $\rho=(\rho_1,\rho_2) \in C^1(\mathbb R : \mathbb R^2)$  such that 
\begin{equation} \label{AsymptStab.2}
u(\cdot,t)-Q_{c_+}(\cdot-\rho(t)) \underset{t \to +\infty}{\longrightarrow} 0 \quad \text{in} \ H^1(x_1>\beta t) \, ,
\end{equation}
\begin{equation} \label{AsymptStab.200}
\rho_1'(t) \underset{t \to +\infty}{\longrightarrow} c_+ \quad \text{and} \quad \rho_2'(t) \underset{t \to +\infty}{\longrightarrow}0 \, .
\quad 
\end{equation}
\end{theorem}

\begin{remark}\label{1R}
It will be clear from the proof that the convergence in \eqref{AsymptStab.2} can also be obtained in regions of the form 
\be\label{ASt}
\mathcal{AS}(t,\theta):= \Big\{(x_1,x_2)\in \R^2 \ : \   x_1-\beta t +(\tan\theta)\, x_2 >0  \Big\}, \ \text{where} \ \theta \in (-\frac{\pi}3,\frac \pi 3) \, .
\ee
Note that the maximal angle of improvement $\theta\geq 0$ must be strictly less than $\frac \pi 3$ on each side of the vertical line $x_1=\beta t$ (see Figure \ref{fig:AS}). We also refer to Lemma \ref{AsymptMonotonicity} for more details and a relation with the nonlinear dynamics of the equation. Moreover, we expect the range of $\theta$ for which asymptotic stability occurs in $\mathcal{AS}(t,\theta)$ to be sharp. Indeed, it will be shown in Appendix \ref{CC} that linear plane waves of ZK exist if and only if the velocity group vector has a negative $x_1$-component and forms an angle $\theta$ with $x_2$ (as in Figure \ref{fig:AS}) satisfying $|\theta| \in [\frac{\pi}3,\frac{\pi}2]$.
\end{remark}

\begin{figure}
\begin{center}
\begin{tikzpicture}[scale=1]
\filldraw[thick, color=lightgray!25] (-1,5)--(-1,2) -- (3.5,-1) --(5,-1)-- (5,5) -- (-1,5);
\draw[thick, color=black] (-1,2) -- (3.5,-1);
\draw[thick,dashed] (0,0)--(3.33,5);
\draw[thick,dashed] (2,-2)--(2,5);
\draw[->] (-1,0) -- (5,0) node[below] {$x_1$};
\draw[->] (0,-2) -- (0,5) node[right] {$x_2$};
\node at (2,0){$\bullet$};
\node at (1.7,-0.35){$\beta t$};
\node at (1.85,0.3){$\theta$};
\draw (2,0.5) arc (90:145:0.5);
\node at (2.7,-1.9){$x_1=\beta t$};
\node at (4,-1.2){$x_1+(\tan \theta) \, x_2=\beta t $};
\node at (0.3,0.15){$\theta$};
\draw (0.5,0) arc (0:55:0.5);
%\node at (0.1,1.1){$\theta$};
%\draw (0,0.9) arc (270:325:0.4);
\node at (4.3,4.5){$\mathcal{AS}(t,\theta)$};
\end{tikzpicture}
\end{center}
\caption{$\mathcal{AS}(t,\theta):= \Big\{(x_1,x_2)\in \R^2 \ : \   x_1-\beta t +(\tan\theta)\, x_2 >0  \Big\}$.} \label{fig:AS}
\end{figure}
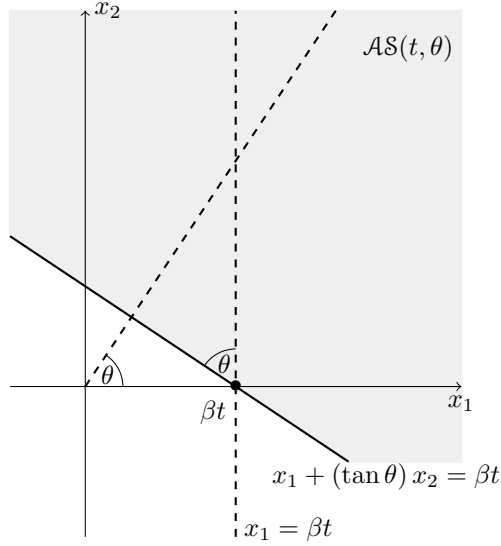
 
\begin{remark} \label{2R}
The angle $\theta=\frac{\pi}3$ is also related to the linear part of ZK. In \cite{CKZ}, Carbery, Kenig and Ziesler proved that
$$\big\||K(D)|^{\frac18}e^{-t\partial_{x_1}\Delta}\varphi\big\|_{L^4_{xyt}}\lesssim \|\varphi\|_{L^2} \, ,$$ where  $|K(D)|^{\frac18}$ is the Fourier multiplier associated to the symbol
$|K(k_1,k_2) |^{\frac18}=|3k_1^2-k_2^2 |^{\frac18}$. This Strichartz estimate was used in \cite{MP} to improve the well-posed results for ZK at low regularity. Note that the multiplier $|K(k_1,k_2) |^{\frac18}$ cancels out along the cone  $|k_2|=\tan(\frac{\pi}3) |k_1|$. We also refer to Apendix \ref{CC} for an interesting relation between the angle $\theta=\frac{\pi}3$ and the linear plane waves of ZK.
\end{remark}

\begin{remark} 
Our proof does not rely on the structure of the nonlinearity of \eqref{ZK} (\textit{i.e.} $\partial_{x_1}(u^2)$) neither on the dimension $d$. Actually, our main theorem could be extended to \eqref{ZK} in dimension $d=3$ or to the following generalization of gZK
\begin{equation} \label{ggZK}
\partial_tu+\partial_{x_1}(\Delta u+|u|^{p-1}u)=0 \, , 
\end{equation}
where $p$ is a real number $1 < p <1+\frac4d$ under the following conditions:
\begin{itemize}
\item{} The Cauchy problem associated to \eqref{ZK} with $d=3$ or to \eqref{ggZK} is well-posed in $H^1(\mathbb R^d)$.

\item{} The spectral condition $\int \mathcal{L}^{-1}\Lambda Q \Lambda Q  <0$ holds true. (Note that $\mathcal{L}^{-1}\Lambda Q$ makes sense since $\Lambda Q$ is radial and orthogonal to $\nabla Q$, and we choose $\mathcal{L}^{-1}\Lambda Q$ orthogonal to $\nabla Q$.)
\end{itemize}
This spectral condition was shown in the appendix to be true in dimension $d=2$ for $2 \le p <p_2$, where $p_2$ is a real number satisfying $2<p_2<3$. 

On the other hand, in dimension $d=3$, it is shown in the appendix that $\int \mathcal{L}^{-1}\Lambda Q, \Lambda Q  >0$. Note however that in this case, one could try to verify the more general property: the operator $\mathcal{L}$ restricted to the space $\big\{\text{ker}\, \mathcal{L}, \Lambda Q \big\}^{\perp}$ is positive definite.
\end{remark}

\begin{remark}
The case $p=3$ in dimension $d=2$ is $L^2$ critical, so that solitons should be unstable (see \cite{MM4} for example), and the validity of Theorem \ref{AsymptStab} is not clear at that level. 
\end{remark}

The proof of Theorem \ref{AsymptStab} is based on the following rigidity result for the solutions of \eqref{ZK} in spatial dimension $d=2$ around the soliton $Q_{c_0}$ which are uniformly localized in the direction $x_1$. 
\begin{theorem}[Nonlinear Liouville property around $Q_{c_0}$] \label{NonLinearLiouville} 
Assume $d=2$. Let $c_0>0$. There exists $\epsilon_0>0$ such that if $0<\epsilon  \le \epsilon_0$ and $u \in C(\mathbb R : H^1(\mathbb R^2))$ is a solution of \eqref{ZK} satisfying for some function $\rho(t)=\big(\rho_1(t),\rho_2(t)\big)$ and  some positive constant $\sigma$
\begin{equation} \label{NonLinearLiouville1}
\|u(\cdot+\rho(t))-Q_{c_0}\|_{H^1} \le \epsilon, \quad \forall \, t \in \mathbb R \, ,
\end{equation}
and
 \begin{equation} \label{NonLinearLiouville2} 
  \int_{x_2}u^2(x_1+\rho_1(t),x_2+\rho_2(t),t)dx_2  \lesssim e^{-\sigma |x_1|}\, , \quad \forall \, (x_1,t) \in \mathbb R^2 \, ,
 \end{equation}
 then, there exist $c_1>0$ (close to $c_0$) and $\rho^0=(\rho^0_{1},\rho^0_{2}) \in \mathbb R^2$ such that 
 \begin{equation} \label{NonLinearLiouville3}
u(x_1,x_2,t)=Q_{c_1}(x_1-c_1t-\rho^0_{1},x_2-\rho^0_{2}) \, .
 \end{equation}
\end{theorem}

\begin{remark}
Due to the stability result of de Bouard \cite{deBo}, Theorems  \ref{AsymptStab}  and \ref{NonLinearLiouville}  still hold true if we assume that 
\begin{equation} \label{NonLinearLiouville4}
\|u_0-Q_{c_0}\|_{H^1} \le\epsilon \, ,
\end{equation}
instead of \eqref{AsymptStab.1} and \eqref{NonLinearLiouville1}.
\end{remark}

\begin{remark} Theorem \ref{NonLinearLiouville}  still holds true if we replace assumption \eqref{NonLinearLiouville2} by the weaker assumption that  the solution $u$ is \textit{$L^2$-compact} in the $x_1$ direction, \textit{i.e.}: 
\begin{displaymath}
\forall \, \epsilon>0, \ \exists \, A>0 \ \text{such that} \quad \sup_{t \in \mathbb R}\int_{  |x_1|>A}u^2(x+\rho(t),t)dx \le \epsilon \, .
\end{displaymath}
\end{remark}

\medskip

We also prove a rigidity theorem for the solutions of the linearized gZK (or \eqref{ggZK}) equation in spatial dimension $d=2$  around $Q_c$ which are uniformly localized in the direction $x_1$. 

\begin{theorem}(Linear Liouville property around $Q_{c_0}$) \label{LinearLiouville} 
Assume $d=2$. There exists $2<p_2<3$ such that for all $2 \le p <p_2$, the following holds true. Let $c_0>0$ and $\eta \in C(\mathbb R : H^1(\mathbb R^2))$ be a solution to 
\begin{equation} \label{LinearLiouville1}
\partial_t\eta=\partial_{x_1}\mathcal{L}_{c_0}\eta
 \quad \text{on} \ \mathbb R^2 \times \mathbb R \, , 
 \end{equation}
 where $\mathcal{L}_{c_0}$ is defined in \eqref{Lc}.
 Moreover, assume that there exists a constant $\sigma>0$ such that 
 \begin{equation} \label{LinearLiouville2} 
  \int_{x_2}\eta^2(x_1,x_2,t)dx_2  \lesssim e^{-\sigma |x_1|}\, , \quad \forall \, (x_1,t) \in \mathbb R^2 \, .
 \end{equation}
 Then, there exists $(a_1,a_2) \in \mathbb R^2$ such that 
 \begin{equation} \label{LinearLiouville3}
 \eta(x,t)=a_1 \partial_{x_1}Q_{c_0}(x) +a_2 \partial_{x_2}Q_{c_0} (x),\quad\forall \, (x,t) \in \mathbb R^3 \, .
 \end{equation}
\end{theorem}

\begin{remark}
It will be clear from the proof (\textit{c.f.} Remark \ref{rema.mono}) that  Theorem \ref{LinearLiouville}  still holds true if we replace assumption \eqref{LinearLiouville2} by the weaker assumption that  the solution $\eta$ is \textit{$L^2$-compact} in the $x_1$ direction, \textit{i.e.}: $\eta \in C_b(\mathbb R : H^1(\mathbb R^2))$ and 
\begin{displaymath}
\forall \, \epsilon>0, \ \exists \, A>0 \ \text{such that} \quad \sup_{t \in \mathbb R}\int_{  |x_1|>A}\eta^2(x,t)dx \le \epsilon \, .
\end{displaymath}
\end{remark}

\begin{remark} \label{remarkscaling}
By using the scaling invariance of \eqref{ZK}, it is enough to prove Theorems \ref{LinearLiouville}, \ref{NonLinearLiouville} and \ref{AsymptStab} in the case where $c_0=1$. 
\end{remark}

\medskip

Recall that the first result of asymptotic stability of solitons for generalized KdV equations was proved by Pego and Weinstein \cite{PW} in weighted spaces (see also \cite{Mi} for some refinements on the weights). In \cite{MM1}, Martel and Merle have given the first asymptotic result for the solitons of gKdV in the energy space $H^1$. They improved their result in \cite{MM5} and generalized it to a larger class of nonlinearities than the pure power case in \cite{MM2}. 

Their proof relies on a Liouville type theorem for $L^2$-compact solutions around a soliton (similar to Theorem \ref{NonLinearLiouville} in one dimension). Then, it is proved that a solution near a soliton converges (up to subsequence) to a limit object, whose emanating solution satisfies a good decay property. Due to the rigidity result, this  limit object has to be a soliton. 

It is worth noting that this technique of proof was also adapted to prove asymptotic stability in the energy space for other one dimensional models such as the BBM equation \cite{ElDi} and the BO equation \cite{KeMa}.

We also refer to \cite{MeVe,AlMuVe,MiTz} for stability results for KdV and mKdV in $L^2$ and to \cite{BGS,GraSme} for asymptotic results for the Gross-Pitaevskii equation in one dimension. For other results on asymptotic stability for nonlinear Schr\"odinger and wave equations, see \cite{SW,BP,KS,KK} and references therein.
\medskip

{\bf About the proofs. Comparison with previous results.} When proving Theorems \ref{AsymptStab}, \ref{NonLinearLiouville} and \ref{LinearLiouville}, we generalize the ideas of Martel and Merle \cite{MM1,MM5,MM2} and Martel \cite{Ma} to a multidimensional model. However, compared with these previous results, the higher dimensional case describing the ZK dynamics presents new challenges, that we explain in the following lines. 

\medskip

First of all, as far as we know, our results represent the first two dimensional model where asymptotic stability is proved, in the energy space, and with no nonstandard spectral assumptions on the linearized dynamical operator.  As we expressed before, we only need to check the numerical condition
\be\label{N_C}
\int \mathcal{L}^{-1}\Lambda Q \Lambda Q  <0.
\ee
Obtaining a direct proof of this result seems far from any reasonable approach because the soliton $Q$, and therefore, the function $ \mathcal{L}^{-1}\Lambda Q$, have no closed and explicit forms. This is the first difference with respect to the one dimensional case: we work with a solitary wave that is \emph{not explicit} at all. 

\medskip

We will see through the proofs that ZK behaves as a KdV equation in the $x_1$ direction, and as a \emph{nonlinear Schr\"odinger} (NLS) equation in the $x_2$ variable. In particular, we are able to prove monotonicity properties (see Lemma \ref{nlL2monotonicity}) along the $x_1$ direction and along a \emph{slighted perturbed cone} around the $x_1$ direction (Lemma \ref{AsymptMonotonicity}). This last result is, to our knowledge, new in the literature and makes use of the geometrical properties of the nonlinear ZK dynamics around a solitary wave. Remark \ref{1R} and the asymptotic stability result inside the set (see \eqref{ASt})
\[
\mathcal{AS}(t,\theta):= \Big\{(x_1,x_2)\in \R^2 \ : \   x_1-\beta t +(\tan\theta)\, x_2 >0  \Big\}, \ \text{where} \ \theta \in (-\frac{\pi}3,\frac \pi 3),
\]
are deep consequences of these geometrical properties. Recall that such rich foliations are not present in the one dimensional case. We also complement our results by a simple  linear analysis leading to the same formal conclusions, carried out in Appendix \ref{CC}. 

\medskip

Another barrier that appears in the higher dimensional case is the \emph{lack of $L^\infty$ control} on the solution if we only assume $H^1$ bounds. We need such a control to ensure pointwise exponential decay around solitons at infinity for a compact part of the solution. In the one dimensional case, the proof of this fact is direct from the Sobolev embedding. However, since $H^1$ is not contained in $L^\infty$ in $\R^2$, we must prove new monotonicity properties at the $H^2$ level (cf. Lemma \ref{nlH2monotonicity}), which are obtained by proving new energy estimates.

\medskip

No monotonicity property seems to hold for the $x_2$ direction, mainly because of the conjectured existence of  trains of small solitons moving to the right in $x_1$ but without restrictions on the $x_2$ coordinate. From the point of view associated to the $x_2$ variable, such solutions represent movement of mass along the $x_2$ direction without a privileged dynamics. In particular, no asymptotic stability result is expected for a half-plane involving the $x_2$ variable only (see Fig. \ref{fig:2}). This is the standard situation  in many $2d$ models like KP-I and NLS equations. However, here we are able to prove the asymptotic stability of ZK solitons because the KdV dynamics is exactly enough to control the movement of mass along the $x_2$ direction.

\begin{figure}
\begin{center}
\begin{tikzpicture}[
	>=stealth',
	axis/.style={semithick,->},
	coord/.style={dashed, semithick},
	yscale = 1,
	xscale = 1]
	\newcommand{\xmin}{0};
	\newcommand{\xmax}{9};
	\newcommand{\ymin}{0};
	\newcommand{\ymax}{5};
	\newcommand{\ta}{3};
	\newcommand{\fsp}{0.2};
	\draw [axis] (\xmin-\fsp,0) -- (\xmax,0) node [right] {$x_1$};
	\filldraw[color=light-gray3] (8,1) circle (0.4);
	\filldraw[color=light-gray2] (0,2) rectangle (9,5); 
	\draw [axis] (0,\ymin-\fsp) -- (0,\ymax) node [below left] {$x_2$};
	\draw [thick,->] (8,1) -- (9,1);
	\filldraw[color=light-gray2] (5,1) circle (0.5); 
	\draw [thick,->] (5,1) -- (5.6,1);
	\draw [thick,-] (-0.2,2) -- (9,2);
	%\filldraw[color=light-gray1] (1,1) circle (0.6); 
	%\draw [thick,->] (1,1) -- (1.4,1);
	\filldraw[color=light-gray1] (2,4) circle (0.7); 
	\draw [thick,->] (2,4) -- (2.5,4);
	%\draw [<->] (-0.2,1) -- (-0.2,4);
	\draw (6.2,-0.5) node [left] {$c_1<x_1<c_2$};
	%\draw [dashed] (-0.2,1) -- (9,1);
	%\draw [dashed] (-0.2,4) -- (2,4);
	\draw [dashed] (4,0) -- (4,5);
	\draw [dashed] (6,0) -- (6,5);
	\draw (3,2) node [above] {$x_2>c_0$};
	%\draw [dashed] (1,1) -- (1,2.5);
	%\draw [dashed] (5,1) -- (5,2.5);
	%\draw [<->] (5,2.5) -- (8,2.5);
	%\draw (6.5,2.5) node [above] {$L$};
	%\draw [dashed] (8,1) -- (8,2.5);
\end{tikzpicture}
\end{center}
\caption{A schematic example of why no asymptotic stability is expected to hold on the $x_2$ direction. The band in the $x_2$ variable defined by fixed $c_1<x_1<c_2$ has increasing and decreasing variation of mass along time. Faster solitons are darker and more concentrated; speed is commensurate with arrow length.}\label{fig:2}
\end{figure}
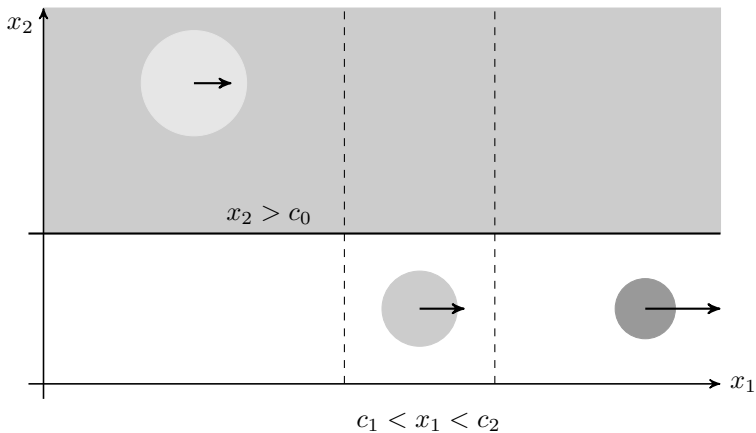

\medskip

The second ingredient in the proofs of Theorems \ref{AsymptStab}, \ref{NonLinearLiouville} and \ref{LinearLiouville} is a \emph{new virial identity} in higher dimensions (cf. \eqref{LinearLiouville9}), which holds only for the half space $\{x_1> c_0 t\}$ and for $p=2$ and slightly larger. Compared with the previous works by Martel and Merle, the additional dimensions make things harder because they induce transversal variations that seem to destroy any virial-type inequality. In order to overcome this difficulty, we use a different orthogonality condition for the function $v$ employed in the virial (see Lemma \ref{coercivity}): 
\be\label{C_O}
\int v\Lambda Q =0.
\ee
We emphasize that this condition is somehow natural and necessary if we want to get full control of the $x_2$ perturbations appearing from the variations of the virial terms. Without using this modified condition, any form of two-dimensional virial identity is no longer true. 

\medskip 

Here is when the nonstandard spectral condition \eqref{N_C} appears: under the orthogonality condition \eqref{C_O}, the \emph{virial identity  holds provided \eqref{N_C} is satisfied}. We prove that \eqref{N_C} holds for the case $p=2$ and $d=2$, as expressed by some numerical computations obtained in Appendix \ref{A}. This condition in fact generalizes the Martel one in \cite{Ma} and seems to be the natural one for the $2d$ case, as described when proving the nonlinear stability result (see \eqref{NonLinearLiouville8a}-\eqref{NonLinearLiouville9} for instance). It is worth noting that this condition has already been used by Kenig and Martel in the Benjamin-Ono context \cite{KeMa}, for different reasons. For powers of the nonlinearity which are definitely larger than $2$, or just the three dimensional case for $p=2$, we have a strong instability effect at the level of the previous spectral theory, probably associated to the dynamics around the soliton in the $x_2$ variable, and the virial identity seems no longer to hold.\footnote{In the one dimensional case, this instability condition does not appear, see Martel \cite{Ma}.} Once again, a good understanding of the dynamics for powers close the the critical case $p=3$ or supercritical  as in \cite{Ma} needs a deep extension of \eqref{LinearLiouville9} by incorporating now the dynamics in the $x_2$ variable, which could be very complicated, in view of some results by del Pino et al. \cite{delP}.  The extension of the ideas introduced by Martel \cite{Ma} to any power of $p$ seems a very interesting problem. 

\medskip

Finally, we mention that another crucial application of the monotonicity formula on \emph{perturbed cones}, needed in the higher dimensional case, is given in Lemma \ref{L2CV}. Here, a new \emph{compact} region $\mathcal R$ of the plane is introduced, outside of which we prove \emph{exponential decay}. This set is constructed in order to prove the strong convergence of sequences of bounded solutions, thanks to the use of the Sobolev embedding theorem. 

\medskip

One can also ask for the nonlinear dynamics in the remaining part of the plane, namely the region $\mathcal{AS}(t,\theta)^c$, see \eqref{ASt}. We believe that in addition to radiation, one can find small solitons $Q_c$ moving to the right in a very slow fashion. No finite energy solitary waves with speed along the $x_2$ direction are present, as shows the following (general) definition and result. As usual, we define the symbol $ \partial_{x_1}^{-1}\partial_{x_j}$ by using its corresponding Fourier representation $\xi_1^{-1}\xi_j \mathcal F(\cdot )$.

\begin{definition}
We say that $v \in H^1(\R^d)$, $ \partial_{x_1}^{-1}\partial_{x_2}v, \cdots, \partial_{x_1}^{-1}\partial_{x_d}v\in L^2(\R^d)$ is the profile of a solitary wave of speeds $(c_1,c_2,\cdots, c_d)\in \mathbb R^d$  if 
\[
u(x_1,x_2, \cdots, x_d,t) := v(x_1-c_1 t, x_2-c_2 t,\cdots, x_d-c_d t), \quad v\not\equiv 0,
\]
is solution of \eqref{gZK}.
\end{definition}

Note that such a $v$ must satisfy the equation in $\R^d$
\be\label{SW}
\Delta v -c_1 v + v^p -\sum_{j=2}^dc_j \partial_{x_1}^{-1}\partial_{x_j} v=0 \, .
\ee

\begin{theorem}\label{NonExistence}
Assume that  $c_j\neq 0$ for some $j \in \{2, 3, \dots, d \}$. Then \eqref{SW} has no finite energy solutions. 
\end{theorem}

We prove this result in Appendix \ref{D}, using adapted Pohozahev identities.

\medskip

Finally, as a consequence of the monotonicity properties associated to the linear part of the dynamics, in particular, using Lemma \ref{monotonicity}, we are able to prove the stability of the sum of $N$ essentially non-colliding solitons.

\begin{definition}\label{Ldec}
Let $N \ge 2$ be an integer and $L \ge 0$. Consider $N$ solitons with scalings $c_1^0, \dots, c_N^0 >0$ and centers $\rho^{1,0}, \dots, \rho^{N,0} \in \R^2$, where $\rho^{j,0} =(\rho^{j,0}_1,\rho^{j,0}_2)$. 
We say that these $N$ solitons are $L$-decoupled if 
\be\label{Condition}
 \inf \left\{ \big|((c_k^0 - c_j^0) t ,0) +  \rho^{k,0} - \rho^{j,0}\big| \mid j \ne k, \ t \ge 0 \right\} \ge L,
\ee
that is, the solitons centers remains separated by a distance of at least $L$ for positive times. (See Fig. \ref{fig:3} below.) 
\end{definition}

$L$-decoupled solitons can be characterized by a condition on the initial data only, at least up to a constant in $L$: indeed, one can check that if, for all $j\ne k$, we have either:
\begin{itemize}
\item $|\rho^{j,0}_2 - \rho^{k,0}_2| \ge L$, or
\item $c_k^0 >c_j^0$ and $\rho^{k,0}_1 - \rho^{j,0}_1  \ge L$,
\end{itemize}
then the $N$ solitons are $L$-decoupled.

\begin{theorem}[Stability of the sum of $N$ decoupled solitons]\label{NSoliton}

Assume $d=2$. Consider a set of $N$ solitons of the form
\[
Q_{c_1^0}(x-\rho^{1,0}),  \, Q_{c_2^0}(x-\rho^{2,0}), \ldots,  \, Q_{c_N^0}(x-\rho^{N,0}),
\] 
where each $c_j^0$ is a fixed positive scaling, $c_j^0 \neq c_k^0$ for all $j\neq k$, and $\rho^{j,0} =(\rho^{j,0}_1,\rho^{j,0}_2) \in \R^2$. Assume that the $N$ solitons are $L$-decoupled, in the sense of Definition \ref{Ldec}.
%Assume additionally the following non-collision condition: 
%\be\label{Condition}
%\hbox{ if for $j< k$ one has }  \rho^{j,0}_2 =\rho^{k,0}_2, \hbox{ then } c_j^0 <c_k^0 \hbox{ and }   \rho^{j,0}_1 < \rho^{k,0}_1.
%\ee
Then there are $\ve_0>0$, $C_0>0$ and $L_0>0$ depending on the previous parameters such that, for all $\ve \in (0,\ve_0)$, and for every $L>L_0$, the following holds.   Suppose that $u_0\in H^1(\R^2)$ satisfies
\be\label{InitialCondition}
\|u_0 - \sum_{j=1}^N Q_{c_j^0}(x-\rho^{j,0}) \|_{H^1} <\ve.
\ee
%and for $j\neq k$, the parameter
%\[
%\delta^{j,k,0} := 
%\begin{cases}  
%| \rho^{j,0}_1 - \rho^{k,0}_1| & \hbox{ if } \rho^{j,0}_2 =\rho^{k,0}_2,\\
%| \rho^{j,0}_2 - \rho^{k,0}_2| & \hbox{ otherwise},\\
%\end{cases}
%\]
%obeys the condition
%\be\label{ConditionL}
%\min \{ \delta^{j,k,0} \ : \ j\neq k \} >L.
%\ee
Then there are $\ga_1>0$ fixed and $\rho^{j}(t) \in \R^2$ defined for all $t\geq 0$ such that $u(t)$, solution of \eqref{ZK} with initial data $u(0) =u_0$ satisfies
\be\label{StabilityN}
\sup_{t\geq 0}\|u(t) - \sum_{j=1}^N Q_{c_j^0}(x-\rho^{j}(t)) \|_{H^1} <C_0(\ve + e^{-\ga_1 L}).
\ee

\end{theorem}

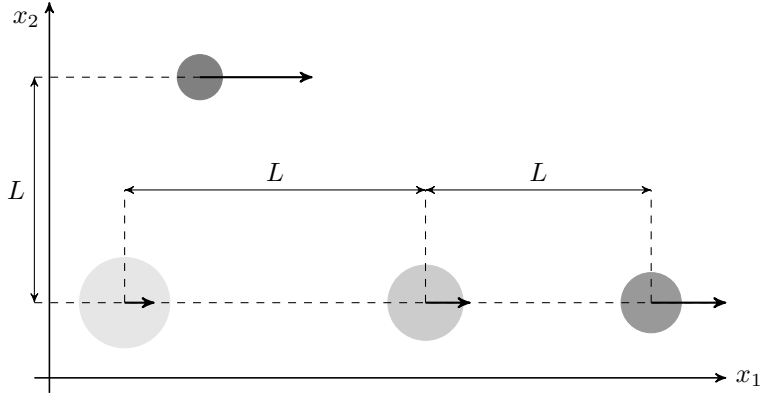
\begin{figure}
\begin{center}
\begin{tikzpicture}[
	>=stealth',
	axis/.style={semithick,->},
	coord/.style={dashed, semithick},
	yscale = 1,
	xscale = 1]
	\newcommand{\xmin}{0};
	\newcommand{\xmax}{9};
	\newcommand{\ymin}{0};
	\newcommand{\ymax}{5};
	\newcommand{\ta}{3};
	\newcommand{\fsp}{0.2};
	\draw [axis] (\xmin-\fsp,0) -- (\xmax,0) node [right] {$x_1$};
	\draw [axis] (0,\ymin-\fsp) -- (0,\ymax) node [below left] {$x_2$};
	\filldraw[color=light-gray3] (8,1) circle (0.4); 
	\draw [thick,->] (8,1) -- (9,1);
	\filldraw[color=light-gray2] (5,1) circle (0.5); 
	\draw [thick,->] (5,1) -- (5.6,1);
	\filldraw[color=light-gray1] (1,1) circle (0.6); 
	\draw [thick,->] (1,1) -- (1.4,1);
	\filldraw[color=gray] (2,4) circle (0.3); 
	\draw [thick,->] (2,4) -- (3.5,4);
	\draw [<->] (-0.2,1) -- (-0.2,4);
	\draw (-0.2,2.5) node [left] {$L$};
	\draw [dashed] (-0.2,1) -- (9,1);
	\draw [dashed] (-0.2,4) -- (2,4);
	\draw [<->] (1,2.5) -- (5,2.5);
	\draw (3,2.5) node [above] {$L$};
	\draw [dashed] (1,1) -- (1,2.5);
	\draw [dashed] (5,1) -- (5,2.5);
	\draw [<->] (5,2.5) -- (8,2.5);
	\draw (6.5,2.5) node [above] {$L$};
	\draw [dashed] (8,1) -- (8,2.5);
\end{tikzpicture}
\end{center}
\caption{A schematic example of admissible initial data. Solitons are represented by the disk where their mass is concentrated.}\label{fig:3}
\end{figure}

The proof of this result is obtained by adapting the ideas by Martel, Merle and Tsai \cite{MMT} for the generalized, one dimensional KdV case. Note that  we do not need strictly well-prepared initial data as in \cite{MMT}. Instead, from \eqref{Condition} we just need sufficient well-separated solitons in the $x_2$ variable (no particular order at the beginning), and in the case where solitons have the same $x_2$ coordinate, we ask for well-ordered solitons to avoid multi-collisions.

\bigskip

The rest of the paper is organized as follows. The linear and nonlinear Liouville properties (Theorems \ref{LinearLiouville} and \ref{NonLinearLiouville}) are proved respectively in Section \ref{sectionLL} and \ref{NonlinearRigidity}. The nonlinear Liouville property is used to show Theorem \ref{AsymptStab} in Section \ref{sectionAS}. Section \ref{sectionNSoliton} is devoted to the proof of Theorem \ref{NSoliton}. Finally, in  Appendix \ref{A}, we present some numerical computations which establish the negativity of a scalar product in the case $p=d=2$. Recall that this condition is a crucial element in the proofs of the rigidity results (Theorems \ref{NonLinearLiouville} and \ref{LinearLiouville}). We also make an interesting observation about the plane wave solutions of the linear part of ZK in Appendix \ref{CC} and give the proof of Theorem \ref{NonExistence} in Appendix \ref{D}.

\section{Linear Liouville property} \label{sectionLL}

This section is devoted to the proof of Theorem \ref{LinearLiouville}. According to Remark \ref{remarkscaling}, we will assume in this section that $c_0=1$.

\subsection{Monotonicity} In this subsection, we prove a monotonicity formula for the solutions of \eqref{LinearLiouville1} satisfying \eqref{LinearLiouville2}. 

Let $L$ denote a positive number such that $L \ge 4$.  We define $\psi_L \in C^{\infty}(\mathbb R : \mathbb R)$ by 
\begin{equation} \label{psi0}
\psi_L(y)=\frac2{\pi}\arctan (e^{y/L}),
\end{equation}
so that $\lim_{y \to +\infty}\psi_L(y)=1$ and $\lim_{y \to -\infty}\psi_L(y)=0$. Note also that 
\begin{equation} \label{psi}
\psi_L'(y)=\frac1{\pi L \cosh (y/L) } \quad \text{and} \quad |\psi'''(y)| \le \frac1{L^2}\psi_L'(y)\le \frac1{16}\psi_L'(y) \, .
\end{equation}

\begin{lemma} \label{monotonicity}
Let  $t_0 \in \mathbb R$, $y_0>0$ and $\tilde{x}_1=x_1-\frac12(t_0-t)-y_0$ for $t \le t_0$. Let $\eta \in C(\mathbb R : H^1(\mathbb R^2))$ be a solution of \eqref{LinearLiouville1} satisfying condition \eqref{LinearLiouville2}. Define $L_0=\max\{4,\frac1{(p-1)\delta}\}$, where $\delta$ is the positive number in \eqref{elliptPDE1}. Then
\begin{equation} \label{monotonicity1}
\int(\partial^{\alpha}\eta)^2(x,t_0)\psi_L(x_1-y_0)dx+
\int_{-\infty}^{t_0}\int\big( |\nabla\partial^{\alpha}\eta|^2+ (\partial^{\alpha}\eta)^2\big)(x,t)\psi'_L(\tilde{x}_1)dxdt \lesssim_L e^{-y_0/L} ,
\end{equation}
for all multi-index $\alpha \in \mathbb N^2$ and all positive number $L \ge L_{0}$.
\end{lemma}

\begin{remark} 
The implicit constant appearing in \eqref{monotonicity1} depends only on $L$. In particular it does not depend on $y_0$ and $t_0$.
\end{remark}

\begin{remark} \label{rema.mono}
It will be clear from the proof below that  Lemma \ref{monotonicity} and thus Theorem \ref{LinearLiouville}  still hold true if we replace \eqref{LinearLiouville2} by the weaker assumption that  the solution $\eta$ is \textit{$L^2$-compact} in the $x_1$ direction, \textit{i.e.}: $\eta \in C_b(\mathbb R : H^1(\mathbb R^2))$ and 
\begin{displaymath}
\forall \, \epsilon>0, \ \exists \, A>0 \ \text{such that} \quad \sup_{t \in \mathbb R}\int_{  |x_1|>A}\eta^2(x,t)dx \le \epsilon \, .
\end{displaymath}
\end{remark}

\begin{proof} First, we prove \eqref{monotonicity1} for $\alpha=(0,0)$. Fix  $L \ge L_0$. Observe from condition \eqref{LinearLiouville2} that 
\begin{equation} \label{monotonicity1b}
\sup_{t \in \mathbb R} \int \eta(x,t)^2dx \le C \, .
\end{equation}
By using the equation \eqref{LinearLiouville1}, integrations by parts and the inequality in \eqref{psi}, we compute that 
\begin{equation} \label{monotonicity2}
\begin{split} 
\frac{d}{dt}\int\eta^2\psi_L(\tilde{x}_1)dx &  =2\int\eta\partial_t\eta \psi_L(\tilde{x}_1)dx+\frac12\int\eta^2\psi_L'(\tilde{x}_1)dx
\\&\le -\int \big(3(\partial_{x_1}\eta)^2+(\partial_{x_2}\eta)^2\big)\psi_L'(\tilde{x}_1)dx-\frac14 \int \eta^2\psi_L'(\tilde{x}_1)dx
\\ & \quad   + p\int \eta^2\big(-\partial_{x_1}(Q^{p-1})\psi_L(\tilde{x}_1)+Q^{p-1}\psi_L'(\tilde{x}_1)\big)dx \, .
\end{split}
\end{equation}

To deal with the last term appearing on the right-hand side of \eqref{monotonicity2}, let us define 
\begin{equation} \label{monotonicity3}
 \mathcal{T}_0(\eta)=p\int \eta^2\big(-\partial_{x_1}(Q^{p-1})\psi_L(\tilde{x}_1)+Q^{p-1}\psi_L'(\tilde{x}_1)\big)dx \, .
 \end{equation}
 We claim that 
\begin{equation} \label{monotonicity4}
\big|\mathcal{T}_0(\eta)\big| \le C e^{-(\frac12(t_0-t)+y_0)/L}+\frac18\int \eta^2\psi_L'(\tilde{x}_1)dx \, .
\end{equation}
To prove \eqref{monotonicity4}, we argue as in Lemma 5 in \cite{Ma}. Recall from \eqref{elliptPDE1}  and \eqref{psi0} that 
\begin{displaymath}
\big|\partial_{x_1}(Q^{p-1})\psi_L(\tilde{x}_1)\big|+\big|Q^{p-1}\psi_L'(\tilde{x}_1)\big| \lesssim e^{-\delta(p-1)|x|}\psi_L(\tilde{x}_1) \lesssim e^{-\delta(p-1)|x_1|}\psi_L(\tilde{x}_1) \, ,
\end{displaymath}
where $\delta$ is the positive constant appearing in \eqref{elliptPDE1}. Let  $R_1>0$ to be fixed later. We consider the three following cases. 

\noindent \textit{Case: $x_1<R_1$.} Then $\tilde{x}_1<R_1-\frac12(t_0-t)-y_0$, so that 
\begin{displaymath}
e^{-\delta(p-1)|x_1|}\psi_L(\tilde{x}_1) \lesssim \psi_L(\tilde{x}_1) \le e^{\tilde{x}_1/L} \le e^{(R_1-\frac12(t_0-t)-y_0)/L} \, .
\end{displaymath}

\noindent \textit{Case: $R_1<x_1<\frac12(t_0-t)+y_0$.} Then, we get that
\begin{displaymath}
e^{-\delta(p-1)|x_1|}\psi_L(\tilde{x}_1) \lesssim e^{-\delta (p-1)x_1}\psi_L'(\tilde{x}_1) \le e^{-\delta (p-1)R_1} \psi_L'(\tilde{x}_1)  \, ,
\end{displaymath}
since $\psi_L(\tilde{x}_1) \lesssim \psi_L'(\tilde{x}_1)$ for $\tilde{x}_1<0$.

 \noindent \textit{Case: $\frac12(t_0-t)+y_0<x_1$.} In this case 
 \begin{displaymath}
e^{-\delta (p-1)|x_1|}\psi_L(\tilde{x}_1) \lesssim e^{-\delta (p-1) (\frac12(t_0-t)+y_0)}\le   e^{- (\frac12(t_0-t)+y_0)/L} \, ,
\end{displaymath}
since $L \ge\frac1{(p-1)\delta}$. 

We deduce then that
\begin{displaymath}
\big|\mathcal{T}_0(\eta)\big| \le Ce^{(R_1-\frac12(t_0-t)-y_0)/L}\int \eta^2dx+Ce^{-\delta (p-1) R_1}\int \eta^2\psi_L'(\tilde{x}_1)dx \, ,
\end{displaymath}
which yields estimate \eqref{monotonicity4} by using \eqref{monotonicity1b} and fixing $R_1$ large enough so that $Ce^{-\delta (p-1) R_1} \le \frac18$.

Thus, we conclude gathering \eqref{monotonicity2}--\eqref{monotonicity4} and integrating between $t$ and $t_0$ that
\begin{equation} \label{monotonicity5}
\begin{split}
&\int\eta^2(x,t_0)\psi_L(x_1-y_0)dx+
\frac18\int_{t}^{t_0}\int\big( |\nabla\eta|^2+ \eta^2\big)(x,s)\psi_L'(\tilde{x}_1)dxds \\ &\lesssim e^{-y_0/L} +\int\eta^2(x,t)\psi_L(x_1-\frac12(t_0-t)-y_0)dx,
\end{split}
\end{equation} 
for all $y_0>0$ and $t<t_0$. To handle the second term of the right-hand side of \eqref{monotonicity5}, we use the fact that $\eta$ satisfies condition \eqref{LinearLiouville2}. Given $\epsilon>0$, there exists $A>0$ such that 
\begin{equation} \label{monotonicity6}
\int_{|x_1|>A}\eta^2(x,t)dx \le \int_{|x_1|>A}e^{-\sigma |x_1|}dx_1 \le \epsilon \, .
\end{equation}
On the other hand, it follows from \eqref{monotonicity1b} that
\begin{equation} \label{monotonicity7}
\int_{|x_1|\le A}\eta^2(x,t)\psi_L(\tilde{x}_1)dx \le \psi_L(A-\frac12(t_0-t)-y_0)\int\eta^2(x,t)dx \underset{t \to -\infty}{\longrightarrow} 0 \, .
\end{equation}
Therefore, we conclude the proof of \eqref{monotonicity1} in the case $\alpha=(0,0)$ by using  \eqref{monotonicity6}--\eqref{monotonicity7} and sending $t$ to $-\infty$ in \eqref{monotonicity5}.

Next, we prove \eqref{monotonicity1} in the general case by induction on $k=|\alpha|$. Let $k \in \mathbb N$ be such that $1 \le k \le k_0$. Assume that estimate \eqref{monotonicity1} is true for all $\tilde{\alpha} \in \mathbb N^2$ such that $|\tilde{\alpha}| \le k-1$. Let $\alpha \in \mathbb N^2$ be such that $|\alpha|=k$. Arguing as in \eqref{monotonicity2}, we get that
\begin{equation} \label{monotonicity8}
\begin{split} 
\frac{d}{dt}\int(\partial^{\alpha}\eta)^2\psi_L(\tilde{x}_1)dx &\le -\int \big(|\nabla\partial^{\alpha}\eta|^2+\frac14(\partial^{\alpha}\eta)^2\big)\psi_L'(\tilde{x}_1)dx+ \mathcal{T}_{\alpha}(\eta)\, ,
\end{split}
\end{equation}
where 
\begin{equation} \label{monotonicity9}
\mathcal{T}_{\alpha}(\eta)=-2p\int \partial_{x_1}\partial^{\alpha}(Q^{p-1}\eta)\partial^{\alpha}\eta\psi_L(\tilde{x_1})dx \, .
\end{equation}
By using the Leibniz rule and integrations by parts, we get that 
\begin{equation} \label{monotonicity10}
\begin{split}
\mathcal{T}_{\alpha}(\eta)&=\sum_{0\le\beta\le \alpha}C_{\beta}\int\partial_{x_1}\big(\partial^{\beta}(Q^{p-1})\partial^{\alpha-\beta}\eta\big)\partial^{\alpha}\eta\psi_L(\tilde{x}_1)dx \\
&=\sum_{0 <\beta\le \alpha}C_{\beta}\int\partial_{x_1}\big(\partial^{\beta}(Q^{p-1})\partial^{\alpha-\beta}\eta\big)\partial^{\alpha}\eta\psi_L(\tilde{x}_1)dx 
\\ & \quad  -p\int\partial_{x_1}(Q^{p-1})(\partial^{\alpha}\eta)^2\psi_L(\tilde{x}_1)dx+p\int Q^{p-1}(\partial^{\alpha}\eta)^2\psi_L'(\tilde{x}_1)dx \, .
\end{split}
\end{equation}
On the other hand, observe from \eqref{elliptPDE1} that 
\begin{equation} \label{monotonicity11}
|\partial^{\beta}(Q^{p-1})\psi_L(\tilde{x}_1)| \lesssim e^{-\delta (p-1) |x|}\psi_L(\tilde{x}_1) \lesssim_L \psi'_L(\tilde{x}_1), \quad \text{for all} \ |\beta| \le k+1\, .
\end{equation}
Indeed, it is clear in the case $\tilde{x}_1\le 0$, since $\psi_L(\tilde{x}_1) \lesssim_L\psi_L'(\tilde{x}_1)$. In the case where $\tilde{x}_1>0$, then $0<\tilde{x}_1<x_1$ and $$e^{-\delta (p-1) |x|}\psi_L(\tilde{x}_1) \lesssim e^{-\delta (p-1) x_1} \lesssim e^{-\delta (p-1) \tilde{x}_1} \lesssim \psi_L'(\tilde{x}_1),$$ since $L \ge \frac1{(p-1)\delta}$.
Therefore, we deduce gathering \eqref{monotonicity10}--\eqref{monotonicity11} and using Young's inequality that
\begin{equation} \label{monotonicity12}
\big|\mathcal{T}_{\alpha}(\eta)\big| \lesssim \sum_{0\le \beta \le \alpha}\int (\partial^{\beta}\eta)^2\psi_L'(\tilde{x}_1)dx \, .
\end{equation}

We integrate \eqref{monotonicity8} between $t$ and $t_0$ and use \eqref{monotonicity12}  to obtain that 
\begin{equation} \label{monotonicity13}
\begin{split}
&\int(\partial^{\alpha}\eta)^2(x,t_0)\psi_L(x_1-y_0)dx+\int_t^{t_0}\int  |\nabla\partial^{\alpha}\eta|^2(x,t)\psi_L'(\tilde{x}_1)dxdt \\ &\lesssim \sum_{0\le \beta \le \alpha}\int_t^{t_0}\int (\partial^{\beta}\eta)^2\psi_L'(\tilde{x}_1)dx  +\int (\partial^{\alpha}\eta)^2(x,t)\psi_L(x_1-\frac12(t_0-t)-y_0)dx  .
\end{split}
\end{equation}
Then, we deduce after letting $t \to -\infty$ in \eqref{monotonicity13} and using the induction hypothesis that
\begin{equation} \label{monotonicity14}
\begin{split}
\int(\partial^{\alpha}\eta)^2(x,t_0)&\psi_L(x_1-y_0)dx+\int_{-\infty}^{t_0}\int  |\nabla\partial^{\alpha}\eta|^2(x,t)\psi_L'(\tilde{x}_1)dxdt \\ &\lesssim e^{-y_0/L}+\liminf_{t \to -\infty}\int (\partial^{\alpha}\eta)^2(x,t)\psi_L(x_1-\frac12(t_0-t)-y_0)dx  .
\end{split}
\end{equation}
To handle the second term on the right-hand side of \eqref{monotonicity14}, we use again \eqref{monotonicity1} with $|\tilde{\alpha}|=k-1$ to get 
\begin{displaymath}
 \int_{-\infty}^{t_0}\int (\partial^{\alpha}\eta)^2(x,t)\psi_L'(\tilde{x}_1)dxdt \lesssim e^{-y_0/L} \, ,
\end{displaymath}
so that
\begin{equation} \label{monotonicity15}
 \int_{-\infty}^{t_0}\int_{x_1<(t_0-t)/2+y_0} (\partial^{\alpha}\eta)^2(x,t)e^{(x_1-\frac12(t_0-t))/L}dxdt \lesssim 1 \, ,
\end{equation}
since $\psi_L'(\tilde{x_1}) \gtrsim e^{\tilde{x}_1/L}$ for $\tilde{x}_1<0$. Note that the implicit constant is independent of $y_0>0$. We deduce by passing to the limit as $y_0 \to +\infty$ in \eqref{monotonicity15} and then multiplying by $e^{-y_0/L}$  that 
\begin{equation} \label{monotonicity16}
 \int_{-\infty}^{t_0}\int (\partial^{\alpha}\eta)^2(x,t)\psi_L(x_1-\frac12(t_0-t)-y_0)dxdt \lesssim e^{-y_0/L} \, ,
\end{equation}
since $\psi_L(x_1) \le e^{x_1/L}$ for $x_1 \in \mathbb R$. Therefore, 
\begin{equation} \label{monotonicity17}
\liminf_{t \to -\infty}\int (\partial^{\alpha}\eta)^2(x,t)\psi_L(x_1-\frac12(t_0-t)-y_0)dx =0 \, ,
\end{equation}
which combined with \eqref{monotonicity14} implies \eqref{monotonicity1} in the case $|\alpha|=k$.

This concludes the proof of Lemma \ref{monotonicity}.
\end{proof}

In particular, we deduce from the monotonicity formula that the exponential decay in the $x_1$ direction of the solutions of \eqref{LinearLiouville1} imply the exponential decay of all their derivatives in the direction $x_1$.
\begin{corollary} \label{coro.monotonicity}
Let $\eta \in C(\mathbb R : H^1(\mathbb R^2))$ be a solution of \eqref{LinearLiouville1} satisfying condition \eqref{LinearLiouville2}. Then, there exists $\tilde{\sigma}>0$ such that
\begin{equation} \label{coro.monotonicity.1}
\sup_{t \in \mathbb R}\int (\partial^{\alpha} \eta)^2(x,t)e^{\tilde{\sigma} |x_1|}dx \lesssim 1, \quad \forall \, \alpha \in \mathbb N^2\, .
\end{equation}
\end{corollary}

\begin{proof} Define $\tilde{\sigma}=\frac1{L_0}$ where $L_0$ is given by Lemma \ref{monotonicity}. Since $\psi_L(\tilde{x_1}) \gtrsim e^{\tilde{\sigma}\tilde{x}_1}$ for $\tilde{x}_1<0$, we deduce by using the inequality on the first term in \eqref{monotonicity1} that 
\begin{equation} \label{coro.monotonicity.2}
\int_{x_1<y_0} (\partial^{\alpha} \eta)^2(x,t)e^{\tilde{\sigma} x_1}dx 
\lesssim 1, \quad \forall \, y_0>0, \ t \in \mathbb R \, .
\end{equation}
Thus, it follows sending $y_0$ to $+\infty$ in \eqref{coro.monotonicity.2} that 
\begin{equation} \label{coro.monotonicity.3}
\sup_{t \in \mathbb R}\int (\partial^{\alpha} \eta)^2(x,t)e^{\tilde{\sigma} x_1}dx
\lesssim 1\, .
\end{equation}

To obtain the exponential decay in the direction $x_1<0$, we observe that $\tilde{\eta}(x,t)=\eta(-x,-t)$ is also a solution to \eqref{LinearLiouville1} satisfying \eqref{LinearLiouville2}. Therefore, we deduce arguing as above that 
\begin{equation} \label{coro.monotonicity.4}
\sup_{t \in \mathbb R}\int (\partial^{\alpha} \eta)^2(x,t)e^{-\tilde{\sigma} x_1}dx=\sup_{t \in \mathbb R}\int (\partial^{\alpha} \tilde{\eta})^2(x,t)e^{\tilde{\sigma} x_1}dx
\lesssim 1\, .
\end{equation}

We conclude the proof of \eqref{coro.monotonicity.1} follows gathering \eqref{coro.monotonicity.3}--\eqref{coro.monotonicity.4}.
\end{proof}

\subsection{Proof of Theorem \ref{LinearLiouville}}  Following Martel in \cite{Ma} for the gKdV equation, we will work with a dual problem. Let us define 
\begin{equation} \label{LinearLiouville4}
v=\mathcal{L}\eta-\alpha_0Q, \quad \text{where} \ \alpha_0=\frac{\int \mathcal{L}\eta(\cdot,0)\Lambda Q dx}{\int Q \Lambda Q dx} \, .
\end{equation}
Note that formula \eqref{LambdaQeq2} implies that $\int Q \Lambda Q dx >0,$ since we are in the subcritical case, so that $\alpha_0$ in \eqref{LinearLiouville4} is well-defined.
Then, we deduce from  \eqref{kernel},  \eqref{LinearLiouville1} and the definition of $v$ in \eqref{LinearLiouville4} that $v$ is a solution to
 \begin{equation} \label{LinearLiouville5}
\partial_t v=\mathcal{L}\partial_{x_1}v+\alpha_0\mathcal{L}\partial_{x_1}Q=\mathcal{L}\partial_{x_1}v,
\end{equation}
and $v$ satisfies the orthogonality conditions
\begin{equation} \label{LinearLiouville6}
\int v\partial_{x_1}Q dx=\int v\partial_{x_2}Q dx=0,
\end{equation}
and
\begin{equation} \label{LinearLiouville7}
\int v \Lambda Q dx=0 \, .
\end{equation}
To verify \eqref{LinearLiouville7}, we first observe gathering \eqref{LambdaQeq} and \eqref{LinearLiouville6} that
\begin{displaymath} 
\frac{d}{dt}\int v \Lambda Q dx= \int \mathcal{L}\partial_{x_1}v \Lambda Q dx=-\int v \partial_{x_1}\mathcal{L}\Lambda Q dx= \int v \partial_{x_1}Q dx =0,
\end{displaymath}
so that \eqref{LinearLiouville7} follows from the choice of $\alpha_0$ in \eqref{LinearLiouville4}. 

Now, we infer from the monotonicity property that $v \in C(\mathbb R:H^1(\mathbb R^2))$ satisfies 
\begin{equation} \label{LinearLiouville8}
\int_{x_2}v^2(x_1,x_2,t)dx_2  
\lesssim e^{-\tilde{\sigma} |x_1|} \, , \quad \forall (x_1,t) \in \mathbb R^2 \, .
\end{equation}
for some $\tilde{\sigma}>0$. Indeed, we get from the definitions of $v$ in \eqref{LinearLiouville4}, the decay properties of $Q$ in \eqref{elliptPDE1} and formula \eqref{coro.monotonicity.1} with $|\alpha| \le 3$ that
\begin{equation} \label{LinearLiouville8b}
\sup_{t \in \mathbb R}\int \big(v^2+(\partial_{x_1}v)^2\big)(x,t)e^{\tilde{\sigma} |x_1|}dx \lesssim 1\, ,
\end{equation}
for some positive constant $\tilde{\sigma}$. We compute then by using the Sobolev embedding $H^1(\mathbb R) \hookrightarrow L^{\infty}(\mathbb R)$ and the Cauchy-Schwarz inequality in $x_2$ that
\begin{equation} \label{LinearLiouville8b2}
\begin{split}
\Big\|\Big( \int_{x_2}&v^2(x,t)e^{\tilde{\sigma}|x_1|}dx_2  \Big)^{\frac12}\Big\|_{L^{\infty}_{x_1}} \\& \lesssim 
\Big\| \Big(\int_{x_2}v^2(x,t)e^{\tilde{\sigma}|x_1|}dx_2 \Big)^{\frac12} \Big\|_{H^1_{x_1}} \\ 
& \lesssim \Big(\int v^2(x,t)e^{\tilde{\sigma}|x_1|}dx\Big)^{\frac12}
+\Big(\int_{x_1}\frac{\big(\int_{x_2}vv_{x_1}dx_2\big)^2}{\int_{x_2}v^2dx_2} e^{\tilde{\sigma}|x_1|}dx_1\Big)^{\frac12} 
\\ & \lesssim  \left(\int \big(v^2+(\partial_{x_1}v)^2\big)(x,t)e^{\tilde{\sigma} |x_1|}dx \right)^{\frac12} \, ,
\end{split}
\end{equation}
which together with estimate \eqref{LinearLiouville8b} implies estimate \eqref{LinearLiouville8}.

Next, we derive a virial identity for the solutions of \eqref{LinearLiouville1}. Let $\phi \in C^2(\mathbb R)$ be an even positive function such that $\phi' \le 0$ on $\mathbb R_+$,
\begin{equation} \label{phi}
\phi_{|_{[0,1]}}=1, \quad \phi(x_1)=e^{-x_1} \ \text{on} \, [2,+\infty),  \quad \ e^{-x_1} \le \phi(x_1) \le 3e^{-x_1} \  \text{on} \,  \mathbb R_+
 \, .\end{equation}
 \begin{equation} \label{phi2}
 |\phi'(x_1)| \le C \phi(x_1) \quad \text{and} \quad |\phi''(x_1)|\le C \phi(x_1) \, ,.
 \end{equation}
 for some positive constant $C$.
Let $\varphi$ be be defined by $\varphi(x_1)=\int_0^{x_1}\phi(y)dy$. Then $\varphi$ is an odd function such that $\varphi(x_1)=x_1$ on $[-1,1]$ and $|\varphi(x_1)|\le 3$. For a parameter $A$ (which will be fixed below), we set 
\begin{equation} \label{varphiA}
\varphi_A(x_1)=A\varphi(x_1/A) \quad \text{so that} \quad \varphi_A'(x_1)=\phi(x_1/A)=:\phi_A(x_1)
\, ,\end{equation}
and 
\begin{equation} \label{varphiA2}
\varphi_A(x_1)=x_1 \ \text{on} \, [-A,A],  \ |\varphi_A(x_1)| \le 3A \ \text{and}  \ e^{-|x|/A} \le \phi_A(x) \le 3e^{-|x|/A} \ \text{on} \, \mathbb R
\end{equation}
Then, we have that
\begin{equation} \label{LinearLiouville9}
\begin{split}
-\frac12\frac{d}{dt}\int \varphi_A v^2 dx &=\int \phi_A(\partial_{x_1}v)^2dx+\frac12\int \phi_A\big(|\nabla v|^2+v^2-pQ^{p-1}v^2 \big)dx \\
& \quad -\frac12\int \phi_A''v^2dx-\frac{p}2\int \varphi_A\partial_{x_1}(Q^{p-1})v^2dx \, .
\end{split}
\end{equation}

The following coercivity property will be proved in the next subsection.
\begin{lemma} \label{coercivity}
Consider the bilinear form 
\begin{equation} \label{coercivity.1}
H_{A}(v,w)=\int \phi_A \big(\nabla v \cdot \nabla w+vw-pQ^{p-1}vw \big)dx \, .
\end{equation}
Then, there exists $2<p_2<3$ such that the following holds true for all $2 \le p<p_2$. There exists $\lambda>0$ and $A_0>0$ such that 
\begin{equation} \label{coercivity.2}
H_A(v,v) \ge \lambda \int \phi_A \big(|\nabla v|^2 +v^2 \big)dx \, ,
\end{equation}
for all $v \in H^1(\mathbb R^2)$ satisfying $(v,\Lambda Q)=(v,\partial_{x_1}Q)=(v,\partial_{x_2}Q)=0$ and $A \ge A_0$. (Recall that $(f,g) :=\int fg dx$.)
\end{lemma}
Observe from the choice of $\varphi_A$ (negative for $x_1<0$ and positive for $x_1>0$) that the last on term on the right-hand side of \eqref{LinearLiouville9} is nonnegative.  Moreover, it follows from \eqref{phi2} and \eqref{varphiA} that $|\phi_A''(x_1)| \le C/A^2\phi_A(x_1)$. We fix $A \ge \max\{A_0,2\sqrt{\frac{C}{\lambda}}\}$. Therefore, it follows from \eqref{LinearLiouville9} and \eqref{coercivity.2} that 
\begin{equation} \label{LinearLiouville10}
-\frac12\frac{d}{dt}\int \varphi_A(x_1) v^2 dx \ge \int \phi_A(x_1)\big((\partial_{x_1}v)^2+\frac{\lambda}2|\nabla v|^2+\frac{\lambda}4v^2 \big)dx \, .
\end{equation}
Integrating \eqref{LinearLiouville10}, we deduce that
\begin{equation} \label{LinearLiouville11}
\int_{-\infty}^{+\infty}\int \phi_A(x_1)v^2(x,t)dxdt  
\le \frac{12}{\lambda}A\sup_{t}\|v(\cdot,t)\|_{L^2}^2<+\infty \, ,
\end{equation}
which is finite from \eqref{LinearLiouville8}. Thus, there exists a sequence $\{t_n\}$ satisfying $t_n \to +\infty$ such that 
\begin{equation} \label{LinearLiouville12}
\int \phi_A(x_1)v^2(x,t_n)dx \underset{n \to +\infty}{\longrightarrow} 0 \, .
\end{equation}
By using, the exponential decay of $v$ in the $x_1$ direction, we infer then that 
\begin{equation} \label{LinearLiouville13}
\int v^2(x,t_n)dx \underset{n \to +\infty}{\longrightarrow} 0 \, .
\end{equation}
Indeed, for all $R>0$, there exists $C_R>0$ such that $\phi_A(x_1) \ge C_R$ if $|x_1| \le R$. Then, we deduced from \eqref{LinearLiouville8} that 
\begin{displaymath} 
\begin{split}
\int v^2(x,t_n)dx &=\int_{|x_1|\le R} v^2(x,t_n)dx+\int_{|x_1| > R} v^2(x,t_n)dx\\ 
&\le \frac1{C_R}\int \phi_A(x_1)v^2(x,t_n)dx+\frac2{\tilde{\sigma}}e^{-\tilde{\sigma}R} \, ,
\end{split}
\end{displaymath}
for all $R>0$ which yields \eqref{LinearLiouville13} in view of \eqref{LinearLiouville12}. We show similarly that there exists a sequence $\{s_n\}$ satisfying $s_n \to -\infty$ and
\begin{equation} \label{LinearLiouville14}
\int v^2(x,s_n)dx \underset{n \to +\infty}{\longrightarrow} 0 \, .
\end{equation}

Therefore, we deduce after integrating \eqref{LinearLiouville10} between $s_n$ and $t_n$ and using \eqref{LinearLiouville13} and \eqref{LinearLiouville14} to let $n \to +\infty$ that 
\begin{equation} \label{LinearLiouville15}
\int_{-\infty}^{+\infty}\int \phi_A(x_1)v^2(x,t)dxdt  
\le \frac{6}{\lambda}A\lim_{n \to +\infty}\Big(\int v^2(x,s_n)^2dx+\int v^2(x,t_n)^2dx\Big)=0 \, .
\end{equation}
Since $\phi_A$ is positive function, \eqref{LinearLiouville15} implies that 
\begin{equation} \label{LinearLiouville16}
v(x,t) =0 \quad \text{for all} \ (x,t) \in \mathbb R^3 \, .
\end{equation}

Thus \eqref{kernel}, \eqref{LambdaQeq}, the definition of $v$ in \eqref{LinearLiouville4} and \eqref{LinearLiouville16} implies that there exist $\beta_0 \in \mathbb R$ and two bounded $C^1$ functions $\alpha_1$, $\alpha_2$ such that 
\begin{equation} \label{LinearLiouville17}
\eta(\cdot,t)=\beta_0\Lambda Q+\alpha_1(t)\partial_{x_1}Q+\alpha_2(t)\partial_{x_2}Q \, .
\end{equation}
By using the equation \eqref{LinearLiouville1}, we obtain that 
\begin{equation} \label{LinearLiouville18}
\left\{\begin{array}{l}\alpha_1'(t)=\beta_0 \\\alpha_2'(t)=0
\end{array}\right.
\quad \Rightarrow \quad
\left\{\begin{array}{l}\beta_0=0 \\\alpha_1(t)=a_1\\ \alpha_2(t)=a_2
\end{array}\right. \, ,
\end{equation}
for two real numbers $a_1$, $a_2$. This finishes the proof of Theorem \ref{LinearLiouville}.

\subsection{Coercivity of the bilinear form $H_{A}$}
The aim of this subsection is to prove Lemma \ref{coercivity}. We first prove a similar result for the non-localized quadratic form. 
\begin{proposition} \label{prop.coercivity}
Consider the bilinear form 
\begin{equation} \label{prop.coercivity.1}
H(v,w)=(\mathcal{L}v,w)=\int \big(\nabla v \cdot \nabla w+vw-2Qvw \big)dx \, .
\end{equation}
Then, there exists $\tilde{\lambda}>0$ such that 
\begin{equation} \label{prop.coercivity.2}
H(v,v) \ge \tilde{\lambda} \|v\|_{H^1}^2 \, ,
\end{equation}
for all $v \in H^1(\mathbb R^2)$ satisfying $(v,\Lambda Q)=(v,\partial_{x_1}Q)=(v,\partial_{x_2}Q)=0$. 
\end{proposition}

The proof of Proposition \ref{prop.coercivity} relies on the following spectral property.

\begin{proposition} \label{SpectralProperty}
Assume that $d=2$. There exists $2<p_2<3$ such that   
\begin{equation} \label{SpectralProperty1}
\big(\mathcal{L}^{-1}\Lambda Q, \Lambda Q \big) <0 \, ,
\end{equation}
for all $2 \le p <p_2$.
\end{proposition}
The proof of Proposition \ref{SpectralProperty} is given in the appendix by using numerical methods. 

\begin{proof}[Proof of Proposition \ref{prop.coercivity}]
 From  \eqref{LambdaQeq}, we have
\begin{displaymath} 
( \Lambda Q, \chi_0 )=-\frac1{\lambda_0}(\Lambda Q,\mathcal{L}\chi_0)=\frac1{\lambda_0} \int Q \chi_0 dx >0 \quad \text{and} \quad \Lambda Q \in \big( \ker \mathcal{L} \big)^{\perp} \, .
\end{displaymath} 
Therefore, we conclude the proof of Proposition \ref{prop.coercivity} by invoking Lemma E.1 (and the proof of Proposition 2.9) in \cite{We}, since $\big(\mathcal{L}^{-1}\Lambda Q, \Lambda Q \big) <0$ in our case due to Proposition \ref{SpectralProperty}.
\end{proof}

To deduce Lemma \ref{coercivity} from Proposition \ref{prop.coercivity}, we follow the ideas in the appendices of \cite{MM3,Co} and first prove a technical lemma. 
\begin{lemma} \label{lemma.coercivity}
There exists $\kappa>0$ (depending on $\tilde{\lambda}$ given by Proposition \ref{prop.coercivity}) such that 
\begin{equation} \label{lemma.coercivity.1}
H(v,v) =\int \big(|\nabla v|^2+v^2-pQ^{p-1}v^2 \big)dx \ge \frac{\tilde{\lambda}}2 \|v\|_{H^1}^2 \, ,
\end{equation}
for all $v \in H^1(\mathbb R^2)$ satisfying 
\begin{equation} \label{lemma.coercivity.2}
\big|\big(v,\frac{\Lambda Q}{\|\Lambda Q\|_{L^2}}\big)\big|+\big|\big(v,\frac{\partial_{x_1}Q}{\|\partial_{x_1}Q\|_{L^2}}\big)\big|+\big|\big(v,\frac{\partial_{x_2}Q}{\|\partial_{x_2}Q\|_{L^2}}\big)\big| \le \kappa \|v\|_{H^1} \, .
\end{equation}
\end{lemma}

\begin{proof} Let $v$ in $H^1(\mathbb R^2)$ satisfying \eqref{lemma.coercivity.2}. We use the decomposition 
\begin{equation} \label{lemma.coercivity.3}
v=v_1+b_0\frac{\Lambda Q}{\|\Lambda Q\|_{L^2}}+b_1\frac{\partial_{x_1}Q}{\|\partial_{x_1}Q\|_{L^2}}+b_2\frac{\partial_{x_2}Q}{\|\partial_{x_2}Q\|_{L^2}} =v_1+v_2\, ,
\end{equation}
with $(v_1,\Lambda Q)=(v_1,\partial_{x_1}Q)=(v_1,\partial_{x_2}Q)=0$, so that \eqref{lemma.coercivity.2} yields 
\begin{equation} \label{lemma.coercivity.4}
|b_0|+|b_1|+|b_2| \le \kappa \|v\|_{H^1} \, .
\end{equation}
Moreover, if $0<\kappa \le \frac12$, \eqref{lemma.coercivity.3} and \eqref{lemma.coercivity.4} imply that 
\begin{equation} \label{lemma.coercivity.5}
\frac{\sqrt{3}}2\|v\|_{H^1} \le \|v_1\|_{H^1} \le \|v\|_{H^1} \, .
\end{equation}

Now, we compute 
\begin{equation} \label{lemma.coercivity.6}
H(v,v)=H(v_1,v_1)+H(v_2,v_2)+2H(v_1,v_2) \, ,
\end{equation} 
On the one hand, it follows by \eqref{prop.coercivity.2} and \eqref{lemma.coercivity.5} that
\begin{equation} \label{lemma.coercivity.7}
H(v_1,v_1)\ge \tilde{\lambda}\|v_1\|_{H^1} \ge \frac{3\tilde{\lambda}}4\|v\|_{H^1}^2 \, .
\end{equation}
On the other hand, the continuity of $H$ and \eqref{lemma.coercivity.4} give that 
\begin{equation} \label{lemma.coercivity.8}
H(v_2,v_2) \lesssim |b_0|^2+|b_1|^2+|b_2|^2 \lesssim \kappa^2\|v\|_{H^1}^2 \le \frac{\tilde{\lambda}}8 \|v\|_{H^1}^2
\end{equation}
and 
\begin{equation} \label{lemma.coercivity.9}
H(v_1,v_2) \lesssim \|v_1\|_{H^1}\|v_2\|_{H^1} \lesssim \|v\|_{H^1}\big( |b_0|+|b_1|+|b_2|\big)
\lesssim \kappa \|v\|_{H^1}^2 \le \frac{\tilde{\lambda}}8 \|v\|_{H^1}^2 \, ,
\end{equation}
as soon as $\kappa$ is chosen small enough (as a function of $\tilde{\lambda}$).

The proof of Lemma \ref{lemma.coercivity} is concluded gathering \eqref{lemma.coercivity.6}--\eqref{lemma.coercivity.9}. 
\end{proof}

\begin{proof}[Proof of Lemma \ref{coercivity}]
Let $v \in H^1(\mathbb R^2)$ be such that $(v,\Lambda Q)=(v,\partial_{x_1}Q)=(v,\partial_{x_2}Q)=0$. Recall from \eqref{phi} that $\phi_A$ is a positive function. Then, a direct computation gives 
\begin{equation} \label{coercivity.3}
H_{A}(v,v)=H(\sqrt{\phi_A}v,\sqrt{\phi_A}v)-\int (\partial_{x_1}\sqrt{\phi_A})^2v^2dx
-\int \phi_A'v\partial_{x_1}vdx \, .
\end{equation}
Thanks to the orthogonality properties on $v$, the definition of $\phi_A$ in \eqref{phi} and \eqref{varphiA} and the decay property of $Q$ and its derivatives \eqref{elliptPDE1}
\begin{equation} \label{coercivity.4}
\Big| \int \sqrt{\phi_A}v \frac{\partial_{x_1}Q}{\|\partial_{x_1}Q\|_{L^2}}dx \Big|=
\Big| \int (1-\sqrt{\phi_A})v \frac{\partial_{x_1}Q}{\|\partial_{x_1}Q\|_{L^2}}dx \Big| \le \kappa \|\sqrt{\phi_A}v\|_{L^2} \, ,
\end{equation}
if $A$ is chosen large enough. Arguing similarly, we have that 
\begin{equation} \label{coercivity.5}
\Big| \int \sqrt{\phi_A}v \frac{\partial_{x_2}Q}{\|\partial_{x_1}Q\|_{L^2}}dx \Big| \le \kappa \|\sqrt{\phi_A}v\|_{L^2}
\end{equation}
and 
\begin{equation} \label{coercivity.6}
\Big| \int \sqrt{\phi_A}v \frac{\Lambda Q}{\|\Lambda Q\|_{L^2}}dx \Big| \le \kappa \|\sqrt{\phi_A}v\|_{L^2} \, ,
\end{equation}
for $A$ chosen large enough. Then, it follows from Lemma \ref{lemma.coercivity} that
\begin{equation} \label{coercivity.7}
H(\sqrt{\phi_A}v,\sqrt{\phi_A}v) \ge  \|\sqrt{\phi_A}v\|_{H^1}^2 \, .
\end{equation}
We deduce gathering \eqref{coercivity.3} and \eqref{coercivity.7} that
\begin{equation} \label{coercivity.8}
H_{A}(v,v)\ge  \frac{\tilde{\lambda}}2 \int \phi_A \big( v^2+|\nabla v|^2\big)-\int (\partial_{x_1}\sqrt{\phi_A})^2v^2dx
+(\frac{\tilde{\lambda}}2-1)\int \phi_A'v\partial_{x_1}vdx  .
\end{equation}
We use \eqref{phi2} and \eqref{varphiA}  to control the last two terms on the right-hand side of \eqref{coercivity.8}. It follows that 
\begin{equation} \label{coercivity.9}
\int (\partial_{x_1}\sqrt{\phi_A})^2v^2dx \le \frac{C}{A^2}\int \phi_A v^2 dx \le \frac{\tilde{\lambda}}8 \int \phi_A v^2 dx 
\end{equation}
and 
\begin{equation} \label{coercivity.10}
\int \phi_A'v\partial_{x_1}vdx \le \frac{C}{2A^2}\int \phi_A\big(v^2+(\partial_{x_1}v)^2\big) dx \le \frac{\tilde{\lambda}}8 \int \phi_A \big(v^2+(\partial_{x_1}v)^2\big)  dx \, ,
\end{equation}
if $A$ is chosen large enough. 

Therefore, we conclude the proof of \eqref{coercivity.2} gathering \eqref{coercivity.8}--\eqref{coercivity.10}. Observe that we can choose $\lambda=\frac{\tilde{\lambda}}4$, where $\tilde{\lambda}$ is given by  Proposition \ref{prop.coercivity}.

\end{proof}

\section{Nonlinear Liouville Property } \label{NonlinearRigidity}
In this section, we give the proof of Theorem \ref{NonLinearLiouville}. According to Remark \ref{remarkscaling}, we will assume in this section that $c_0=1$.

\subsection{Modulation of a solution close to the soliton $Q$}

\begin{lemma} \label{modulation}
There exist $\epsilon_0>0$,  $\delta_0>0$ and $K_0>0$ such that for any $0<\epsilon \le \epsilon_0$ the following is true. For any solution $u \in C(\mathbb R:H^1(\mathbb R^2))$ of \eqref{ZK} satisfying
\begin{equation} \label{modulation.1}
\inf_{\tau \in \mathbb R^2} \|u(\cdot,t)-Q(\cdot-\tau)\|_{H^1} \le \epsilon \quad \forall \, t \in \mathbb R \, ,
\end{equation}
there exist $\rho=\big(\rho_1,\rho_2 \big) \in C^1(\mathbb R : \mathbb R^2)$ and $c \in C^1(\mathbb R : \mathbb R)$ such that 
\begin{equation} \label{modulation.2}
\eta(x,t)=u(x+\rho(t),t)-Q_{c(t)}(x)
\end{equation}
satisfies for all $t \in \mathbb R$
\begin{equation} \label{modulation.3}
|c(t)-1|+\|\eta(\cdot,t)\|_{H^1} \le K_0 \epsilon \, ,
\end{equation}

\begin{equation} \label{modulation.4}
\int \eta(x,t)\partial_{x_1}Q_{c(t)}(x)dx=\int \eta(x,t)\partial_{x_2}Q_{c(t)}(x)dx
=\int \eta(x,t) Q_{c(t)}(x)dx=0
\end{equation}
and
\begin{equation} \label{modulation.5}
|c'(t)|^{\frac12}+|\rho_1'(t)-c(t)|+|\rho_2'(t)| \le K_0\Big(\int \eta(x,t)^2e^{-\delta |x|} dx\Big)^{\frac12} \, .
\end{equation}
Moreover the functions $\rho$ and $c$ satisfying \eqref{modulation.2}--\eqref{modulation.5} are unique.
\end{lemma}

\begin{proof} The proof of Lemma \ref{modulation} is a classical application of the implicit function theorem (see for example Proposition 1 in \cite{MM4} or page 225 in \cite{MM1}). Note that the non-degeneracy conditions to satisfy the orthogonality conditions \eqref{modulation.4} are given by 
\begin{equation} \label{nondegenerancy}
\int (\partial_{x_1}Q)^2dx>0, \quad \int (\partial_{x_2}Q)^2dx>0 \quad \text{and} \ \int Q\Lambda Q dx>0\, .
\end{equation}
The last condition in \eqref{nondegenerancy} is satisfied since we are in the subcritical case (see formula \eqref{LambdaQeq2}). 

For the sake of completeness, we explain how to deduce \eqref{modulation.5} from \eqref{modulation.2}--\eqref{modulation.4}. In particular, the fact that $|c'(t)|$ is bounded by a quadratic function of $\|\eta\|_{L^2}$ is a consequence of the orthogonality of $\eta$ and $Q_{c}$ in \eqref{modulation.4}  and will be of crucial importance in the proof of Theorem \ref{NonLinearLiouville}.

First, we derive the equation on $\eta$. Since $u$ is a solution to \eqref{ZK}, we compute that 
\begin{equation} \label{equationeta}
\begin{split} 
\partial_t\eta&=-\partial_{x_1}(\Delta u+u^2)+\rho'\cdot\nabla u-c'\Lambda Q_{c} \\
&=\partial_{x_1}\big(\mathcal{L}_c\eta-\eta^2 \big)+(\rho_1'-c)\partial_{x_1}(Q_c+\eta)+\rho_2'\partial_{x_2}(Q_c+\eta)-c'\Lambda Q_c \, ,
\end{split}
\end{equation}
where $\mathcal{L}_c$ was defined in \eqref{Lc}. Thus, we obtain by deriving the last orthogonality condition in \eqref{modulation.4} with respect to the time that 
\begin{equation} \label{modulation.6}
\begin{split} 
\int \partial_{x_1}\mathcal{L}_c&\eta Q_cdx-\int\partial_{x_1}(\eta^2)Q_cdx+(\rho_1'-c)\int \partial_{x_1}(Q_c+\eta)Q_cdx
\\&+\rho_2'\int\partial_{x_2}(Q_c+\eta)Q_cdx-c'\int Q_c\Lambda Q_cdx+c'\int \eta \Lambda Q_cdx=0 \, .
\end{split}
\end{equation}
Now, observe  that 
\begin{equation} \label{modulation.7}
\int \partial_{x_1}\mathcal{L}_c\eta Q_cdx=-\int \eta  \mathcal{L}_c\partial_{x_1}Q_cdx=0 \, ,
\end{equation}
since $\partial_{x_1}Q_c$ belongs to the kernel of $\mathcal{L}_c$. Moreover, the orthogonality conditions in \eqref{modulation.4} yield
\begin{equation} \label{modulation.8}
\int \partial_{x_1}(Q_c+\eta)Q_cdx=\int\partial_{x_2}(Q_c+\eta)Q_cdx=0 \, .
\end{equation}
Thus, we deduce from \eqref{elliptPDE1}, \eqref{modulation.3} and\eqref{modulation.6}--\eqref{modulation.8} that
\begin{equation} \label{modulation.9}
|c'|=\frac{\big|\int \eta^2\partial_{x_1}Q_cdx \big|}{\big|\int Q_c\Lambda Q_c dx-\int \eta \Lambda Q_c dx\big|} \le \frac{\int\eta^2e^{-\delta |x|}dx}{\int Q_c\Lambda Q_c dx-K_0\epsilon_0}
\, .
\end{equation}
Recall from formula \eqref{LambdaQeq2} that $\int Q_c\Lambda Q_cdx>0$ since we are in the subcritical case. Therefore, \eqref{modulation.9} yields the first inequality in \eqref{modulation.5} if $\epsilon_0$ and $K_0$ are chosen correctly.
\end{proof}

\begin{lemma} \label{modulationdecay}
Under the assumptions of Lemma \ref{modulation}. Assume moreover that there exist $\sigma>0$ and some function $\tilde{\rho} \in C(\mathbb R : \mathbb R^2)$
\begin{equation} \label{modulationdecay.1}
 \int_{x_2}u^2(x+\tilde{\rho}(t),t)dx_2  \lesssim e^{-\sigma |x_1|}\, , \quad \forall \, (x_1,t) \in \mathbb R^2 \, .
\end{equation}
Then, 
\begin{equation} \label{modulationdecay.2}
 \int_{x_2}u^2(x+\rho(t),t)dx_2  \lesssim e^{-\sigma |x_1|}\, , \quad \forall \, (x_1,t) \in \mathbb R^2 \, ,
\end{equation}
where $\rho$ is the function obtained from the modulation theory in Lemma \ref{modulation}.
\end{lemma}

\begin{proof} First, we infer that there exists $A>0$ such that 
\begin{equation} \label{modulationdecay.3}
\big|\rho_1(t)-\tilde{\rho}_1(t) \big| \le A, \quad \forall \, t \in \mathbb R \, .
\end{equation}
Indeed, on the one hand we get from the triangle inequality that 
\begin{displaymath} 
\|u(\cdot+\rho(t),t)\|_{L^2(|x|\le 1)} \ge \|Q\|_{L^2(|x|\le 1)}-\|Q-Q_{c(t)}\|_{L^2}-\|\eta(\cdot,t)\|_{L^2} \, .
\end{displaymath}
Then, since the function $:t \in \mathbb R \mapsto Q_{c(t)} \in H^1(\mathbb R^2)$ is continuous, we conclude from \eqref{modulation.3} that 
\begin{equation} \label{modulationdecay.4}
\|u(\cdot+\rho(t),t)\|_{L^2(|x|\le 1)} \ge \frac12\|Q\|_{L^2(|x|\le 1)}, \quad \forall \, t \in \mathbb R \, ,
\end{equation}
if $\epsilon_0$ is chosen small enough. On the other hand, we deduce from \eqref{modulationdecay.1} that there exists $\tilde{A}>0$ such that
\begin{equation} \label{modulationdecay.5}
\Big(\int_{|x_1|\ge \tilde{A}}u^2(x+\tilde{\rho}(t),t)dx\Big)^{\frac12}\lesssim 
\Big(\int_{|x_1|\ge \tilde{A}}e^{-\sigma |x_1|}dx_1\Big)^{\frac12} \le \frac14\|Q\|_{L^2(|x|\le 1)} \, .
\end{equation}
Let us define $A=\tilde{A}+1$. Assume by contradiction that $|\tilde{\rho}(t)-\rho(t)| \ge A$ for some $t \in \mathbb R$. Then, \eqref{modulationdecay.5} implies that 
\begin{displaymath} 
\|u(\cdot+\rho(t),t)\|_{L^2(|x|\le 1)}\le
\Big(\int_{|x_1+\rho_1(t)-\tilde{\rho}_1(t)| \le 1} u^2(x+\tilde{\rho}(t),t) dx\Big)^{\frac12} \le \frac14\|Q\|_{L^2(|x|\le 1)} \, .
\end{displaymath}
This contradicts \eqref{modulationdecay.4} and thus proves \eqref{modulationdecay.3}.

Finally, we get from \eqref{modulationdecay.1}  that
\begin{displaymath}
\begin{split}
 \int_{x_2}u^2(x+\rho(t),t)dx_2&= \int_{x_2}u^2(x+\tilde{\rho}(t)+(\rho(t)-\tilde{\rho}(t)),t)dx_2 
 \lesssim e^{-\sigma|x_1+\rho_1(t)-\tilde{\rho}_1(t)|} \, ,
\end{split}
\end{displaymath}
which together with \eqref{modulationdecay.3} concludes the proof of estimate \eqref{modulationdecay.2}.

\end{proof}

\subsection{Monotonicity} In this subsection, we prove monotonicity properties first for the $L^2$-norm of $u$, then for an energy in $H^1$ associated to $u$ and finally for the $L^2$-norm of $\partial^{\alpha}u$, for any $\alpha \in \mathbb N^2$, by induction on $|\alpha|=k$. 

Let $u$ be a solution to \eqref{ZK} satisfying \eqref{NonLinearLiouville1}. Then, by using the decomposition in Lemma \ref{modulation}, there exists $\rho=(\rho_1,\rho_2) \in C(\mathbb R : \mathbb R^2)$ and $c \in C^1(\mathbb R : \mathbb R)$ satisfying \eqref{modulation.2}-\eqref{modulation.5}. 

Let us define $\psi_M$ be defined as in \eqref{psi0}--\eqref{psi}. For $y_0>0$, $t_0 \in \mathbb R$ and $t \le t_0$, we also define 
\begin{equation} \label{tildex1}
\tilde{x}_1=x_1-\rho_1(t_0)+\frac12(t_0-t)-y_0 \, .
\end{equation}

We first derive the $L^2$-monotonicity property.
\begin{lemma} \label{nlL2monotonicity}
Assume that $u \in C(\mathbb R : H^1(\mathbb R^2))$ is a solution of \eqref{ZK} satisfying \eqref{modulation.2}-\eqref{modulation.5}. For  $y_0>0$, $t_0 \in \mathbb R$ and $t \le t_0$, let us define 
\begin{equation}\label{nlL2monotonicity.1}
I_{y_0,t_0}(t)=\int u^2(x,t)\psi_M(\tilde{x}_1)dx \, ,
\end{equation}
where $\psi_M$ is defined as in \eqref{psi0}--\eqref{psi} and $\tilde{x}_1$ is defined as in \eqref{tildex1}. Then
\begin{equation} \label{nlL2monotonicity.2}
I_{y_0,t_0}(t_0)-I_{y_0,t_0}(t) \lesssim e^{-y_0/M} \, ,
\end{equation}
 if $\epsilon_0$ in \eqref{NonLinearLiouville1} is chosen small enough and $M \ge 4$.

If moreover, $u$ satisfies the decay assumption \eqref{NonLinearLiouville2}, then 
\begin{equation} \label{nlL2monotonicity.3}
\begin{split}
\int &u^2(x,t_0)\psi_M(x_1-\rho_1(t_0)-y_0)dx \\ &+\int_{-\infty}^{t_0}\int\big(|\nabla u|^2+u^2 \big)(x,t)\psi_M'(\tilde{x}_1)dxdt \lesssim e^{-y_0/M} \, .
\end{split}
\end{equation}
\end{lemma} 

\begin{remark} \label{remarkL2monotonicity}
It will be clear from the proof that \eqref{nlL2monotonicity.2} still holds true if, for any $0<\beta<1$, we redefine $I_{y_0,t_0}(t)$ by 
$$I_{y_0,t_0}(t)=\int u^2(x,t)\psi_M(x_1-\rho_1(t_0)+\beta(t_0-t)-y_0)dx \, ,$$
and choose $M=M(\beta)>0$ big enough.
\end{remark}

\begin{proof} Fix $M \ge 4$. We compute by using \eqref{ZK} and the inequality in \eqref{psi} that
\begin{equation} \label{nlL2monotonicity.4}
\begin{split}
\frac{d}{dt}I_{y_0,t_0}(t) &=2\int u\partial_tu\psi_M(\tilde{x}_1)dx-\frac12\int u^2 \psi_M'(\tilde{x}_1)dx
\\ & \le -\int \big(3(\partial_{x_1}u)^2+(\partial_{x_2}u)^2+\frac14u^2\big)\psi_M'(\tilde{x}_1)dx+\frac23\int u^3\psi_M'(\tilde{x}_1)dx \, .
\end{split}
\end{equation}
We decompose the nonlinear term on the right-hand side of \eqref{nlL2monotonicity.4} as follows 
\begin{equation} \label{nlL2monotonicity.5}
\int u^3\psi_M'(\tilde{x}_1)dx=\int Q_c(\cdot-\rho)u^2\psi_M'(\tilde{x}_1)dx+\int \big(u-Q_c(\cdot-\rho)\big)u^2\psi_M'(\tilde{x}_1)dx \, .
\end{equation}
To deal with the second term on the right-hand side of \eqref{nlL2monotonicity.5}, we use the Sobolev embedding $H^1(\mathbb R^2) \hookrightarrow L^3(\mathbb R^2)$. Then
\begin{equation} \label{nlL2monotonicity.6}
\begin{split}
\Big|\int \big(u-Q_{c}(\cdot-\rho)\big)u^2\psi_M'(\tilde{x}_1)dx\Big| &\lesssim \|u-Q_{c}(\cdot-\rho)\|_{H^1}\|\sqrt{\psi_M'}u\|_{H^1}^2 
\\ & \lesssim K_0\epsilon_0\int \big(|\nabla u|^2+u^2 \big)\psi_M'dx \, ,
\end{split}
\end{equation}
in view of \eqref{modulation.3} and \eqref{psi}. 

Now, we treat the first term on the right-hand side of \eqref{nlL2monotonicity.5}. Let $R_1$ be a positive number to be chosen later. In the case where $|x-\rho(t)| \ge R_1$, we have that 
\begin{equation} \label{nlL2monotonicity.7}
\Big| \int_{|x-\rho(t)| \ge R_1} Q_c(\cdot-\rho)u^2\psi_M'(\tilde{x}_1)dx\Big| \le Ce^{-\delta R_1}\int u^2\psi_M'dx \, ,
\end{equation}
where $\delta$ is the positive number given in \eqref{elliptPDE1}. 

In the case where $|x-\rho(t)| \le R_1$, we observe by using \eqref{modulation.3}, \eqref{modulation.5} and the mean value theorem that
\begin{displaymath}
|\tilde{x}_1| \ge |\rho_1(t_0)-\rho_1(t)+y_0|-\frac12(t_0-t)-|x_1-\rho_1(t)|
\ge \frac14(t_0-t)+y_0-R_1\, ,
\end{displaymath}
if $\epsilon_0$ is chosen small enough. Thus 
\begin{equation} \label{nlL2monotonicity.8}
\Big| \int_{|x-\rho(t)| \le R_1} Q_c(\cdot-\rho)u^2\psi_M'(\tilde{x}_1)dx\Big| \le e^{R_1/M}e^{-(\frac1{4}(t_0-t)+y_0)/M}\int u_0^2dx \, ,
\end{equation}
since $\psi_M'(\tilde{x}_1) \le e^{-|\tilde{x}_1|/M}$ and the $L^2$-norm of $u$ is conserved. 

We deduce gathering \eqref{nlL2monotonicity.5}--\eqref{nlL2monotonicity.8}, fixing the value of $R_1$ and choosing $\epsilon_0$ small enough that
\begin{equation} \label{nlL2monotonicity.9}
\frac23\Big|\int u^3\psi_M'(\tilde{x}_1)dx\Big| \le \frac18\int \big(|\nabla u|^2+u^2\big)\psi_M'(\tilde{x}_1)dx+Ce^{-(\frac1{4}(t_0-t)+y_0)/M} \, .
\end{equation}
Therefore, it follows integrating \eqref{nlL2monotonicity.4} between $t$ and $t_0$ and using 
\eqref{nlL2monotonicity.9} that 
\begin{equation} \label{nlL2monotonicity.10}
I_{y_0,t_0}(t_0)-I_{y_0,t_0}(t)+\frac18\int_t^{t_0}\int\big(|\nabla u|^2+u^2\big)(x,s)\psi_M'(\tilde{x}_1)dxds \lesssim e^{-y_0/M} \, ,
\end{equation}
which in particular implies estimate \eqref{nlL2monotonicity.2}.

Now, we assume that $u$ satisfies the decay assumption \eqref{NonLinearLiouville2}.  Then,  arguing as in \eqref{monotonicity6}--\eqref{monotonicity7}, we get that 
\begin{displaymath}
\lim_{t \to -\infty}\int u^2(x,t)\psi_M(\tilde{x}_1)dx = 0 \, .
\end{displaymath}
Therefore, we prove estimate \eqref{nlL2monotonicity.3} by passing to the limit as $t \to -\infty$ in \eqref{nlL2monotonicity.10}. This concludes the proof of Lemma \ref{nlL2monotonicity}.
\end{proof}

Next, we derive a monotonicity property for the energy. 
\begin{lemma} \label{nlEnergymonotonicity}
Assume that $u \in C(\mathbb R : H^1(\mathbb R^2))$ is a solution of \eqref{ZK} satisfying \eqref{modulation.2}-\eqref{modulation.5}. For  $y_0>0$, $t_0 \in \mathbb R$ and $t \le t_0$, let us define 
\begin{equation}\label{nlEnergymonotonicity.1}
J_{y_0,t_0}(t)=\int \big(|\nabla u|^2-\frac23u^3\big)(x,t)\psi_M(\tilde{x}_1)dx \, ,
\end{equation}
where $\psi_M$ is defined as in \eqref{psi0}--\eqref{psi} and $\tilde{x}_1$ is defined as in \eqref{tildex1}. Then
\begin{equation} \label{nlEnergymonotonicity.2}
J_{y_0,t_0}(t_0)-J_{y_0,t_0}(t) \lesssim e^{-y_0/M} \, ,
\end{equation}
if $\epsilon_0$ in \eqref{NonLinearLiouville1} is chosen small enough and $M \ge 4$.

If moreover $u$ satisfies the decay assumption \eqref{NonLinearLiouville2}, then 
\begin{equation} \label{nlEnergymonotonicity.3}
\begin{split}
\int &|\nabla u|^2(x,t_0)\psi_M(x_1-\rho_1(t_0)-y_0)dx \\ & 
+\int_{-\infty}^{t_0}\int\big(|\nabla^2 u|^2+|\nabla u|^2+u^4 \big)(x,t)\psi_M'(\tilde{x}_1) dxdt\lesssim e^{-y_0/M} \, .
\end{split}
\end{equation}
\end{lemma}

\begin{remark} \label{remarkEnergymonotonicity}
It will be clear from the proof that \eqref{nlEnergymonotonicity.2} still holds true if, for any $0<\beta<1$, we redefine $J_{y_0,t_0}(t)$ by 
$$J_{y_0,t_0}(t)=\int \big(|\nabla u|^2-\frac23u^3\big)(x,t)\psi_M(x_1-\rho_1(t_0)+\beta(t_0-t)-y_0)dx \, ,$$
and choose $M=M(\beta)>0$ big enough.
\end{remark}

\begin{proof} Straightforward computations using \eqref{ZK} and \eqref{psi} show that
\begin{equation} \label{nlEnergymonotonicity.4}
\begin{split}
\frac{d}{dt}\int |\nabla u|^2&\psi_M(\tilde{x}_1)dx\\ & \le -\int \big(|\nabla^2u|^2+\frac14|\nabla u|^2\big)\psi_M'(\tilde{x}_1)dx+2\int u|\nabla u|^2\psi_M'(\tilde{x}_1)dx  \\ &\quad
-2\int \big((\partial_{x_1}u)^3+\partial_{x_1}u(\partial_{x_2}u)^2\big)\psi_M(\tilde{x}_1)dx \, ,
\end{split}
\end{equation}
and 
\begin{equation} \label{nlEnergymonotonicity.5}
\begin{split}
&-\frac23\frac{d}{dt}\int u^3\psi_M(\tilde{x}_1)dx \\ & =   -\int u^4\psi_M'(\tilde{x}_1)dx+
\int \big(6u(\partial_{x_1} u)^2+2u(\partial_{x_2} u)^2\big)\psi_M'(\tilde{x}_1)dx
\\ &\quad  +\frac13\int u^3\big(\psi_M'(\tilde{x}_1)-2\psi_M'''(\tilde{x}_1) \big)dx+2\int \big((\partial_{x_1}u)^3+\partial_{x_1}u(\partial_{x_2}u)^2\big)\psi_M(\tilde{x}_1)dx \, .
\end{split}
\end{equation}
Observe that the last terms on the right-hand side of \eqref{nlEnergymonotonicity.4} and \eqref{nlEnergymonotonicity.5} cancel out. Therefore, it follows by adding \eqref{nlEnergymonotonicity.4} and \eqref{nlEnergymonotonicity.5} and using \eqref{psi} again that
\begin{equation} \label{nlEnergymonotonicity.6}
\begin{split}
\frac{d}{dt}J_{y_0,t_0}(t) &\le -\int \big(|\nabla^2u|^2+\frac14|\nabla u|^2+u^4\big)\psi_M'(\tilde{x}_1)dx+C\int u^3\psi_M'(\tilde{x}_1)dx \\ &\quad+4\int \big(2u(\partial_{x_1} u)^2+u(\partial_{x_2} u)^2\big)\psi_M'(\tilde{x}_1)dx \, .
\end{split}
\end{equation}
We deduce arguing exactly as in \eqref{nlL2monotonicity.5}--\eqref{nlL2monotonicity.9} that the last two terms on the right-hand side of \eqref{nlEnergymonotonicity.6} are bounded by
\begin{displaymath}
\frac18\int \big(|\nabla^2 u|^2+|\nabla u|^2+u^2\big)\psi_M'(\tilde{x}_1)dx+Ce^{-(\frac1{4}(t_0-t)+y_0)/M} \, ,
\end{displaymath}
if $\epsilon_0$ is chosen small enough. Thus after integrating \eqref{nlEnergymonotonicity.6} between $t$ and $t_0$, we get from \eqref{nlL2monotonicity.10} that
\begin{equation} \label{nlEnergymonotonicity.7}
J_{y_0,t_0}(t_0)-J_{y_0,t_0}(t)+\frac18\int_t^{t_0}\int\big(|\nabla^2 u|^2+|\nabla u|^2+u^4\big)(x,s)\psi_M'(\tilde{x}_1)dxds \lesssim e^{-y_0/M} \, ,
\end{equation}
which in particular implies estimate \eqref{nlEnergymonotonicity.2}. 

Now, we assume moreover that $u$ satisfies the decay assumption \eqref{NonLinearLiouville2}. On the one hand, by using the Sobolev embedding $H^1(\mathbb R^2) \hookrightarrow L^3(\mathbb R^2)$ and the fact that $u$ is bounded in $H^1$ we get that 
\begin{equation} \label{nlEnergymonotonicity.8}
J_{y_0,t_0}(t) \le \int \big(u^2+|\nabla u|^2\big)(x,t)\psi_M(\tilde{x}_1)dx\, .
\end{equation}
On the other hand, we deduce by using the second inequality in \eqref{nlL2monotonicity.3} and arguing exactly as in \eqref{monotonicity15}--\eqref{monotonicity17} that
\begin{displaymath}
\liminf_{t \to -\infty}\int \big(u^2+|\nabla u|^2\big)(x,t)\psi_M(\tilde{x}_1)dx=0 \, .
\end{displaymath}
It follows then by letting $t \to -\infty$ in \eqref{nlEnergymonotonicity.7} that 
\begin{equation} \label{nlEnergymonotonicity.9}
J_{y_0,t_0}(t_0)+\frac18\int_{-\infty}^{t_0}\int\big(|\nabla^2 u|^2+|\nabla u|^2+u^4\big)(x,s)\psi_M'(\tilde{x}_1)dxds \lesssim e^{-y_0/M} \, .
\end{equation}
Next, observe that 
\begin{displaymath}
\begin{split}
\int |\nabla u|^2&(x,t_0)\psi_M(x_1-\rho_1(t_0)-y_0)dx  \\ &\le J_{y_0,t_0}(t_0)+\frac23\int u^3(x,t_0)\psi_M(x_1-\rho_1(t_0)-y_0)dx  \, .
\end{split}
\end{displaymath}
 Thus, we use the decomposition in \eqref{nlL2monotonicity.5}, the Sobolev embedding $H^1(\mathbb R^2) \hookrightarrow L^3(\mathbb R^2)$, \eqref{modulation.3} and the first inequality in \eqref{nlL2monotonicity.3} to get that
 \begin{displaymath}
 \int |\nabla u|^2(x,t_0)\psi_M(x_1-\rho_1(t_0)-y_0)dx \lesssim J_{y_0,t_0}(t_0)+e^{-y_0/M} \, ,
 \end{displaymath}
 which yields estimate \eqref{nlEnergymonotonicity.3} in view of \eqref{nlEnergymonotonicity.9}.
 \end{proof}
 
 \begin{corollary} \label{coro.nlmonotonicity}
 Let $u \in C(\mathbb R : H^1(\mathbb R^2))$ be a solution of \eqref{ZK} satisfying \eqref{modulation.2}-\eqref{modulation.5} and the decay assumption in $x_1$ \eqref{NonLinearLiouville2}. Assume that $\epsilon_0$ in Lemma \ref{modulation} is chosen small enough, then, there exists $\tilde{\sigma}>0$ such that
 \begin{equation} \label{coro.nlmonotonicity.1}
 \sup_{t \in \mathbb R} \int |\nabla u|^2(x+\rho(t),t)e^{\tilde{\sigma}|x_1|}dx \lesssim 1 \, .
 \end{equation}
Moreover, there exists $M_0 \ge 4$ such that
\begin{equation} \label{coro.nlmonotonicity.2}
\int_{-\infty}^{t_0}\int \Big(|\nabla^2 u|^2+|\nabla u|^2+u^2\Big)(x,t)e^{\tilde{x}_1/M} \lesssim e^{-y_0/M} \, ,
\end{equation} 
for all $M \ge M_0$, $t_0 \in \mathbb R$, $y_0>0$ and where $\tilde{x}_1$ is defined in \eqref{tildex1}.
 \end{corollary}
 
 \begin{proof} Estimate \eqref{coro.nlmonotonicity.1} follows from the first inequality in \eqref{nlEnergymonotonicity.3} arguing exactly as in the proof of Corollary \ref{coro.monotonicity}. 
 
 Now, we prove estimate \eqref{coro.nlmonotonicity.2}. Since $\psi'_M(\tilde{x}_1) \gtrsim e^{\tilde{x}_1/M}$ for $\tilde{x}_1<0$, it follows from \eqref{nlL2monotonicity.3} and \eqref{nlEnergymonotonicity.3} that
\begin{displaymath} 
\int_{-\infty}^{t_0}\int_{x_1<\rho_1(t_0)-\frac12(t_0-t)+\tilde{y}_0} \Big(|\nabla^2 u|^2+|\nabla u|^2+u^2\Big)(x,t)e^{(x_1-\rho_1(t_0)+\frac12(t_0-t))/M} dxdt\lesssim 1\, ,
\end{displaymath}
for all $\tilde{y}_0>0$. This yields \eqref{coro.nlmonotonicity.2} by passing to the limit as $\tilde{y}_0 \to +\infty$ and multiplying the result by $e^{-\tilde{y}_0/M}$.
 \end{proof}
 
Due to the failure of the Sobolev embedding $H^1 \hookrightarrow L^{\infty}$ in two dimensions, we are not able at this point to derive monotonicity properties for $(\partial^{\alpha}u)^2$ at any order of  $|\alpha|$ by induction as it was done for the KdV equation in \cite{LM}.  We need first to derive a monotonicity property in $H^2$, which in turn will implies that the solutions of \eqref{ZK} close to a soliton are bounded in $H^3$.
 \begin{lemma} \label{nlH2monotonicity}
Assume that $u \in C(\mathbb R : H^1(\mathbb R^2))$ is a solution of \eqref{ZK} satisfying \eqref{modulation.2}-\eqref{modulation.5} and the decay property \eqref{NonLinearLiouville2}. If $\epsilon_0$ in Lemma \ref{modulation} is chosen small enough, then  there exists $M_0 \ge 12$ such that
\begin{equation} \label{nlH2monotonicity.1}
\begin{split}
\int &|\nabla^2 u|^2(x,t_0)\psi_M(x_1-\rho_1(t_0)-y_0)dx \\ & 
+\int_{-\infty}^{t_0}\int \big( |\nabla^2 u|^2+\sum_{|\alpha|=3}(\partial^{\alpha} u)^2\big)(x,t)\psi_M'(\tilde{x}_1) dxdt\lesssim e^{-y_0/M} \, .
\end{split}
\end{equation}
for  $M\ge M_0$, $y_0>0$, $t_0 \in \mathbb R$ and $t \le t_0$ where $\psi_M$ is defined as in \eqref{psi0}--\eqref{psi} and $\tilde{x}_1$ is defined as in \eqref{tildex1}. 
\end{lemma}

\begin{proof} Arguing as previously, we get that 
\begin{equation} \label{nlH2monotonicity.2}
\frac{d}{dt}\int (\partial^{\alpha}u)^2 \psi_M(\tilde{x}_1)dx \le -\int\big(|\nabla \partial^{\alpha}u|^2+\frac14(\partial^{\alpha}u)^2 \big)\psi_M'(\tilde{x}_1)dx+\mathcal{N}_{\alpha}(u) \, ,
\end{equation}
for all multi-index $\alpha \in \mathbb N^2 $, where 
\begin{equation} \label{nlH2monotonicity.3}
\mathcal{N}_{\alpha}(u) =2\int \partial^{\alpha}\partial_{x_1}(u^2)\partial^{\alpha}u \psi_M(\tilde{x}_1)dx\, .
\end{equation}

Here, we explain how to handle $\mathcal{N}_{\alpha}$ when $|\alpha|=2$. We will only look at the nonlinearity $\mathcal{N}_{(2,0)}(u)$ since the other nonlinearities $\mathcal{N}_{\alpha}(u)$ with $|\alpha|=2$ could be treated similarly. Integrations by parts and the Leibniz rule give that 
\begin{equation} \label{nlH2monotonicity.4}
\begin{split}
\mathcal{N}_{(2,0)}(u)&=-20\int u\partial_{x_1}^3u\partial_{x_1}^2u\psi_M(\tilde{x}_1)dx
-24\int u(\partial_{x_1}^2u)^2\psi_M'(\tilde{x}_1)dx \\ & 
=:-20\mathcal{N}_{(2,0)}^1(u)-24\mathcal{N}_{(2,0)}^2(u) \, .
\end{split}
\end{equation}
To deal with $\mathcal{N}_{(2,0)}^1(u)$, we use the decomposition 
\begin{equation} \label{nlH2monotonicity.5}
\mathcal{N}_{(2,0)}^1(u)=\mathcal{N}_{(2,0)}^{1.1}(u)+\mathcal{N}_{(2,0)}^{1.2}(u)\, .
\end{equation} 
where
\begin{displaymath}
\mathcal{N}_{(2,0)}^{1.1}(u)=\int Q_c(\cdot-\rho)\partial_{x_1}^3u\partial_{x_1}^2u\psi_M(\tilde{x}_1)dx \, ,
\end{displaymath} 
and
\begin{displaymath} 
\mathcal{N}_{(2,0)}^{1.2}(u)=\int \big(u-Q_c(\cdot-\rho) \big)\partial_{x_1}^3u\partial_{x_1}^2u\psi_M(\tilde{x}_1)dx \, .
\end{displaymath} 
Integrating by parts and arguing as in \eqref{monotonicity11}, we get that 
\begin{equation} \label{nlH2monotonicity.6}
\mathcal{N}_{(2,0)}^{1.1}(u)=-\frac12\int (\partial_{x_1}^2u)^2 \partial_{x_1}\big(Q_c(\cdot-\rho)\psi_M(\tilde{x}_1)\big)dx 
\lesssim \int (\partial_{x_1}^2u)^2 \psi_M'(\tilde{x}_1)dx \, .
\end{equation}
On the other hand, it follows from Young's inequality that 
\begin{equation} \label{nlH2monotonicity.7}
\begin{split}
\mathcal{N}_{(2,0)}^{1.2}(u)&\le \frac1{32}\int (\partial_{x_1}^3u)^2 \psi_M'(\tilde{x}_1)dx 
\\ & \quad+ C\int \big(u-Q_c(\cdot-\rho) \big)^2(\partial_{x_1}^2u)^2\frac{\psi_M(\tilde{x}_1)^2}{\psi_M'(\tilde{x}_1)}dx \, .
\end{split}
\end{equation}
Applying Young's inequality again, we bound the second term on the right-hand side of \eqref{nlH2monotonicity.7} by 
\begin{displaymath}
C\int \big(u-Q_c(\cdot-\rho) \big)^4(\partial_{x_1}^2u)^2\psi_M'(\tilde{x}_1)dx+C\int (\partial_{x_1}^2u)^2\frac{\psi_M(\tilde{x}_1)^4}{\psi_M'(\tilde{x}_1)^3}dx \, .
\end{displaymath}
Now, since 
$$\frac{\psi_M(\tilde{x}_1)^4}{\psi_M'(\tilde{x}_1)^3} \lesssim \psi'_M(\tilde{x}_1), \ \text{for} \ \tilde{x}_1 \le 0 \quad \text{and} \quad \frac{\psi_M(\tilde{x}_1)^4}{\psi_M'(\tilde{x}_1)^3} \lesssim e^{3\tilde{x}_1/M}, \ \text{for} \ \tilde{x}_1 >0 \, ,$$ we deduce from the Sobolev embedding $H^1(\mathbb R^2) \hookrightarrow L^6(\mathbb R^2)$ and \eqref{modulation.3} that 
\begin{equation} \label{nlH2monotonicity.8}
\begin{split}
\mathcal{N}_{(2,0)}^{1.2}(u)&\le \frac1{16}\int \big( (\partial_{x_1}^3u)^2 +(\partial_{x_1}^2u)^2\big)\psi_M'(\tilde{x}_1)dx  \\& \quad
+C\int  (\partial_{x_1}^2u)^2 \big( \psi'_M(\tilde{x}_1)+e^{3\tilde{x}_1/M} \big)dx \, ,
\end{split}
\end{equation}
assuming $\epsilon_0$ is chosen small enough. Now, arguing as in \eqref{nlL2monotonicity.5}--\eqref{nlL2monotonicity.6}, we obtain that 
\begin{equation} \label{nlH2monotonicity.9}
\mathcal{N}_{(2,0)}^{2}(u)\le \frac1{16}\int \big( (\partial_{x_1}^3u)^2 +(\partial_{x_1}^2u)^2\big)\psi_M'(\tilde{x}_1)dx  
+C\int  (\partial_{x_1}^2u)^2  \psi'_M(\tilde{x}_1)dx \, .
\end{equation}

Therefore, we deduce gathering \eqref{nlH2monotonicity.2}--\eqref{nlH2monotonicity.3} for $|\alpha|=2$ and arguing as in \eqref{nlH2monotonicity.4}--\eqref{nlH2monotonicity.6}, \eqref{nlH2monotonicity.8} and \eqref{nlH2monotonicity.9} that 
\begin{displaymath}
\begin{split}
\frac{d}{dt}\int |\nabla^2 u|^2(x,t)\psi_M(\tilde{x}_1)dx 
+\frac18\sum_{|\alpha|=3}&\int (\partial^{\alpha} u)^2(x,t)\psi_M'(\tilde{x}_1) dx \\&\lesssim 
 \int  |\nabla^2 u|^2 \big( \psi'_M(\tilde{x}_1)+e^{3\tilde{x}_1/M} \big)dx \, ,
\end{split}
\end{displaymath}
which implies \eqref{nlH2monotonicity.1} by integrating between $t$ and $t_0$, letting $t \to -\infty$ and using \eqref{nlEnergymonotonicity.3} and \eqref{coro.nlmonotonicity.2} with $\tilde{M}=M/3$.
\end{proof}

\begin{corollary} \label{coro.H2monotonicity}
Assume that $u \in C(\mathbb R : H^1(\mathbb R^2))$ is a solution of \eqref{ZK} satisfying \eqref{modulation.2}-\eqref{modulation.5} and the decay property \eqref{NonLinearLiouville2}. If $\epsilon_0$ in Lemma \ref{modulation} is chosen small enough, then $u$ is bounded in $H^3$.
\end{corollary}

\begin{proof} We follow the strategy of \cite{LM}. Arguing as in the proof of Corollary \ref{coro.nlmonotonicity}, we deduce from \eqref{nlH2monotonicity.1} that
\begin{equation} \label{coro.H2monotonicity.1}
\int_{-\infty}^{t_0}\int \Big(|\nabla^2u|^2+\sum_{|\alpha|=3}(\partial^{\alpha}u)^2\Big)(x,t)e^{x_1-\rho_1(t_0)+\frac12(t_0-t)/M}dxdt \lesssim 1\, ,
\end{equation}
for any $M \ge 12$. Moreover, it follows from \eqref{modulation.3}, \eqref{modulation.5} and the mean value theorem that $\rho_1(t)-\rho_1(t_0)+\frac12(t_0-t) \ge -\frac34(t_0-t) \ge -\frac34$ if $t \in [t_0-1,t]$. Thus we deduce from \eqref{coro.nlmonotonicity.2} and \eqref{coro.H2monotonicity.1}, (after applying the same argument to $v(x,t)=u(-x,-t)$), that
\begin{equation} \label{coro.H2monotonicity.2}
\int_{t_0-1}^{t_0}\int \Big(u^2+|\nabla u|^2+|\nabla^2 u|^2+\sum_{|\alpha|=3}(\partial^{\alpha}u)^2\Big)(x+\rho(t),t)e^{|x_1|/M}dxdt \lesssim 1 \, .
\end{equation}

In particular, we deduce from \eqref{coro.H2monotonicity.2} that there exists $\theta>0$ such that 
\begin{equation} \label{coro.H2monotonicity.3}
\int_{t_0-1}^{t_0}\int \Big(u^2+|\nabla u|^2+|\nabla^2 u|^2+\sum_{|\alpha|=3}(\partial^{\alpha}u)^2\Big)(x,t)dxdt \le \theta\, ,
\end{equation}
for all $t_0 \in \mathbb R$. Thus, $u(\cdot,t) \in H^3(\mathbb R^2)$ at least for some $t \in \mathbb R$ and  the persistence property of the well-posedness result implies that $u \in C(\mathbb R : H^3(\mathbb R^2))$. 

On the other hand, direct energy estimates using the equation \eqref{ZK} give 
\begin{displaymath}
\frac{d}{dt}\|u\|_{H^3}^2 \lesssim \|\nabla u\|_{L^{\infty}_x}\|u\|_{H^3}^2 \, .
\end{displaymath}
Then, we deduce from Gronwall's inequality that
\begin{displaymath} 
\|u(\cdot,t_0)\|_{H^3}^2 \le e^{C\int_{t_1}^{t_0}\|\nabla u(\cdot,s)\|_{L^{\infty}_x}ds}\|u(\cdot,t_1)\|_{H^3}^2 \, ,
\end{displaymath}
for all $t_1<t_0$. Therefore, it follows from \eqref{coro.H2monotonicity.3} and the Sobolev embedding $H^2(\mathbb R^2) \hookrightarrow L^{\infty}(\mathbb R^2)$ that 
\begin{equation} \label{coro.H2monotonicity.5}
\|u(\cdot,t_0)\|_{H^3}^2 \le e^{C(t_0-t_1)^{\frac12}\theta}\|u(\cdot,t_1)\|_{H^3}^2 \, ,
\end{equation}
for all $t_1 <t_0$.

Now, fix some $t_0 \in \mathbb R$. From  \eqref{coro.H2monotonicity.3}, there exists $t_1 \in (t_0-1,t_0)$ such that $\|u(\cdot,t_1)\|_{H^3}^2 \le \theta$. Thus, we infer from \eqref{coro.H2monotonicity.5} that 
\begin{displaymath} 
\|u(\cdot,t_0)\|_{H^3}^2 \le e^{C\theta}\theta \, ,
\end{displaymath}
which finishes the proof of Corollary \ref{coro.H2monotonicity}, since $\theta$ does not depend on $t_0$.
\end{proof}

The $H^3$ bound obtained in Corollary \ref{coro.H2monotonicity} allows us to derive a monotonicity properties for $(\partial^{\alpha}u)^2$ at any order of  $|\alpha|$. 
\begin{lemma} \label{nlhomonotonicity}
Assume that $u \in C(\mathbb R : H^1(\mathbb R^2))$ is a solution of \eqref{ZK} satisfying \eqref{modulation.2}-\eqref{modulation.5} and the decay property \eqref{NonLinearLiouville2}. Let $k \in \mathbb N$ be given. If $\epsilon_0$ in Lemma \ref{modulation} is chosen small enough, then, $u$ is bounded in $H^{k+1}$ and
\begin{equation} \label{nlhomonotonicity.1}
\begin{split}
\int &(\partial^{\alpha} u)^2(x,t_0)\psi_M(x_1-\rho_1(t_0)-y_0)dx \\ & 
+\int_{-\infty}^{t_0}\int \big( |\nabla\partial^{\alpha} u|^2+(\partial^{\alpha} u)^2\big)(x,t)\psi_M'(\tilde{x}_1) dxdt\lesssim e^{-y_0/M} \, .
\end{split}
\end{equation}
for any multi-index $\alpha \in \mathbb N^2$ with $|\alpha|=k$,  $M\ge 12$, $y_0>0$, $t_0 \in \mathbb R$ and $t \le t_0$ where $\psi_M$ is defined as in \eqref{psi0}--\eqref{psi} and $\tilde{x}_1$ is defined as in \eqref{tildex1}. 
\end{lemma}

\begin{proof} We know from Lemmas \ref{nlL2monotonicity}, \ref{nlEnergymonotonicity} and \ref{nlH2monotonicity} that \eqref{nlhomonotonicity.1} holds true for $k=0,1,2$. 

To prove \eqref{nlhomonotonicity.1} for the general case, we argue by induction on $k$ . Let $k \in \mathbb N$ be given such that $k \ge 3$. Assume that \eqref{nlhomonotonicity.1} is true for any $\tilde{\alpha} \in \mathbb N^2$ with $|\tilde{\alpha}| \le k-1$. Let $\alpha \in \mathbb N^2$ with $|\alpha|=k$.

Recalling estimate \eqref{nlH2monotonicity.2}, we need to control $\mathcal{N}_{\alpha}(u)$ defined in \eqref{nlH2monotonicity.3}. By using the Leibniz rule, we have that 
\begin{equation}  \label{nlhomonotonicity.2}
\mathcal{N}_{\alpha}(u)=\sum_{0<\beta<\alpha}C_{\beta}\int \partial_{x_1}\big(\partial^{\beta}u\partial^{\alpha-\beta}u\big)\partial^{\alpha}u\psi_M(\tilde{x}_1)dx+4\int \partial_{x_1}\big(u\partial^{\alpha}u \big)\partial^{\alpha}u\psi_M(\tilde{x}_1)dx \, .
\end{equation}
We only treat the last term appearing on the right-hand side of \eqref{nlhomonotonicity.2} which is the most difficult. After integrating by parts, we get that
\begin{displaymath}  
\begin{split}
\int \partial_{x_1}\big(u\partial^{\alpha}u \big)\partial^{\alpha}u\psi_M(\tilde{x}_1)dx&= 
\frac12\int \partial_{x_1}u(\partial^{\alpha}u)^2\psi_M(\tilde{x}_1)dx 
-\frac12\int u(\partial^{\alpha}u)^2\psi_M'(\tilde{x}_1)dx \, .
\end{split}
\end{displaymath}
Thus, the Sobolev embedding $H^2(\mathbb R^2) \hookrightarrow L^{\infty}(\mathbb R^2)$ and the $H^3$-bound obtained in  Corollary \ref{coro.H2monotonicity} imply that
\begin{equation} \label{nlhomonotonicity.3}
\int \partial_{x_1}\big(u\partial^{\alpha}u \big)\partial^{\alpha}u\psi_M(\tilde{x}_1)dx\le C\int (\partial^{\alpha}u)^2\big(\psi_M(\tilde{x}_1)+\psi_M'(\tilde{x}_1)\big)dx \, .
\end{equation}

Estimate \eqref{nlhomonotonicity.1} follows then by integrating \eqref{nlH2monotonicity.2} between $t$ and $t_0$, using \eqref{nlhomonotonicity.2}--\eqref{nlhomonotonicity.3} and the induction hypothesis to control the term corresponding to $\mathcal{N}_{\alpha}(u)$ and  letting $t \to -\infty $ as previously. Finally, the $H^{k+1}$-bound can be deduced from \eqref{nlhomonotonicity.1} arguing as we did in the proof of Corollary \ref{coro.H2monotonicity}.
\end{proof}

Arguing as in the proof of Corollary \ref{coro.nlmonotonicity}, we finally deduce from Lemma \ref{nlhomonotonicity} the exponential decay in the $x_1$-direction at any order.
 \begin{corollary} \label{coro.homonotonicity}
 Let $u \in C(\mathbb R :  H^1(\mathbb R^2))$ be a solution of \eqref{ZK} satisfying \eqref{modulation.2}-\eqref{modulation.5} and the decay assumption in $x_1$ \eqref{NonLinearLiouville2}. Assume that $\epsilon_0$ in Lemma \ref{modulation} is chosen small enough, then there exists $\tilde{\sigma}>0$ such that
 \begin{equation} \label{coro.homonotonicity.1}
 \sup_{t \in \mathbb R} \int(\partial^{\alpha}u)^2(x+\rho(t),t)e^{\tilde{\sigma}|x_1|}dx \lesssim 1 \, ,
 \end{equation}
for any $\alpha \in \mathbb N^2$.
 \end{corollary}

\subsection{Proof of Theorem \ref{NonLinearLiouville}} 
 We first decompose $u$ by using the modulation theory in Lemma \ref{modulation}. Then, we can assume that there exist $\rho=\big(\rho_1,\rho_2 \big) \in C^1(\mathbb R : \mathbb R^2)$ and $c \in C^1(\mathbb R : \mathbb R)$ such that
\begin{equation} \label{NonLinearLiouville5} 
\eta(x,t)=u(x+\rho(t),t)-Q_{c(t)}(x)
\end{equation}
satisfies \eqref{modulation.3}--\eqref{modulation.5}. Moreover, due to Lemma \ref{modulationdecay}, $u$ still satisfies the exponential decay assumption \eqref{NonLinearLiouville2} in the $x_1$ direction.\\

\noindent \underline{\textit{Introduction of a dual problem.}} Following Martel and Merle \cite{MM2}, we will work on a dual problem. Let us define 
\begin{equation} \label{NonLinearLiouville6}
v=\mathcal{L}_c\eta-\eta^2=-\Delta \eta+c\eta-2Q_c\eta-\eta^2 \, .
\end{equation}
By using the equation satisfied by $\eta$ in \eqref{equationeta}, we get that 
\begin{displaymath} 
\begin{split}
\partial_tv&=\mathcal{L}_c\partial_t\eta-2\eta\partial_t\eta+c'\eta-2c'\Lambda Q_c \eta \\ 
& =\mathcal{L}_c\partial_{x_1}v+(\rho_1'-c)\mathcal{L}_c\partial_{x_1}\eta+\rho_2'\mathcal{L}_c\partial_{x_2}\eta+c'Q_c \\ & \quad -2\eta\partial_{x_1}v-2(\rho_1'-c)\eta\partial_{x_1}(Q_c+\eta)
-2\rho_2'\eta\partial_{x_2}(Q_c+\eta)+c'\eta \, .
\end{split}
\end{displaymath}
Now, direct computations give that 
\begin{displaymath}
\mathcal{L}_c\partial_{x_1}\eta=\partial_{x_1}\mathcal{L}_c\eta+2\eta\partial_{x_1}Q_c=\partial_{x_1}v+2\eta\partial_{x_1}\eta+2\eta\partial_{x_1}Q_c  
\end{displaymath}
and similarly
\begin{displaymath}
\mathcal{L}_c\partial_{x_2}\eta=\partial_{x_2}v+2\eta\partial_{x_2}\eta+2\eta\partial_{x_2}Q_c \, .
\end{displaymath}
Thus, the equation satisfied by $v$ writes
\begin{equation} \label{NonLinearLiouville7}
\partial_tv=\mathcal{L}_c\partial_{x_1}v-2\eta\partial_{x_1}v+(\rho_1'-c)\partial_{x_1}v+\rho_2'\partial_{x_2}v+c'(Q_c+\eta) \, .
\end{equation}

\noindent \underline{\textit{Almost orthogonality conditions.}} We have from the definition of $v$ in \eqref{NonLinearLiouville6} and  \eqref{kernel} that 
\begin{displaymath} 
\int v\partial_{x_i}Q_cdx=\int \eta\mathcal{L}_c\partial_{x_i}Q_cdx-\int\eta^2\partial_{x_i}Q_cdx=-\int\eta^2\partial_{x_i}Q_cdx \, ,
\end{displaymath}
for $i=1,2$, so that 
\begin{equation} \label{NonLinearLiouville8}
\Big|\int v\partial_{x_i}Q_cdx\Big| \le  \|\partial_{x_i}Q_c\|_{L^{\infty}}\|\eta\|_{L^2}^2 \quad i=1,2 \, .
\end{equation}
In a similar way, it follows from formula \eqref{LambdaQeq} and  the third orthogonality condition in \eqref{modulation.4} 
\begin{equation}\label{NonLinearLiouville8a}
\int v\Lambda Q_cdx=\int \eta\mathcal{L}_c\Lambda Q_cdx-\int\eta^2\Lambda Q_cdx=-\int \eta Q_cdx-\int\eta^2\Lambda Q_c \, ,
\end{equation}
 so that 
\begin{equation} \label{NonLinearLiouville9}
\Big|\int v\Lambda Q_cdx \Big| \le  \|\Lambda Q_c\|_{L^{\infty}}\|\eta\|_{L^2}^2 \, .
\end{equation}
 
It was proved in \cite{We} that the bilinear form $H(v,v)=(\mathcal{L}_cv,v)$ is coercive under the orthogonality conditions \eqref{modulation.4}. Thus, there exists $\lambda_1>0$ such that 
\begin{displaymath}
(v,\eta) =(\mathcal{L}_c\eta,\eta)-(\eta^2,\eta) \ge \lambda_1\|\eta\|_{H^1}^2-\int\eta^3dx \, .
\end{displaymath}
The Sobolev embedding $H^1(\mathbb R^2) \hookrightarrow L^3(\mathbb R^2)$ and \eqref{modulation.3} implies then that
\begin{displaymath}
(v,\eta)  \ge (\lambda_1-C\|\eta\|_{H^1})\|\eta\|_{H^1}^2 \ge(\lambda_1-K_0\epsilon_0)\|\eta\|_{H^1}^2 \ge \frac{\lambda_1}2\|\eta\|_{H^1}^2 \, ,
\end{displaymath}
provided $\epsilon_0$ is chosen small enough. Therefore, we deduce from the Cauchy-Schwarz inequality that 
\begin{equation} \label{NonLinearLiouville10}
\|\eta\|_{H^1} \le \frac2{\lambda_1}\|v\|_{L^2} \, .
\end{equation}

\noindent \underline{\textit{Exponential decay in the $x_1$ direction.}} It is a consequence of the monotonicity property that there exists $\tilde{\sigma}>0$ such that
\begin{equation} \label{NonLinearLiouville11}
\int_{x_2}v^2(x_1,x_2,t)dx_2 \lesssim e^{-\tilde{\sigma}|x_1|}, \quad \forall \, (x_1,t) \in \mathbb R^2 \, .
\end{equation}
Indeed, we get from the definition of $v$ in \eqref{NonLinearLiouville6}, the decay properties of $Q$ in \eqref{elliptPDE1} and formula \eqref{coro.homonotonicity.1} with $|\alpha| \le 3$ that 
\begin{equation} \label{NonLinearLiouville12}
\sup_{t\in \mathbb R} \int \big(v^2+(\partial_{x_1}v)^2 \big)e^{\tilde{\sigma}|x_1|}dx \lesssim 1 \, .
\end{equation}
Estimate \eqref{NonLinearLiouville11} is then deduced from estimate \eqref{NonLinearLiouville12} just as in the linear case (see \eqref{LinearLiouville8}). \\ 

\noindent \underline{\textit{Virial type estimate.}} Let $A$ be a positive number to be chosen later. We define $\varphi_A=\varphi_A(x_1)$ and $\phi_A=\phi_A(x_1)$ as in \eqref{phi}--\eqref{varphiA}. Then $\phi_A$ and $\varphi_A$ satisfy the properties in \eqref{varphiA2}. 

By using \eqref{NonLinearLiouville7}, a direct computation gives that
\begin{displaymath}
\begin{split}
-\frac12\frac{d}{dt}\int \varphi_Av^2dx&=-\int \varphi_A\mathcal{L}_c\partial_{x_1}v vdx+2\int \varphi_A \eta \partial_{x_1}v v dx-(\rho_1'-c)\int\varphi_A \partial_{x_1}v v dx 
 \\ & \quad-\rho_2'\int \varphi_A \partial_{x_2}v v dx-c'\int \varphi_A (Q_c+\eta) v dx
\end{split}
\end{displaymath}
Then, arguing as in \eqref{LinearLiouville9}, we deduce that 
\begin{equation} \label{NonLinearLiouville13}
\begin{split}
-\frac12\frac{d}{dt}\int \varphi_Av^2dx&= \int \phi_A (\partial_{x_1}v)^2dx+\frac12\int \phi_A \big(|\nabla v|^2+v^2-2Q_cv^2 \big)dx \\ & \quad-\frac12\int \phi_A''v^2dx 
-\int\varphi_A\partial_{x_1}Q v^2dx+\mathcal{R}_A(\eta,v) \, ,
\end{split}
\end{equation}
where 
\begin{equation} \label{RA.0}
\begin{split}
\mathcal{R}_A(\eta,v)&=-\int \phi_A \eta v^2dx-\int \varphi_A \partial_{x_1}\eta v^2\\ &\quad +\frac{\rho_1'-c}2\int \phi_A v^2 dx-c'\int \varphi_A (Q_c+\eta)vdx \, .
\end{split}
\end{equation}

The following lemmas will be proved in the next subsection. 
\begin{lemma} \label{nonlinearcoercivity}
Consider the bilinear form 
\begin{equation} \label{nonlinearcoercivity.1}
H_A(v,w)=\int \phi_A \big(\nabla v \cdot \nabla w +vw-2Q_cvw\big)dx \, .
\end{equation}
Then, there exists $\lambda_2>0$ and $A_2>0$ such that
\begin{equation} \label{nonlinearcoercivity.1b}
H_A(v,v) \ge \lambda_2\int \phi_A \big(|\nabla v|^2+v^2 \big)dx \, ,
\end{equation}
for all functions $v$ defined in \eqref{NonLinearLiouville5}--\eqref{NonLinearLiouville6} provided $A \ge A_2$ and $\epsilon_0>0$ is chosen small enough.
\end{lemma}

\begin{lemma} \label{RA}
There exist $K_3>0$ and $A_3>0$ such that 
\begin{equation} \label{RA.1}
\Big|\mathcal{R}_A(\eta,v)\Big| \le K_3 A \|\eta\|_{H^1}\int \phi_A \big(|\nabla v|^2+v^2 \big)dx \, ,
\end{equation} 
for all $(\eta,v)$ defined as in \eqref{NonLinearLiouville5}--\eqref{NonLinearLiouville6} provided $A \ge A_3$.
\end{lemma}

By fixing $A = \max\{A_2,A_3,2\sqrt{C/\lambda_2}\}$ (where $C$ is the positive constant appearing in \eqref{phi2}), we deduce using \eqref{modulation.3}, \eqref{NonLinearLiouville13}, \eqref{nonlinearcoercivity.1b}, \eqref{RA.1} and arguing as in \eqref{LinearLiouville10} that
\begin{equation} \label{NonLinearLiouville14}
-\frac12\frac{d}{dt}\int \varphi_A(x_1)v^2dx \ge\frac{\lambda_2}8\int \phi_A(x_1)\big(|\nabla v|^2+v^2 \big)dx \, ,
\end{equation}
provided $\epsilon_0$ is small enough.
Then, we conclude by using \eqref{NonLinearLiouville11} and arguing as in \eqref{LinearLiouville11}--\eqref{LinearLiouville16} that $v(x,t)=0$ for all $(x,t) \in \mathbb R^3$. This implies in view of \eqref{NonLinearLiouville5} and \eqref{NonLinearLiouville10} that 
\begin{equation} \label{NonLinearLiouville15}
\eta(x,t)=0 \quad \Rightarrow \quad u(x+\rho(t),t)=Q_{c(t)}(x) \quad \forall \, (x,t) \in \mathbb R^3 \, .
\end{equation}
Moreover, \eqref{modulation.5} yields
\begin{displaymath}
c(t)=c(0), \quad \rho_1(t)=\rho_1(0)+c(0)t \quad \text{and} \quad \rho_2(t)=\rho_2(0) \, ,
\end{displaymath} 
which concludes the proof of Theorem \ref{NonLinearLiouville}.

\subsection{Coercivity of the bilinear form $H_A(v,v)$} The aim of this subsection is to prove Lemmas \ref{nonlinearcoercivity} and \ref{RA}. We begin first with a technical lemma. 
\begin{lemma} \label{technicalcoercivity}
There exists $K_4>0$ and $A_4$ such that 
\begin{equation} \label{technicalcoercivity.1}
\|\sqrt{\phi_A}\eta\|_{H^1} \le K_4 \|\sqrt{\phi_A}v\|_{L^2} \, ,
\end{equation}
if $A \ge A_4$ and $\epsilon_0$ small enough.
\end{lemma}

\begin{proof} A direct computation shows that 
\begin{displaymath}
\int \phi_Av\eta dx=H_A(\eta,\eta)+\int\phi_A''\eta^2 dx-\int \phi_A \eta^3 dx \, .
\end{displaymath}
Since $\eta$ satisfies the orthogonality conditions \eqref{modulation.4}, we deduce from Lemma \ref{coercivity}, \eqref{phi2} and the Sobolev embedding $H^1(\mathbb R^2) \hookrightarrow L^3(\mathbb R^2)$ that
\begin{displaymath} 
\|\sqrt{\phi_A}v\|_{L^2}\|\sqrt{\phi_A}\eta\|_{L^2} \ge \lambda \int \phi_A\big(|\nabla \eta|^2+\eta^2\big)dx-\frac{C}{A^2}\int \phi_A \eta^2dx-\|\eta\|_{H^1}\|\sqrt{\phi_A}\eta\|_{H^1}^2 \, ,
\end{displaymath}
which implies \eqref{technicalcoercivity.1} in view of \eqref{modulation.3} if we choose $A \ge A_5=\sqrt{\frac{2C}{\lambda}}$ and $\epsilon_0$ is chosen small enough.
\end{proof}

\begin{proof}[Proof of Lemma \ref{nonlinearcoercivity}] The proof of Lemma \ref{nonlinearcoercivity} follows the lines of the one of Lemma \ref{coercivity}. In order to use Lemma \ref{lemma.coercivity}, we need to verify that $\sqrt{\phi_A}v$ satisfies \eqref{lemma.coercivity.2}.

By using the definition of $v$ in \eqref{NonLinearLiouville6}, we compute for $i=1,2$,
\begin{equation} \label{nonlinearcoercivity.2}
\Big|\int \sqrt{\phi_A}v\partial_{x_i}Q_c dx \Big|  \le \Big|\int \sqrt{\phi_A}\mathcal{L}_c\eta\partial_{x_i}Q_c dx \Big| +\Big|\int \sqrt{\phi_A}\eta^2\partial_{x_i}Q_c dx \Big| 
\end{equation}
On the one hand, by using the Cauchy-Schwarz inequality, \eqref{modulation.3} and \eqref{technicalcoercivity.1}, we have that
\begin{equation} \label{nonlinearcoercivity.3}
\Big|\int \sqrt{\phi_A}\eta^2\partial_{x_i}Q_c dx \Big| \le \|\partial_{x_i}Q_c\|_{L^{\infty}}\|\eta\|_{L^2}\|\sqrt{\phi_A}\eta\|_{L^2}\le \frac{\kappa}2\|\sqrt{\phi_A}v\|_{L^2} \, ,
\end{equation}
if $\epsilon_0$ is chosen small enough. On the other hand, it follows from \eqref{kernel} that
\begin{displaymath} 
\Big|\int \sqrt{\phi_A}\mathcal{L}_c\eta\partial_{x_i}Q_c dx \Big| \le 
\Big|\int \big(\sqrt{\phi_A}\big)''\eta\partial_{x_i}Q_c dx \Big|
+2\Big|\int \big(\sqrt{\phi_A}\big)'\eta\partial^2_{x_1x_i}Q_c dx \Big| 
\end{displaymath}
Thus, \eqref{phi2} and \eqref{technicalcoercivity.1} yield 
\begin{equation} \label{nonlinearcoercivity.4}
\Big|\int \sqrt{\phi_A}\mathcal{L}_c\eta\partial_{x_i}Q_c dx \Big| \lesssim \frac1A\|\sqrt{\phi_A}\eta\|_{H^1} \le \frac{\kappa}2\|\sqrt{\phi_A}v\|_{L^2} \, ,
\end{equation}
if $A$ is chosen large enough. Hence, we deduce gathering \eqref{nonlinearcoercivity.2}--\eqref{nonlinearcoercivity.4} that 
\begin{equation} \label{nonlinearcoercivity.5}
\Big|\int \sqrt{\phi_A}v\partial_{x_i}Q_c dx \Big|  \le  \kappa\|\sqrt{\phi_A}v\|_{L^2} \, ,
\end{equation}
if $\epsilon_0$ is small enough and $A$ is large enough, where $\kappa$ is the small positive number given by Lemma \ref{lemma.coercivity}.

Arguing as above, we get that
\begin{equation} \label{nonlinearcoercivity.6}
\Big|\int \sqrt{\phi_A}v\Lambda Q_c dx \Big|  \le \Big|\int \sqrt{\phi_A}\eta\mathcal{L}_c \Lambda Q_c dx \Big| +\frac{\kappa}2\|\sqrt{\phi_A}v\|_{L^2} \, .
\end{equation}
To deal with the first term on the right-hand side of \eqref{nonlinearcoercivity.6}, observe from \eqref{LambdaQeq} and \eqref{modulation.4} that 
\begin{equation} \label{nonlinearcoercivity.7}
\int \sqrt{\phi_A}\eta\mathcal{L}_c \Lambda Q_c dx =\int \frac{1-\sqrt{\phi_A}}{\sqrt{\phi_A}}Q_c \sqrt{\phi_A}\eta dx \, .
\end{equation}
Moreover, we infer from the decay property of $Q_c$ (\textit{c.f.} \eqref{elliptPDE1}) and the definition of $\phi_A$ that 
\begin{displaymath}
\Big| \frac{1-\sqrt{\phi_A}}{\sqrt{\phi_A}}Q_c \Big| \lesssim e^{\frac32\frac{|x_1|}A}e^{-c^{\frac12}\delta |x|}\chi_{|x_1| \ge A} \lesssim e^{-c^{\frac12}\frac{\delta}2 |x|}\chi_{|x| \ge A} \, ,
\end{displaymath}
if $A$ is chosen large enough. Hence, we deduce from the Cauchy-Schwarz inequality, \eqref{technicalcoercivity.1} and \eqref{nonlinearcoercivity.7} that
\begin{equation} \label{nonlinearcoercivity.7b}
\Big|\int \sqrt{\phi_A}\eta\mathcal{L}_c \Lambda Q_c dx \Big| \lesssim \|e^{-c^{\frac12}\frac{\delta}2 |x|}\chi_{|x| \ge A}\|_{L^2}\|\sqrt{\phi_A}\eta\|_{L^2} \le \frac{\kappa}2 \|\sqrt{\phi_A}v\|_{L^2} \, ,
\end{equation} 
provided $A$ is large enough. Thus \eqref{nonlinearcoercivity.6} and \eqref{nonlinearcoercivity.7b} imply that 
\begin{equation} \label{nonlinearcoercivity.8}
\Big|\int \sqrt{\phi_A}v\Lambda Q_c dx \Big|  \le  \kappa\|\sqrt{\phi_A}v\|_{L^2} \, .
\end{equation}

With \eqref{nonlinearcoercivity.5} and \eqref{nonlinearcoercivity.8} in hand, we conclude the proof of Lemma \ref{nonlinearcoercivity} following the arguments in the proof of Lemma \ref{coercivity}.
\end{proof}

\begin{proof}[Proof of Lemma \ref{RA}] The Sobolev embedding $H^1(\mathbb R^2) \hookrightarrow L^3(\mathbb R^2)$ yields
\begin{equation} \label{RA.2}
\big| \int \phi_A \eta v^2dx\big| \lesssim \|\eta\|_{H^1}\|\sqrt{\phi_A}v\|_{H^1}^2 \lesssim 
\|\eta\|_{H^1}\int \phi_A\big(|\nabla v|^2+v^2 \big)dx \, .
\end{equation}
Moreover, we deduce from \eqref{modulation.5} and \eqref{technicalcoercivity.1} that there exists $A_3>0$ such that
\begin{displaymath}
|c'(t)|^{\frac12}+|\rho_1'(t)-c(t)| \lesssim \|\sqrt{\phi_A} \eta\|_{L^2} \lesssim \|\sqrt{\phi_A}v\|_{L^2} \, ,
\end{displaymath}
if $A \ge A_3$. Now we fix $A \ge A_3$. Thus,
\begin{equation} \label{RA.3}
|\rho_1'-c|\int \phi_A v^2 dx+|c'| \Big|\int \varphi_A (Q_c+\eta)vdx\Big| \lesssim \|\eta\|_{H^1}\int \phi_A\big(|\nabla v|^2+v^2 \big)dx \, .
\end{equation}
Finally, in addition to \eqref{phi}--\eqref{phi2}, we can assume that 
$|\varphi_A(x_1)| \le CA|\phi_A(x_1)|$ for some positive constant $C$. Then, it follows from H\"older's inequality and the Sobolev embedding $H^1(\mathbb R^2) \hookrightarrow L^4(\mathbb R^2)$ that 
\begin{equation} \label{RA.4}
\Big|\int \varphi_A\partial_{x_1}\eta v^2dx \Big| \le CA\int \big|\phi_A\partial_{x_1}\eta v^2\big|dx \lesssim A\|\eta\|_{H^1}\|\sqrt{\phi_A}v\|_{H^1}^2 \, .
\end{equation}
We conclude the proof of \eqref{RA.1} gathering \eqref{RA.0} and \eqref{RA.2}--\eqref{RA.4}.
\end{proof}

\bigskip

\section{Proof of the asymptotic stability result} \label{sectionAS}
This section is devoted to the proof of Theorem \ref{AsymptStab}. According to Remark \ref{remarkscaling}, we will assume in this section that $c_0=1$.
We follow the arguments of Martel and Merle in \cite{MM1,MM5,MM2} for the gKdV equations. The heart of the proof is the following proposition which states that a solution in a neighborhood of a soliton converges (up to a subsequence) to a limit object, whose emanating solution satisfies  a good decay property  in the $x_1$ direction. Then, due to the rigidity theorem proved in Section \ref{NonlinearRigidity}, this limit object has to be a soliton. 

\begin{proposition} \label{prop.AsymptStab}
Assume $d=2$. There exists $\epsilon_0>0$ such that if $0<\epsilon  \le \epsilon_0$ and $u \in C(\mathbb R : H^1(\mathbb R^2))$ is a solution of \eqref{ZK} satisfying 
\begin{equation} \label{prop.AsymptStab.1a}
\inf_{\tau \in \mathbb R^2}\|u(\cdot,t)-Q(\cdot-\tau)\|_{H^1} \le \epsilon, \quad \forall \, t \in \mathbb R \, ,
\end{equation}
then the following holds true. 

For any sequence $t_n \to +\infty$, there exists a subsequence $\{t_{n_k}\}$ and $\tilde{u}_0 \in H^1(\mathbb R)$ such that 
\begin{equation} \label{prop.AsymptStab.1}
u(\cdot+\rho(t_{n_k}),t_{n_k}) \underset{k \to +\infty}{\longrightarrow} \tilde{u}_0 \quad \text{in} \quad H^1(x_1>-A) \, ,
\end{equation}
for any $A>0$ and where $c(t)$ and $\rho(t)$ are the functions associated to the decomposition of $u$ given by the modulation theory in Lemma \ref{modulation}.

Moreover, the solution $\tilde{u}$ of \eqref{ZK} corresponding to $\tilde{u}(\cdot,0)=\tilde{u}_0$ satisfies 
\begin{equation} \label{prop.AsymptStab.2}
\|\tilde{u}(\cdot+\tilde{\rho}(t))-Q\|_{H^1} \lesssim \epsilon, \quad \forall \, t \in \mathbb R \,
\end{equation}
and
 \begin{equation} \label{prop.AsymptStab.3} 
  \int_{x_2}\tilde{u}^2(x+\tilde{\rho}(t),t)dx_2  \lesssim e^{-\tilde{\sigma} |x_1|}\, , \quad \forall \, (x_1,t) \in \mathbb R^2 \, ,
\end{equation}
for some positive constant $\tilde{\sigma}$, where $\tilde{c}(t)$ and $\tilde{\rho}(t)$ are the functions associated to the decomposition of $\tilde{u}$ given by the modulation theory in Lemma \ref{modulation}. Note also that $\tilde{\rho}(0)=0$.
 \end{proposition}
 
 First we give the proof of Theorem \ref{AsymptStab} assuming Proposition \ref{prop.AsymptStab}.
 \begin{proof}[Proof of Theorem \ref{AsymptStab}]
 Let $u$ be a solution of \eqref{ZK} satisfying the hypotheses of Theorem \ref{AsymptStab}. Assume moreover that $\epsilon_0$ is chosen small enough so that Theorem \ref{NonLinearLiouville} and Proposition \ref{prop.AsymptStab} hold true. 
 
 From Proposition \ref{prop.AsymptStab}, for any sequence $\{t_n\}$ with $t_n \to +\infty$, there exist a subsequence $\{t_{n_k}\}$, $\tilde{c}_0>0$ and $\tilde{u}_0 \in H^1(\mathbb R)$ such that 
  \begin{equation} \label{AsymptStab.3}
 c(t_{n_k}) \underset{k\to +\infty}{\longrightarrow} \tilde{c}_0 \quad \text{and} \quad u(\cdot+\rho(t_{n_k}),t_{n_k}) \underset{k \to +\infty}{\longrightarrow} \tilde{u}_0 \quad \text{in} \quad H^1(x_1>-A)\, ,
 \end{equation}
 for any $A>0$. Moreover, the solution $\tilde{u}$ of \eqref{ZK} satisfying $\tilde{u}(\cdot,0)=\tilde{u}_0$ satisfies \eqref{prop.AsymptStab.2}--\eqref{prop.AsymptStab.3} and it holds that $\tilde{c}(0)=\tilde{c}_0$ and $\tilde{\rho}(0)=0$.
 
 Thus, we deduce applying Theorem \ref{NonLinearLiouville} to $\tilde{u}$ that there exist $c_1$ close to $1$ and $\rho^1=(\rho^1_{1},\rho^1_{2}) \in \mathbb R^2$ such that 
 \begin{equation} \label{AsymptStab.4}
 \tilde{u}(x_1,x_2,t)=Q_{c_1}(x_1-c_1t-\rho^1_{1},x_2-\rho^1_{2}) \, .
 \end{equation} 
 The uniqueness of the decomposition in Lemma \ref{modulation} implies that $\rho^1=\tilde{\rho}(0)=0$ and $c_1=\tilde{c}_0$. Then, we deduce from \eqref{AsymptStab.3} and \eqref{AsymptStab.4} that $u(\cdot+\rho(t_{n_k}),t_{n_k})-Q_{c(t_{n_k})}$ tends to $0$ as $k$ tends to $+\infty$ in $H^1(x_1>-A)$. Since this is true for any sequence $\{t_n\}$ such that $t_n \to +\infty$, we conclude that
 \begin{equation} \label{AsymptStab.5b}
 u(\cdot+\rho(t),t)-Q_{c(t)} \underset{t \to +\infty}{\longrightarrow} 0\quad \text{in} \quad H^1(x_1>-A)\, , \end{equation}
 for any $A>0$.
 Then, if we define $\tilde{\eta}(x,t)=u(x,t)-Q_{c(t)}(x-\rho(t))$,  \eqref{AsymptStab.5b} implies that 
\begin{equation} \label{AsymptStab.10b}
\lim_{t \to +\infty}\int \big(|\nabla \tilde{\eta}|^2+\tilde{\eta}^2 \big) (x,t) \psi_M(x_1-\rho_1(t)+y_0)dx =0 \, ,
\end{equation}
for all $y_0>0$.

To prove the convergence of the scaling parameter $c(t)$ as $t$ tends to $+\infty$, we use Lemma \ref{tildemonotonicity}.  Thus, for any $\alpha>0$ and $M \ge 4$, there exists $y_1>0$ such that 
\begin{equation} \label{AsymptStab.7b} 
\int u^2(x,t)\psi_M(x_1-\rho_1(t)+y_0)dx \le \int u^2(x,t')\psi_M(x_1-\rho_1(t')+y_0)dx+\alpha \, ,
\end{equation}
for all $t \ge t'$ and $y_0>y_1$. On the other hand, it follows from \eqref{elliptPDE1} and \eqref{AsymptStab.5b} that there exist $y_2=y_2(\alpha)>0$ and $T_0=T_0(\alpha)>0$ such that 
\begin{equation} \label{AsymptStab.8b}
\Big|\int u^2(x,t)\psi_M(x_1-\rho_1(t)+y_0)dx -\int Q_{c(t)}^2(x)dx \Big| \le \alpha \, ,
\end{equation}
if $t \ge T_0$ and $y_0 >y_2$. Hence, we deduce combining \eqref{AsymptStab.7b} and \eqref{AsymptStab.8b} that 
\begin{displaymath}  
\int Q_{c(t)}^2(x)dx  \le \int Q_{c(t')}^2(x)dx + 3\alpha \, ,
\end{displaymath}
whenever $t \ge t' \ge T_0$. Since $\alpha>0$ is an arbitrarily small number, this implies that $ \int Q_{c(t)}^2(x)dx $ has a limit as $t$ tends to $+\infty$. 
Then, since $\int Q_{c(t)}^2(x)dx=c(t)^{\frac{2}{p-1}-\frac{d}2}\int Q^2(x)dx$ and we are in the case $p=d=2$, which is subcritical, we conclude that there exists $c_+>0$ such that 
\begin{equation} \label{AsymptStab.9b}
|c_+-1| \le K_0\epsilon \quad \text{and} \quad c(t) \underset{t \to +\infty}{\longrightarrow} c_+ \, . 
\end{equation}

Now, we improve the convergence result following the arguments in the proof of Proposition 3 in \cite{MM5}. Fix $0<\beta<1$. Observe that $u(-x,-t)$ is still a solution of \eqref{ZK} satisfying \eqref{AsymptStab.1}. Then, we deduce applying \eqref{nlL2monotonicity.2}, \eqref{nlEnergymonotonicity.2}, Remarks \ref{remarkL2monotonicity}, \ref{remarkEnergymonotonicity} to $u(-x,-t)$, using the property $\psi_M(-x)=1-\psi_M(x)$ and the conservation of the $L^2$-norm that 
\begin{equation} \label{AsymptStab.11b}
\begin{split}
\int \big(|\nabla u|^2&+u^2\big)(x,t_2)\psi_M\big(x_1-\rho_1(t_1)-\frac{\beta}2(t_2-t_1)+y_0\big)dx
\\ &\le \int \big(|\nabla u|^2+u^2\big)(x,t_1)\psi_M\big(x_1-\rho_1(t_1)+y_0\big)dx+K_1e^{-y_0/M} \, ,
\end{split}
\end{equation}
for all $t_1 \le t_2$ and $y_0>0$ if $M=M(\beta)$ is chosen big enough. Moreover, we observe from the third orthogonality condition in \eqref{modulation.4} that 
\begin{equation} \label{AsymptStab.12b} 
\begin{split}
\Big|\int \tilde{\eta}&(x+\rho(t),t)Q_{c(t)}(x)\psi_M(x_1+y_0)dx \Big| 
\\&=\Big|\int \tilde{\eta}(x+\rho(t),t)Q_{c(t)}(x)\big(1-\psi_M(x_1+y_0)\big)dx \Big|
 \lesssim e^{-y_0/2M} \, ,
\end{split}
\end{equation}
if $M$ is chosen large enough such that $M \ge \frac1{2\delta}$, where $\delta$ is the positive constant appearing in \eqref{elliptPDE1}, and $\epsilon_0>$ is chosen small enough. Therefore by using \eqref{AsymptStab.11b}, \eqref{AsymptStab.12b} and the decomposition
$\tilde{\eta}^2(x,t)=u^2(x,t)-2\tilde{\eta}(x,t)Q_{c(t)}(x-\rho(t))-Q_{c(t)}^2(x-\rho(t))$, we deduce that
\begin{equation} \label{AsymptStab.13b}
\begin{split}
&\int \tilde{\eta}^2(x,t_2)\psi_M\big(x_1-\rho_1(t_1)-\frac{\beta}2(t_2-t_1)+y_0\big)dx
\\ &\le \int \tilde{\eta}^2(x,t_1)\psi_M\big(x_1-\rho_1(t_1)+y_0\big)dx+\|Q\|_{L^2}^2\big|c(t_2)-c(t_1) \big|+Ke^{-y_0/M} \, .
\end{split}
\end{equation}
For $t>0$ large enough, we define $0<t'<t$ such that $\rho_1(t')+\frac{\beta}2(t-t')-y_0=\beta t$. Then, observe that $t \to +\infty$ as $t' \to +\infty$. Moreover, we deduce applying \eqref{AsymptStab.13b} with $t_2=t$ and $t_1=t'$ that 
\begin{displaymath} 
\begin{split}
\int \tilde{\eta}^2(x,t)\psi_M\big(x_1-\beta t\big)dx
&\le \int \tilde{\eta}^2(x,t')\psi_M\big(x_1-\rho_1(t')+y_0\big)dx 
\\ & \quad +\|Q\|_{L^2}^2\big|c(t)-c(t') \big|+Ke^{-y_0/M} \, .
\end{split}
\end{displaymath}
Thus, it follows from \eqref{AsymptStab.10b} and \eqref{AsymptStab.9b} that 
\begin{displaymath} 
\limsup_{t \to +\infty} \int \tilde{\eta}^2(x,t)\psi_M\big(x_1-\beta t\big)dx \le Ke^{-y_0/M} \, , 
\end{displaymath}
for any $y_0>0$, which yields 
\begin{equation} \label{AsymptStab.14b}
\int \tilde{\eta}^2(x,t)\psi_M\big(x_1-\beta t\big)dx \underset{t \to +\infty}{\longrightarrow} 0 \, .
\end{equation}
Arguing similarly, we get that
\begin{equation} \label{AsymptStab.15b}
\int | \nabla\tilde{\eta}|^2(x,t)\psi_M\big(x_1-\beta t\big)dx \underset{t \to +\infty}{\longrightarrow} 0 \, .
\end{equation}
Therefore, we conclude the proof of \eqref{AsymptStab.2} gathering \eqref{AsymptStab.9b}, \eqref{AsymptStab.14b} and \eqref{AsymptStab.15b}.

Finally, we prove \eqref{AsymptStab.200} gathering \eqref{elliptPDE1}, \eqref{modulation.5}, 
\eqref{AsymptStab.5b} and \eqref{AsymptStab.9b}.
\end{proof}
 
 Finally, it remains to prove Proposition \ref{prop.AsymptStab}. \\
 
\noindent \textit{Proof of Proposition \ref{prop.AsymptStab}.}
 By applying Lemma \ref{modulation}, there exist $c \in C^1(\mathbb R : \mathbb R)$ and $\rho=(\rho_1,\rho_2) \in C^1(\mathbb R : \mathbb R^2)$ such that 
 \begin{displaymath}
 \|u(\cdot+\rho(t),t)-Q\|_{H^1} \le \|u(\cdot+\rho(t),t)-Q_{c(t)}\|_{H^1}+\|Q-Q_{c(t)}\|_{H^1} \le 2K_0\epsilon \, ,
 \end{displaymath}
 if $\epsilon_0$ is small enough. 
 
 Let $\{t_n\}$ be a sequence such that $t_n \to +\infty$. Since $\{ u(\cdot+\rho(t_n),t_n)\}$ is bounded in $H^1$, there exist a subsequence extracted from $\{t_n\}$ (still denoted by $\{t_n\}$) and $\tilde{u}_0 \in H^1(\mathbb R^2)$ such that 
 \begin{equation} \label{prop.AsymptStab.4}
 u(\cdot+\rho(t_n),t_n) \underset{n \to +\infty}{\overset{w}{\longrightarrow}} \tilde{u}_0 \quad \text{in} \quad H^1(\mathbb R^2)\, .
 \end{equation}
 Moreover, 
 \begin{displaymath}
 \|\tilde{u}_0-Q\|_{H^1}\le \liminf_{n \to +\infty}\| u(\cdot+\rho(t_n),t_n)-Q\|_{H^1} \le 2K_0\epsilon \, .
 \end{displaymath}
 Let $\tilde{u}$ be the solution of \eqref{ZK} corresponding to $\tilde{u}(\cdot,0)=\tilde{u}_0$. Note from the global well-posedness result for ZK in $H^1$ in \cite{Fa} that $\tilde{u} \in C(\mathbb R : H^1(\mathbb R^2))$. Thus, de Bouard' stability result in \cite{deBo} implies that 
 \begin{equation} \label{AsymptStab.5}
 \sup_{t \in \mathbb R}\|\tilde{u}(\cdot+\tilde{\rho}(t),t)-Q\|_{H^1} \le \tilde{K}_0 \epsilon, 
 \end{equation}
 where $\tilde{\rho}$ is the corresponding modulation function defined in Lemma \ref{modulation}. We split the proof of Proposition \ref{prop.AsymptStab} into several lemmas. \\ 
 
 \noindent \underline{Step 1}: \textit{monotonicity properties for $u$ on the right.} Recall the definition of $\psi_M$ in \eqref{psi0}--\eqref{psi} and fix $M \ge 4$. First, we deduce from Lemmas \ref{nlL2monotonicity} and \ref{nlEnergymonotonicity} a monotonicity property in the $x_1$ direction. 
 \begin{lemma} \label{Asymptmono} 
Let $M \ge 4$. Then, we have that
 \begin{equation} \label{AsymptStab.6} 
 \limsup_{t \to +\infty} \int\big(u^2+|\nabla u|^2 \big)(x+\rho(t),t) \psi_M(x_1-y_0)dx \lesssim e^{-y_0/M} \, ,
 \end{equation}
 for every $y_0>0$. 
 \end{lemma}
 
\begin{proof} Given $\varepsilon>0$, there exists $R_1>0$ such that $\int_{x_1 > R_1} u_0^2dx \le \varepsilon$. Thus, with the notations of Lemma \ref{nlL2monotonicity}, we have that
 \begin{equation} \label{AsymptStab.7} 
 I_{y_0,t}(0) \le \varepsilon +\psi_M(R_1-\rho_1(t)+\frac12t-y_0)\int u_0^2dx \, .
 \end{equation}
 Now, observe from \eqref{modulation.5} that
$ \lim_{t \to +\infty} \psi_M(R_1-\rho_1(t)+\frac12t-y_0)=0 \, ,$ 
which together with \eqref{AsymptStab.7} yields 
 \begin{equation} \label{AsymptStab.8} 
 \limsup_{t \to +\infty} I_{y_0,t}(0)=0 \, . 
 \end{equation}
 Arguing similarly we get that 
 \begin{equation} \label{AsymptStab.9} 
 \limsup_{t \to +\infty} J_{y_0,t}(0)=0 \, ,
 \end{equation}
 where $J_{y_0,t}$ is defined in Lemma \ref{nlEnergymonotonicity}. 
 
 Therefore, we conclude the proof of estimate \eqref{AsymptStab.6} gathering \eqref{nlL2monotonicity.1}, \eqref{nlEnergymonotonicity.1}, \eqref{AsymptStab.8}, \eqref{AsymptStab.9} and using the same arguments as in \eqref{nlL2monotonicity.5}--\eqref{nlL2monotonicity.7} with $\psi_M$ instead of $\psi_M'$.
 \end{proof}
 
We will also need to derive a monotonicity property along the lines in a cone around $x_1$ in order to recover the strong convergence on the right in $L^2$. 

Recall the definition of $\psi_M$ in \eqref{psi0}--\eqref{psi} and fix $M \ge 4$. For $y_0 >0$, $t_0 \in \mathbb R$, and $\theta_0 \in (-\frac{\pi}3,\frac{\pi}3)$ and $t \le t_0$, we define
\begin{equation} \label{AsymptMonotonicity.1} 
I_{y_0,t_0,\theta_0}(t)=\int u^2(x,t)\psi_M\big(x_1+(\tan \theta_0) \, x_2-\rho_1(t_0)+\frac12(t_0-t)-y_0\big)dx \, .
\end{equation}

\begin{lemma} \label{AsymptMonotonicity} 
Let $\theta_0 \in (-\frac{\pi}3,\frac{\pi}3)$. Assume that $\epsilon_0$ is chosen small enough. Then, it holds that
\begin{equation} \label{AsymptMonotonicity.2}
I_{y_0,t_0,\theta_0}(t_0)-I_{y_0,t_0,\theta_0}(t) \lesssim e^{-y_0/M},
\end{equation}
for every $y_0>0$, $t_0 \in \mathbb R$ and $t \le t_0$, and 
\begin{equation} \label{AsymptMonotonicity.2b} 
 \limsup_{t \to +\infty} \int u^2(x+\rho(t),t) \psi_M(x_1+(\tan\theta_0)\, x_2-y_0)dx \lesssim e^{-y_0/M} \, ,
 \end{equation}
 for every $y_0>0$. 
\end{lemma}

\begin{remark}\label{2R_}
Note that the angle $\frac \pi 3$ in the previous result determines the validity of Theorem \ref{AsymptStab} as expanded in Remark \ref{1R}. We believe that this result is sharp for the strong energy norm. For a formal proof of this fact, see Appendix \ref{CC}.  We also expect that this a general result in any dimensions.
\end{remark}

\begin{proof} We briefly sketch the proof of \eqref{AsymptMonotonicity.2}, since it is very similar to the one of Lemma \ref{nlL2monotonicity}.  Define $\tilde{x}=x_1+\tan (\theta) x_2-\rho(t_0)+\frac12(t_0-t)-y_0$. Then, we compute by using \eqref{ZK} that, for all $t \le t_0$,
\begin{equation}   \label{AsymptMonotonicity.3}
\begin{split} 
\frac{d}{dt}I_{y_0,t_0,\theta_0}(t) & =-2\int u\partial_{x_1}^3u\psi_M(\tilde{x})dx-2\int u\partial_{x_1x_2x_2}^3u\psi_M(\tilde{x})dx
\\ & \quad    +\frac23\int u^3\psi_M'(\tilde{x})dx-\frac12\int u^2\psi'_M(\tilde{x})dx \, .
\end{split}
\end{equation}
We observe integrating by parts that 
\begin{displaymath}  
-2\int u\partial_{x_1}^3u\psi_M(\tilde{x})dx=-3\int \big( \partial_{x_1}u\big)^2\psi_M'(\tilde{x})dx+\int u^2\psi_M'''(\tilde{x})dx \, .
\end{displaymath}
and
\begin{displaymath}  
\begin{split}
-2\int u\partial_{x_1x_2x_2}^3u\psi_M(\tilde{x})dx&=-\int \big( \partial_{x_2}u\big)^2\psi'_M(\tilde{x})dx
-2\tan \theta_0\int\partial_{x_1}u\partial_{x_2}u\psi_M'(\tilde{x})dx \\ &\quad +\tan^2\theta_0\int u^2\psi_M'''(\tilde{x})dx \, .
\end{split}
\end{displaymath}
Now, we get applying Cauchy-Schwarz and Young's inequalities that 
\begin{displaymath} 
2\big| \tan\theta_0\int \partial_{x_1}u\partial_{x_2}u \psi'_M(\tilde{x})dx\big| \le \int \big(\kappa_0^2\tan^2 \theta_0 \, (\partial_{x_1}u)^2+\frac1{\kappa_0^2}(\partial_{x_2}u)^2 \big)(x,t)\psi_M'(\tilde{x})dx\, ,
\end{displaymath}
for some $\kappa_0>0$ satisfying $1<\kappa_0^2<\frac{3}{\tan^2\theta_0}$, which is possible since $|\theta_0| < \frac{\pi}3$. 
Therefore, we deduce from \eqref{psi} that
\begin{equation}   \label{AsymptMonotonicity.4}
\begin{split}
\frac{d}{dt}I_{y_0,t_0,\theta_0}(t) & \le 
-\int \big(\frac18u^2+(3-\kappa_0^2\tan^2 \theta_0)(\partial_{x_1}u)^2+(1-\frac1{\kappa_0^2})(\partial_{x_2}u)^2\big)(x,t)\psi_M'(\tilde{x})dx \\ & \quad+\frac23\int u^3\psi_M'(\tilde{x})dx
\end{split}
\end{equation}
for all $t \le t_0$.

To handle the second term on the right-hand side of \eqref{AsymptMonotonicity.4}, we 
argue exactly as in \eqref{nlL2monotonicity.5}--\eqref{nlL2monotonicity.9} and deduce that 
\begin{equation} \label{AsymptMonotonicity.5}
\begin{split}
\frac23\Big|\int u^3\psi_M'(\tilde{x})dx\Big|&\le K_0\epsilon_0\int \big(u^2+|\nabla u|^2(x,t)\psi_M'(\tilde{x})dx +C_{\epsilon_0}e^{-(\frac14(t_0-t)+y_0)/M} \, ,
\end{split}
\end{equation}
where $C_{\epsilon_0}$ is a positive constant depending on $\epsilon_0$.

Hence, we conclude the proof of \eqref{AsymptMonotonicity.2} gathering \eqref{AsymptMonotonicity.4} and \eqref{AsymptMonotonicity.5}, integrating between $t$ and $t_0$ and choosing $\epsilon_0$ small enough.

The proof of \eqref{AsymptMonotonicity.2b} follows exactly as the one of \eqref{AsymptStab.6}.
\end{proof}

 \noindent \underline{Step 2}: \textit{Strong $L^2$-convergence of $u(\cdot+\rho(t_n),t_n)$ to $\tilde{u}_0$ on the right.} 

 \begin{lemma} \label{L2CV} We have that
\begin{equation} \label{L2CV.1}
u(\cdot+\rho(t_n),t_n) \underset{n \to +\infty}{\longrightarrow} \tilde{u}_0 \quad \text{in} \quad L^2(x_1>-A) \, ,
\end{equation}
for all $A>0$. 
\end{lemma}

\begin{proof} Let $A>0$ and $\epsilon>0$. From \eqref{AsymptMonotonicity.2b}, there exists $R_{\frac{\pi}4}>0$ such that 
\begin{equation} \label{L2CV.2}
\|\tilde{u}_0\|_{L^2(x_1+x_2>R_{\frac{\pi}4})}+\|\tilde{u}_0\|_{L^2(x_1- x_2>R_{\frac{\pi}4})} \le \epsilon
\end{equation} 
and
\begin{equation} \label{L2CV.3}
\limsup_{n \to +\infty}\Big(\|u(\cdot+\rho(t_n),t_n)\|_{L^2(x_1+ x_2>R_{\frac{\pi}4})} +\|u(\cdot+\rho(t_n),t_n)\|_{L^2(x_1- x_2>R_{\frac{\pi}4})}\Big) \le \epsilon \, .
\end{equation} 

Let us denote by $\mathcal{R}$ the compact region of $\mathbb R^2$ (see Fig. \ref{fig3}) defined by $$\mathcal{R}=\big\{(x_1,x_2) \in \mathbb R^2 : x_1\ge -A, \ x_1+ x_2 \le R_{\frac{\pi}4}, \ x_1- x_2 \le R_{\frac{\pi}4} \big\} \, .$$

Since the embedding $H^1(\mathcal{R}) \hookrightarrow L^2(\mathcal{R})$ is compact, we deduce from \eqref{prop.AsymptStab.4} that 
\begin{equation} \label{L2CV.4}
\lim_{n \to +\infty}\|u(\cdot+\rho(t_n),t_n)-\tilde{u}_0\|_{L^2(\mathcal{R})}=0 \, .
\end{equation}

Therefore, we conclude the proof of \eqref{L2CV.1} gathering \eqref{L2CV.2}--\eqref{L2CV.4} and using the triangle inequality.
\end{proof} 

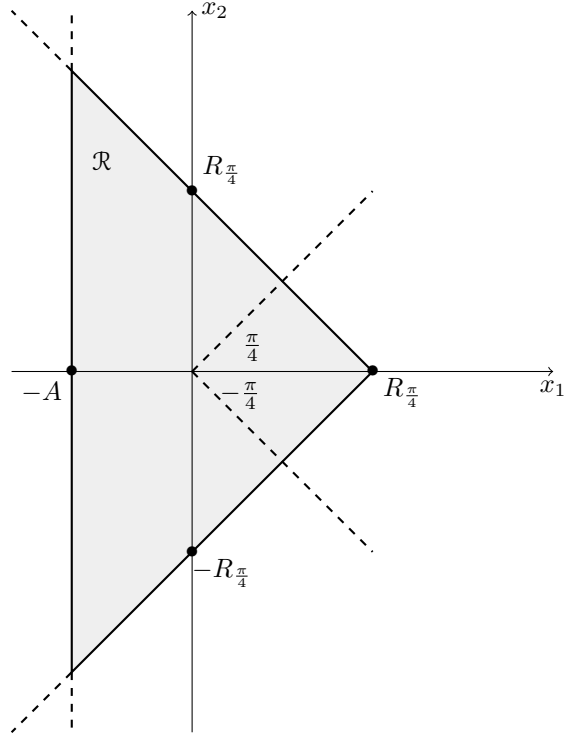
\begin{figure} 
\begin{center}
\begin{tikzpicture}[scale=0.8] 
%\draw[very thin,color=gray] (-0.1,-1.1) grid (6.4,3.9);
%\draw[color=blue, thick] plot (\x, -0.21*\x*\x + 0.92*\x +2) node at %(5.3,3.3) 
%{\small{$g(t)= -\frac{|E_0|}{2}t^2 + t\,\text{Im}\displaystyle\int %x\cdot \nabla u_0\,dx +h(0)$}};
\filldraw[thick, color=lightgray!25] (3,0) -- (-2,5) -- (-2,-5) -- (3,0);
\draw[thick, color=black] (3,0) -- (-2,5) -- (-2,-5) -- (3,0);
\draw[thick, dashed] (-2,5) -- (-3,6);
\draw[thick,dashed] (-2,-5)-- (-3,-6);
\draw[thick,dashed] (-2,5)--(-2,6);
\draw[thick,dashed] (-2,-5)--(-2,-6);
\draw[thick,dashed] (-0,0)--(3,3);
\draw[thick,dashed] (0,0)--(3,-3);
\draw[->] (-3,0) -- (6,0) node[below] {$x_1$};
\draw[->] (0,-6) -- (0,6) node[right] {$x_2$};
\node at (-2,0){$\bullet$};
\node at (-2.5,-0.35){$-A$};
\node at (3,0){$\bullet$};
\node at (3.5,-0.35){$R_{\frac{\pi}{4}}$};
\node at (0,3){$\bullet$};
\node at (0.5,3.35){$R_{\frac{\pi}{4}}$};
\node at (0,-3){$\bullet$};
\node at (0.5,-3.35){$-R_{\frac{\pi}{4}}$};
\node at (-1.5,3.5){$\mathcal{R}$};
\node at (1,0.4){$\frac{\pi}4$};
\node at (0.8,-0.4){$-\frac{\pi}4$};
\end{tikzpicture}
\end{center}
\caption{$\mathcal{R}$ is a compact set of $\mathbb R^2$.}\label{fig3}
\end{figure}

 \noindent \underline{Step 3}: \textit{Exponential decay of $\tilde{u}$ on the right on finite time intervals.} 
 
 \begin{lemma} \label{exptildeu}
Let $M \ge 4$ be given. Then, 
\begin{equation} \label{exptildeu.1}
\int \big(\tilde{u}_0^2+|\nabla \tilde{u}_0|^2 \big)(x)\psi_M(x_1-y_0)dx \lesssim e^{-y_0/M} \, ,
\end{equation}
for all $y_0>0$. 

Moreover, for all $t_0 \ge 0$, there exists $K(t_0)>0$ such that
\begin{equation} \label{exptildeu.2}
\sup_{t \in [0,t_0]} \int \big(\tilde{u}^2+|\nabla \tilde{u}|^2  \big)(x,t) e^{x_1/M}dx \le K(t_0) \, .
\end{equation}
 \end{lemma} 
  
 \begin{proof} Observe from \eqref{prop.AsymptStab.4} that $u(\cdot+\rho(t_n),t_n) \sqrt{\psi_M(\cdot_1-y_0)} \overset{w}{\longrightarrow} \tilde{u}_0\sqrt{\psi_M(\cdot_1-y_0)}$ in $H^1(\mathbb R^2)$. Thus, 
 \begin{displaymath} 
 \| \tilde{u}_0 \sqrt{\psi_M(\cdot_1-y_0)} \|_{H^1} \le \liminf_{n \to +\infty} \| u(\cdot+\rho(t_n),t_n) \sqrt{\psi_M(\cdot_1-y_0)} \|_{H^1} \, ,
 \end{displaymath}
 which combined with \eqref{AsymptStab.6} yields \eqref{exptildeu.1}. 
 
 Now, we turn to the proof of \eqref{exptildeu.2}.  Fix $t_0>0$ and $y_0>0$. Since $\tilde{u}$ is a solution to \eqref{ZK}, we obtain after some integration by parts that 
 \begin{equation} \label{exptildeu.3}
 \begin{split}
 \frac{d}{dt}\int &\tilde{u}^2(x,t)\psi_{M}(x_1-y_0)dx \\ &\le -\int \Big(3(\partial_{x_1}\tilde{u})^2+(\partial_{x_2}\tilde{u})^2 \Big)(x,t)\psi_M'(x_1-y_0)dx \\ 
 &\quad +\int \tilde{u}^2(x,t)\psi_M(x_1-y_0)dx+\frac23\int \tilde{u}^3(x,t)\psi_M'(x_1-y_0)dx \, .
 \end{split} 
 \end{equation} 
 To deal with the last term on the right-hand side of the above expression, we use the decomposition $\tilde{u}^3=Q(\cdot-\tilde{\rho}(t))\tilde{u}^2+\big(\tilde{u}-Q(\cdot-\tilde{\rho}(t))\big)\tilde{u}^2$, the Sobolev embedding $H^1(\mathbb R^2) \hookrightarrow L^3(\mathbb R^2)$ and estimate \eqref{AsymptStab.5}. Thus, 
 \begin{equation} \label{exptildeu.4}
\begin{split}
\frac23\int \tilde{u}^3&(x,t)\psi_M'(x_1-y_0)dx \\ &\lesssim  K_0\epsilon_0 \| \tilde{u}\sqrt{\psi_M'(\cdot_1-y_0)}\|_{H^1}+\int \tilde{u}^2\psi_M'(x_1-y_0)dx \, .
\end{split}
 \end{equation}
 which implies that 
 \begin{displaymath} 
 \frac{d}{dt}\int \tilde{u}^2(x,t)\psi_{M}(x_1-y_0)dx \lesssim \int \tilde{u}^2(x,t)\psi_M(x_1-y_0)dx \, ,
 \end{displaymath}
by choosing $\epsilon_0$ small enough. Hence, Gronwall's inequality and \eqref{exptildeu.1} yields 
 \begin{equation} \label{exptildeu.5}
\sup_{t \in [0,t_0]} \int \tilde{u}^2(x,t)\psi_{M}(x_1-y_0)dx \le K(t_0) e^{-y_0/M} \, ,
 \end{equation}
 for some positive constant $K(t_0)$ depending only on $t_0$.
 
 Now, by using \eqref{ZK} and integrating by parts as in \eqref{nlEnergymonotonicity.4}--\eqref{nlEnergymonotonicity.5}, we get that 
 \begin{displaymath}
 \begin{split}
 \frac{d}{dt} \int &\big(|\nabla \tilde{u}|^2-\frac23\tilde{u}^3 \big) \psi_M(x_1-y_0)dx
 \\ &=-\int \Big( 3(\partial_{x_1}^2\tilde{u})^2+4(\partial_{x_1x_2}^2\tilde{u})^2+(\partial_{x_2}^2\tilde{u})^2+ \tilde{u}^4\Big)(x,t)\psi_M'(x_1-y_0)dx \\ & \quad +\int \big(|\nabla \tilde{u}|^2-\frac23\tilde{u}^3\big)(x,t) \psi_M'''(x_1-y_0)dx \\ & \quad 
 +4\int \tilde{u} \big(2(\partial_{x_1}\tilde{u})^2+(\partial_{x_2}\tilde{u})^2\big)(x,t)\psi_M(x_1-y_0)dx \, .
 \end{split}
 \end{displaymath}
Thus, it follows using \eqref{psi} and arguing as in \eqref{exptildeu.4} that 
\begin{displaymath} 
 \begin{split}
 \frac{d}{dt} \int \big(|\nabla \tilde{u}|^2&-\frac23\tilde{u}^3 \big) \psi_M(x_1-y_0)dx
 \\ &\lesssim \int|\nabla \tilde{u}|^2\psi_M'(x_1-y_0)dx +\int \tilde{u}^2\psi_M(x_1-y_0)dx \, ,
 \end{split}
\end{displaymath}
if $\epsilon_0$ is choosen small enough. Hence, we infer from \eqref{exptildeu.3} that 
\begin{equation} \label{exptildeu.6}
 \frac{d}{dt} \int \big(|\nabla \tilde{u}|^2-\frac23\tilde{u}^3+K_1\tilde{u}^2 \big) \psi_M(x_1-y_0)dx
 \lesssim \int \tilde{u}^2\psi_M(x_1-y_0)dx  \, ,
\end{equation}
for some positive constant $K_1$. Therefore, we conclude integrating \eqref{exptildeu.6} between $0$ and $t$ and using \eqref{exptildeu.4}--\eqref{exptildeu.5} that 
\begin{displaymath} 
\sup_{t \in [0,t_0]} \int \big( \tilde{u}^2+|\nabla \tilde{u}|^2\big)(x,t)\psi_{M}(x_1-y_0)dx \le \tilde{K}(t_0) e^{-y_0/M} \, ,
\end{displaymath}
which yields \eqref{exptildeu.2} since $\psi_M(x_1-y_0) \gtrsim e^{(x_1-y_0)/M}$ for $x_1-y_0 \le 0$.
 \end{proof}
 
  \noindent \underline{Step 4}: \textit{Strong $L^2$-convergence of $u(\cdot+\rho(t_n),t_n+t)$ to $\tilde{u}(\cdot,t)$ on the right.} 
 
 \begin{lemma} \label{L2bCV} We have that
\begin{equation} \label{L2bCV.1}
u(\cdot+\rho(t_n),t_n+t) \underset{n \to +\infty}{\longrightarrow} \tilde{u}(\cdot,t) \quad \text{in} \quad L^2(x_1>-A) \, ,
\end{equation}
for all $A>0$, $t \in \mathbb R$ and 
 \begin{equation} \label{L2bCV.1b}
 u(\cdot+\rho(t_n),t_n+t) \underset{n \to +\infty}{\overset{w}{\longrightarrow}} \tilde{u}(\cdot,t) \quad \text{in} \quad  H^1(\mathbb R^2)\, ,
 \end{equation}
for all $t \in \mathbb R$. Moreover, 
\begin{equation} \label{L2bCV.2}
\rho(t_n+t)-\rho(t_n) \longrightarrow \tilde{\rho}(t) \quad \text{and} \quad \tilde{\rho}(0)=0 \, ,
\end{equation}
for all $t \in \mathbb R$, where $\tilde{\rho}$ is the $C^1$-function associated to the decomposition of $\tilde{u}$ in Lemma \ref{modulation}.
\end{lemma} 

\begin{proof} Let us define $v_n(x,t)=u(x+\rho(t_n),t_n+t)-\tilde{u}(x,t)$. Then, it follows from \eqref{L2CV.1} that 
\begin{equation} \label{L2bCV.3}
\int v_n(x,0)^2\psi_M(x_1)dx \underset{n \to +\infty}{\longrightarrow} 0 \, .
\end{equation}
Moreover an easy computation using \eqref{ZK} shows that 
\begin{equation} \label{L2bCV.4}
\partial_t v_n+\partial_{x_1}\Delta v_n+\partial_{x_1}\big(2\tilde{u}v_n+v_n^2 \big)=0 \, . 
\end{equation}

Next, we infer that for all $t_0>0$ and $M \ge 8$, there exists $K_1(t_0)>0$ such that 
\begin{equation} \label{L2bCV.5}
\sup_{t \in [0,t_0]} \int v_n(x,t)^2\psi_M(x_1)dx \le K_1(t_0) \int v_n(x,0)^2 \psi_M(x_1)dx \, ,
\end{equation}
which together with \eqref{L2bCV.3} yields \eqref{L2bCV.1} for any $t \ge 0$. The weak convergence  in $\eqref{L2bCV.1b}$ for any $t \ge 0$ follows then by uniqueness of the weak limit in $H^1$.

Now, we prove \eqref{L2bCV.5}. Fix $t_0>0$ and $M \ge 8$. By using \eqref{L2bCV.4} and arguing as previously, we get that
\begin{equation} \label{L2bCV.6}
\begin{split}
\frac{d}{dt} \int v_n^2 \psi_M(x_1)dx &= 
-\int \big(3(\partial_{x_1}v_n)^2+(\partial_{x_2}v_n)^2 \big)\psi_M'(x_1)dx 
 +\int v_n^2 \psi_M'''(x_1)dx \\ & \quad
+2\int \big(2\tilde{u}v_n+v_n^2 \big)\partial_{x_1}\big(v_n\psi_M(x_1) \big)dx \, ,
\end{split}
\end{equation}
for all $t \in [0,t_0]$.
Observe that the last term on the right-hand side of the above formula can be rewritten as 
\begin{equation} \label{L2bCV.7}
\begin{split}
\int \big(2\tilde{u}&v_n+v_n^2 \big)\partial_{x_1}\big(v_n\psi_M(x_1) \big)dx \\ & =
\int \big(\frac23 v_n+\tilde{u}\big)v_n^2\psi_M'(x_1)dx-\int \partial_{x_1}\tilde{u}v_n^2\psi_M(x_1)dx \, .
\end{split}
\end{equation}
By using the Gagliardo-Nirenberg inequality $\|f\|_{L^4} \lesssim \|f\|_{L^2}^{\frac12}\|f\|_{H^1}^{\frac12}$ in two dimensions, we get that 
\begin{equation} \label{L2bCV.8}
\begin{split}
\Big|\int \big(\frac23 v_n+\tilde{u}\big)v_n^2\psi_M'(x_1)dx\Big| &\lesssim \big(\|v_n\|_{L^2}+\|\tilde{u}\|_{L^2} \big)\|v_n\sqrt{\psi_M'(x_1)}\|_{L^2}\|v_n\sqrt{\psi_M'(x_1)}\|_{H^1} \\ & \le
K\int v_n^2\psi_M'(x_1)dx+\frac1{32}\int |\nabla v_n|^2\psi_M'(x_1)dx \, ,
\end{split}
\end{equation}
where $K$ is a positive constant depending on $\|u_0\|_{L^2}$ and $\|\tilde{u}_0\|_{L^2}$. To estimate the second term on the right-hand side of \eqref{L2bCV.7}, we observe arguing as above that 
\begin{displaymath}
\Big|\int \partial_{x_1}\tilde{u}v_n^2\psi_M(x_1)dx \Big| \lesssim\big\| \partial_{x_1}\tilde{u} \frac{\psi_M(x_1)}{\psi_M'(x_1)}\big\|_{L^2}\|v_n\sqrt{\psi_M'(x_1)}\|_{L^2}\|v_n\sqrt{\psi_M'(x_1)}\|_{H^1} \, .
\end{displaymath}
Now, since 
\begin{displaymath}
\frac{\psi_M(x_1)^2}{\psi_M'(x_1)^2}  \lesssim 1, \ \text{for} \ x_1 \le 0 \quad \text{and} \quad \frac{\psi_M(x_1)^2}{\psi_M'(x_1)^2} \lesssim e^{2x_1/M},  \ \text{for} \ x_1 \ge 0 \, ,
\end{displaymath}
we deduce from \eqref{exptildeu.2} that 
\begin{equation} \label{L2bCV.9}
\Big|\int \partial_{x_1}\tilde{u}v_n^2\psi_M(x_1)dx \Big| \le \tilde{K}(t_0)\int v_n^2\psi_M'(x_1)dx+\frac1{32}\int |\nabla v_n|^2\psi_M'(x_1)dx \, ,
\end{equation}
where $\tilde{K}(t_0)$ is a positive constant depending on $\|\tilde{u}_0\|_{H^1}$ and $K(t_0)$. Hence,  it follows from \eqref{psi} and \eqref{L2bCV.6}--\eqref{L2bCV.9} that 
\begin{equation} \label{L2bCV.9b} 
\frac{d}{dt} \int v_n^2(x,t) \psi_M(x_1)dx +
\int |\nabla v_n|^2(x,t)\psi_M'(x_1)dx \le \tilde{K}(t_0) \int v_n^2(x,t) \psi_M(x_1)dx \, ,
\end{equation}
for all $0 \le t \le t_0$, which together with Gronwall's inequality implies \eqref{L2bCV.5}. 

To prove that \eqref{L2bCV.1} and \eqref{L2bCV.1b} also hold for $t \le 0$, we fix some $\tilde{t}_1 <0$. Since $\{u(\cdot+\rho(t_n),t_n+\tilde{t}_1)\}$ is bounded in $H^1(\mathbb R^2)$, there exists a subsequence extracted from $\{t_n\}$ (still denoted $\{t_n\}$) and $\tilde{u}_{1,0} \in H^1(\mathbb R^2)$ such that
\begin{equation} \label{L2bCV.10}
 u(\cdot+\rho(t_n),t_n+\tilde{t}_1) \underset{n \to +\infty}{\overset{w}{\longrightarrow}} \tilde{u}_{1,0} \quad \text{in} \ H^1\, .
 \end{equation}
Let $\tilde{u}_1 \in C(\mathbb R : H^1(\mathbb R^2))$ be the solution of \eqref{ZK} satisfying $\tilde{u}_1(\cdot,0)=\tilde{u}_{1,0}$. By reproducing the above analysis on $\tilde{u}_1$, we obtain that 
\begin{equation} \label{L2bCV.11} 
\begin{split}
& u(\cdot+\rho(t_n),t_n+\tilde{t}_1+t) \underset{n \to +\infty}{\overset{w}{\longrightarrow}} \tilde{u}_{1}(\cdot,t) \quad \text{in} \quad H^1(\mathbb R^2) \, , \\
& u(\cdot+\rho(t_n),t_n+\tilde{t}_1+t) \underset{n \to +\infty}{\longrightarrow} \tilde{u}_{1}(\cdot,t) \quad \text{in} \quad  L^2(x_1>-A) \, ,
\end{split}
\end{equation}
for all $A>0$ and $t \ge 0$. In particular, we deduce from \eqref{L2bCV.11} with $t=-\tilde{t_1}$ and \eqref{prop.AsymptStab.4} that $\tilde{u}_1(\cdot,-\tilde{t}_1)=\tilde{u}_0$. Thus the uniqueness of the Cauchy problem associated to \eqref{ZK} in $H^1(\mathbb R^2)$ implies that $\tilde{u}_1(\cdot,t-\tilde{t}_1)=\tilde{u}(\cdot,t)$, for all $t \in \mathbb R$. We conclude from \eqref{L2bCV.11} that \eqref{L2bCV.1}--\eqref{L2bCV.1b} still hold true for $t \ge \tilde{t}_1$ and then for $t \in \mathbb R$, since $\tilde{t}_1<0$ was chosen arbitrarily.

Finally, we prove \eqref{L2bCV.2}. By extracting another subsequence if necessary, we can assume from \eqref{modulation.3} and \eqref{modulation.5} that for all $t \in \mathbb R$, there exist $d(t) \in \mathbb R$ and $\beta(t) \in \mathbb R^2 $ such that
\begin{equation} \label{L2bCV.12} 
c(t_n+t) \underset{n \to +\infty}{\longrightarrow} d(t) \quad \text{and} \quad \rho(t_n+t)-\rho(t_n) \underset{n \to +\infty}{\longrightarrow} \beta(t) \, .
\end{equation}
Observe that $\beta(0)=0$. Let us define 
\begin{displaymath}
\eta_n(\cdot,t)=u(\cdot+\rho(t_n+t),t_n+t)-Q_{c(t_n+t)} \quad \text{and} \quad 
\tilde{\eta}(\cdot,t)=\tilde{u}(\cdot+\beta(t),t)-Q_{d(t)} \, .
\end{displaymath} 
It follows then from \eqref{modulation.4} that
\begin{displaymath}
\int \eta_n(x,t)Q_{c(t_n+t)}(x)dx=\int \eta_n(x,t)\partial_{x_i}Q_{c(t_n+t)}(x)dx
=0, \ i=1,2 \, .
\end{displaymath}
Therefore, we deduce letting $n \to +\infty$ and using \eqref{L2bCV.1b} and \eqref{L2bCV.12} that
\begin{equation} \label{L2bCV.13}
\int \tilde{\eta}(x,t)Q_{d(t)}(x)dx=\int \tilde{\eta}(x,t)\partial_{x_i}Q_{d(t)}(x)dx
=0, \ i=1,2 \, .
\end{equation}
This implies  by the uniqueness of the  decomposition of $\tilde{u}$ in Lemma \ref{modulation} that $d(t)=\tilde{c}(t)$ and $\beta(t)=\tilde{\rho}(t)$, for all $t \in \mathbb R$, which concludes the proof of \eqref{L2bCV.2}.
\end{proof}

 \noindent \underline{Step 5}: \textit{Exponential decay of $\tilde{u}$ on the right.}
 \begin{lemma} \label{rightexpdecay}
 Let $M \ge 4$. Then, 
 \begin{equation} \label{rightexpdecay.1}
 \int \big( \tilde{u}^2+|\nabla \tilde{u}|^2 \big)(x+\tilde{\rho}(t),t) \psi_M(x_1-y_0)dx \lesssim e^{-y_0/M} \, ,
 \end{equation}
 for all $y_0>0$ and $t \in \mathbb R$. 
 
 Moreover, 
 \begin{equation} \label{rightexpdecay.2}
 \int_{x_2}\tilde{u}^2(x+\tilde{\rho}(t),t)dx_2 \lesssim e^{-x_1/M}, \quad \forall \, (x_1,t) \in \mathbb R^2 \, .
 \end{equation} 
 \end{lemma}
 
 \begin{proof} Observe from \eqref{L2bCV.1b} and \eqref{L2bCV.2} that 
\begin{displaymath}
u(\cdot+\rho(t_n+t),t_n+t) \sqrt{\psi_M(\cdot_1-y_0)} \underset{n \to +\infty}{\overset{w}{\longrightarrow}} \tilde{u}(\cdot+\tilde{\rho}(t),t)\sqrt{\psi_M(\cdot_1-y_0)} \quad \text{in} \ H^1(\mathbb R^2) \, ,
\end{displaymath} 
for all $t \in \mathbb R$. Thus, 
 \begin{displaymath} 
 \| \tilde{u}(\cdot+\tilde{\rho}(t),t) \sqrt{\psi_M(\cdot_1-y_0)} \|_{H^1} \le \liminf_{n \to +\infty} \| u(\cdot+\rho(t_n+t),t_n+t) \sqrt{\psi_M(\cdot_1-y_0)} \|_{H^1} \, ,
 \end{displaymath}
 which combined with \eqref{AsymptStab.6} yields \eqref{rightexpdecay.1}. 
 
 Since $\psi_M(x_1-y_0) \ge e^{x_1-y_0}$ for $x_1-y_0 \le 0$, it follows from \eqref{rightexpdecay.1} that
 \begin{displaymath}
 \int \big( \tilde{u}^2+|\nabla \tilde{u}|^2 \big)(x+\tilde{\rho}(t),t) e^{x_1/M}dx \lesssim 1 \, , 
 \end{displaymath}
 which yields \eqref{rightexpdecay.2} arguing as in \eqref{LinearLiouville8b2}.
 \end{proof}
 
\noindent \underline{Step 6}: \textit{Strong $H^1$-convergence of $u(\cdot+\rho(t_n),t_n+t)$ to $\tilde{u}(\cdot,t)$ on the right.} 
\begin{lemma} \label{H1CV}
We have that
\begin{equation} \label{H1CV.1}
u(\cdot+\rho(t_n),t_n+t) \underset{n \to +\infty}{\longrightarrow} \tilde{u}(\cdot,t) \quad \text{in} \quad H^1(x_1>-A) \, ,
\end{equation}
for all $A>0$, $t \in \mathbb R$.
\end{lemma} 

\begin{proof} Arguing as in the proof of Lemma \ref{L2bCV}, it is enough to prove that \eqref{H1CV.1} holds for $t \ge 0$.

Recall that $v_n(x,t)=u(x+\rho(t_n),t_n+t)-\tilde{u}(x,t)$ satisfies the equation in \eqref{L2bCV.4}. Let $M \ge 12$. We claim that 
\begin{equation} \label{H1CV.2} 
\int |\nabla v_n|^2(x,t) \psi_M(x_1)dx \underset{n \to +\infty}{\longrightarrow} 0 \, ,
\end{equation}
for all $t \ge 0$, which implies \eqref{H1CV.1} together with \eqref{L2bCV.1}. 

To prove \eqref{H1CV.2}, we fix $t_0 >0$. First, we integrate \eqref{L2bCV.9} between $t_0-1$ and $t_0$ and use \eqref{L2bCV.1} to get that 
\begin{equation} \label{H1CV.3}
\int_{t_0-1}^{t_0}\int |\nabla v_n|^2(x,t) \psi_M'(x_1)dxdt \underset{n \to +\infty}{\longrightarrow} 0 \, .
\end{equation}
Thus, we infer that 
\begin{equation} \label{H1CV.4}
\int_{t_0-1}^{t_0}\int |\nabla v_n|^2(x,t) \psi_M(x_1)dxdt \underset{n \to +\infty}{\longrightarrow} 0 \, .
\end{equation}
Indeed, for all $y_0>0$, there exists $C_{y_0}>0$ such that $\psi_M'(x_1) \ge \frac1{C_{y_0}}$ if $x_1 \le y_0+\tilde{\rho}_1(t_0)$, so that 
\begin{equation} \label{H1CV.5}
\begin{split}
\int_{t_0-1}^{t_0}\int &|\nabla v_n|^2(x,t) \psi_M(x_1)dxdt \\&\lesssim C_{y_0}\int_{t_0-1}^{t_0}\int_{x_1 \le y_0+\tilde{\rho}_1(t)} |\nabla v_n|^2(x,t) \psi_M'(x_1)dxdt \\ & \quad
+\int_{t_0-1}^{t_0}\int_{x_1 \ge y_0+\tilde{\rho}_1(t)} |\nabla v_n|^2(x,t) \psi_M(x_1-y_0-\tilde{\rho}_1(t))dxdt \, .
\end{split}
\end{equation}
By using \eqref{AsymptStab.6}, \eqref{rightexpdecay.1} and \eqref{L2bCV.2}, we can make the second term on the right-hand side of \eqref{H1CV.5}arbitrarily small as soon as $y_0$ is chosen large enough. This fact together  with \eqref{H1CV.3} implies \eqref{H1CV.4}.

Now, we claim that 
\begin{equation} \label{H1CV.6}
\begin{split}
\int |\nabla &v_n|^2(x,t_0) \psi_M(x_1)dx \\ & \lesssim_{t_0} \int |\nabla v_n|^2(x,t) \psi_M(x_1)dx
 +\int_{t_0-1}^{t_0}\int |\nabla v_n|^2(x,t') \psi_M(x_1)dxdt'\\ & \quad+\sup_{t'\in [t_0-1,t_0]} \int v_n^2(x,t')\psi_M(x_1)dx \, ,
\end{split}
\end{equation}
for all $t \in [t_0-1,t_0]$, which implies \eqref{H1CV.2} after integrating between $t_0-1$ and $t_0$ and using \eqref{L2bCV.3}, \eqref{L2bCV.5} and \eqref{H1CV.4}.

It remains to prove \eqref{H1CV.6}. Let us define $J_n(t)=\int \big(|\nabla v_n|^2-\frac23 v_n^3 \big)(x,t)\psi_M(x_1)dx.$ It follows from \eqref{L2bCV.4} and after using some integrations by parts that 
\begin{equation} \label{H1CV.7}
\begin{split}
\frac{d}{dt} J_n(t)&=-\int \Big(3(\partial_{x_1}^2v_n)^2+4(\partial^2_{x_1x_2}v_n)^2+(\partial^2_{x_2}v_n)^2 +v_n^4\Big)\psi_M'(x_1)dx 
\\ & \quad +\int \big( |\nabla v_n|^2-\frac23v_n^3\big) \psi_M''(x_1)dx
\\ & \quad +
4\int \big(2(\partial_{x_1}v_n)^2+(\partial_{x_2}v_n)^2\big) \psi_M(x_1)dx 
\\ & \quad +
4\int \big(\partial_{x_1}\tilde{u} v_n\partial^2_{x_1}v_n +\partial_{x_2}\tilde{u} v_n\partial^2_{x_1x_2}
v_n \big)\psi_M(x_1)dx
\\ & \quad +
4\int \big(\partial_{x_1}\tilde{u} v_n\partial_{x_1}v_n +\partial_{x_2}\tilde{u} v_n\partial_{x_2}
v_n \big)\psi_M'(x_1)dx 
\\ & \quad 
-2\int \partial_{x_1}\tilde{u} |\nabla v_n|^2\psi_M(x_1)dx+2\int \tilde{u} |\nabla v_n|^2\psi_M'(x_1)dx 
\\ & \quad +\frac83 \int \partial_{x_1}\tilde{u} v_n^3 \psi_M(x_1)dx-\frac43 \int \tilde{u} v_n^3 \psi'_M(x_1)dx \, .
\end{split}
\end{equation}
Therefore, we deduce arguing as in the proofs of Lemmas \ref{exptildeu} and \eqref{L2bCV} and using \eqref{exptildeu.2}, \eqref{L2bCV.3} and \eqref{L2bCV.5}  that there exists $K_2(t_0)>0$ such that
\begin{equation} \label{H1CV.8}
\begin{split}
\frac{d}{dt} \int \big(|\nabla v_n|^2-\frac23 v_n^3\big)(x,t)\psi_M(x_1)dx  \le K_2(t_0)\int \big( |\nabla v_n|^2+v_n^2 \big)(x,t) \psi_M(x_1)dx \, .
\end{split}
\end{equation}
For example, we explain how to deal with the fourth term appearing on the right-hand side of \eqref{H1CV.7}, which is the most difficult. By using Young and H\"older's inequalities, we get that 
\begin{displaymath}
\begin{split}
\int \partial_{x_1}&\tilde{u} v_n\partial^2_{x_1}v_n \psi_M(x_1)dx \\
&\le \frac1{64}\int (\partial_{x_1}^2v_n)^2 \psi_M'(x_1)dx
+K \int (\partial_{x_1}\tilde{u})^2v_n^2\frac{\psi_M(x_1)^2}{\psi_M'(x_1)}dx
\\ & \le\frac1{64}\int (\partial_{x_1}^2v_n)^2 \psi_M'(x_1)dx
+K\|v_n \sqrt{\psi_M'}\|_{L^{\infty}} \int (\partial_{x_1}\tilde{u})^2\frac{\psi_M(x_1)^2}{\psi_M'(x_1)^2}dx \, .
\end{split}
\end{displaymath}
Observe that $\frac{\psi_M(x_1)^2}{\psi_M'(x_1)^2} \lesssim \big( 1+e^{2x_1/M}\big)$. Therefore, the Gagliardo-Nirenberg inequality 
$$\|v_n \sqrt{\psi_M'(x_1)}\|_{L^{\infty}}^2 \lesssim \|v_n \sqrt{\psi_M'(x_1)}\|_{H^1} \|v_n \sqrt{\psi_M'(x_1)}\|_{H^2} \, ,$$ 
Young's inequality and \eqref{exptildeu.2} give that 
\begin{displaymath}
\begin{split}
\int \partial_{x_1}&\tilde{u} v_n\partial^2_{x_1}v_n \psi_M(x_1)dx \\
&\le \frac1{32}\int | \nabla^2 v_n|^2 \psi_M'(x_1)dx
+K(t_0) \int \big( |\nabla v_n|^2+v_n^2\big) \psi_M'(x_1)dx \, .
\end{split}
\end{displaymath}

Finally, we conclude the proof of \eqref{H1CV.6} integrating \eqref{H1CV.8} between $t$ and $t_0$ and using the estimate
\begin{displaymath} 
\int v_n^3(x,t) \psi_M(x_1)dx \le K\int v_n^2(x,t) \psi_M(x_1)dx+\frac1{32}\int |\nabla v_n|^2(x,t) \psi_M(x_1)dx \, ,
\end{displaymath}
which follows arguing as in \eqref{L2bCV.8}.
\end{proof}

\noindent \underline{Step 7}: \textit{Exponential decay of $\tilde{u}$ on the left.}
Before proving the exponential decay of $\tilde{u}$ on the left, we need to derive another monotonicity property for $u$ in the $x_1$-direction for times moving forward. 
\begin{lemma} \label{tildemonotonicity}
Assume that $u \in C(\mathbb R : H^1(\mathbb R^2))$ is a solution of \eqref{ZK} satisfying \eqref{modulation.2}--\eqref{modulation.3} for $\epsilon_0$ small enough. For $M \ge 4$, let $\psi_M$ be defined as in \eqref{psi0}. For $y_0>0$, $t_0 \in \mathbb R$ and $t \ge t_0$, we define 
\begin{equation} \label{tildemonotonicity.1}
\widetilde{I}_{y_0,t_0}(t)= \int u^2(x,t) \psi_M(\tilde{\tilde{x}}_1)dx \quad \text{where} \quad 
\tilde{\tilde{x}}_1=x_1-\rho_1(t)+\frac12(t-t_0)+y_0 \, .
\end{equation}
Then, 
\begin{equation} \label{tildemonotonicity.2}
\tilde{I}_{y_0,t_0}(t)-\tilde{I}_{y_0,t_0}(t_0) \lesssim e^{-y_0/M} \, ,
\end{equation}
for all $t \ge t_0$.
\end{lemma}

\begin{proof} Let $t \ge t_0$. Since $u$ is a solution of \eqref{ZK}, we deduce from \eqref{psi} and \eqref{modulation.5} that
\begin{equation} \label{tildemonotonicity.3}
\begin{split}
\frac{d}{dt} \tilde{I}_{y_0,t_0}(t) &= 2 \int u\partial_tu \psi_M(\tilde{\tilde{x}}_1)dx-(\rho_1'(t)-\frac12)\int u^2 \psi_M'(\tilde{\tilde{x}}_1)dx \\ 
& \le -\int\big(3(\partial_{x_1}u)^2+(\partial_{x_2}u)^2+\frac14u^2 \big)\psi_M'(\tilde{\tilde{x}}_1)dx+\frac23\int u^3\psi_M'(\tilde{\tilde{x}}_1)dx \, ,
\end{split}
\end{equation} 
provided $\epsilon_0$ is chosen small enough. We decompose the nonlinear term on the right-hand side of \eqref{tildemonotonicity.3} as 
\begin{equation} \label{tildemonotonicity.4}
\int u^3\psi_M'(\tilde{\tilde{x}}_1)dx=\int Q_c(\cdot-\rho)u^2\psi_M'(\tilde{\tilde{x}}_1)dx
+\int \big(u-Q_c(\cdot-\rho)\big)u^2\psi_M'(\tilde{\tilde{x}}_1)dx \, .
\end{equation}
By using \eqref{modulation.3} and the Sobolev embedding $H^1(\mathbb R^2) \hookrightarrow L^3(\mathbb R^2)$, we get that
\begin{equation} \label{tildemonotonicity.5}
\Big| \int \big(u-Q_c(\cdot-\rho)\big)u^2\psi_M'(\tilde{\tilde{x}}_1)dx  \Big| \lesssim K_0 \epsilon_0 \int u^2 \psi_M'(\tilde{\tilde{x}}_1)dx \, .
\end{equation}
Let $R_1$ be a positive number to be fixed later. To deal with the first term on the right-hand side of \eqref{tildemonotonicity.4}, we fix first consider the case where $|x-\rho(t)| > R_1$. It follows then from \eqref{elliptPDE1} that
\begin{equation} \label{tildemonotonicity.6}
\Big| \int Q_c(\cdot-\rho)u^2\psi_M'(\tilde{\tilde{x}}_1)dx \Big| \lesssim e^{-\delta R_1} \int u^2 \psi_M'(\tilde{\tilde{x}}_1)dx \, .
\end{equation}
In the case where $|x-\rho(t)| \le R_1$, we have 
\begin{displaymath} 
|\tilde{\tilde{x}}_1| \ge |y_0+\frac12(t-t_0)|-|x_1-\rho_1(t)| \ge y_0+\frac12(t-t_0)-R_1 \, ,
\end{displaymath}
so that
\begin{equation} \label{tildemonotonicity.7}
\Big| \int Q_c(\cdot-\rho)u^2\psi_M'(\tilde{\tilde{x}}_1)dx \Big| \lesssim e^{R_1/M}e^{-\big(y_0+\frac12(t-t_0)\big)/M} \int u_0^2dx \, ,
\end{equation}
since $\psi'_M(\tilde{\tilde{x}}_1) \lesssim e^{-|\tilde{\tilde{x}}_1|/M}$.
Therefore, we deduce gathering \eqref{tildemonotonicity.4}--\eqref{tildemonotonicity.7}, fixing the value of $R_1$ and choosing $\epsilon_0$ small enough that
\begin{equation} \label{tildemonotonicity.8} 
\frac23\Big| \int u^3\psi_M'(\tilde{\tilde{x}}_1)dx \Big| \le \frac18 \int \Big(|\nabla u|^2+u^2 \Big)\psi_M'(\tilde{\tilde{x}}_1) dx+Ce^{-\big(y_0+\frac12(t-t_0)\big)/M} \, ,
\end{equation}
where $C$ is a positive constant depending of $\|u_0\|_{L^2}^2$.

Therefore, we conclude the proof of \eqref{tildemonotonicity.2} by integrating \eqref{tildemonotonicity.3} between $t_0$ and $t$ and using \eqref{tildemonotonicity.8}.
\end{proof} 

We are now in position to prove that $\tilde{u}$ decays exponentially on the left in the $x_1$-direction.
\begin{lemma} \label{leftexpdecay}
 Let $M \ge 4$. Then, 
 \begin{equation} \label{leftexpdecay.1}
 \int  \tilde{u}^2(x+\tilde{\rho}(t),t) \big(1-\psi_M(x_1+y_0)\big)dx \lesssim e^{-y_0/M} \, ,
 \end{equation}
 for all $y_0>0$ and $t \in \mathbb R$. 
 
 Moreover, 
 \begin{equation} \label{leftexpdecay.2}
 \int_{x_2}\tilde{u}^2(x+\tilde{\rho}(t),t)dx_2 \lesssim e^{x_1/M}, \quad \forall \, x_1 \le 0, \,  t \in \mathbb R \, .
 \end{equation} 
\end{lemma}

\begin{proof} Fix $\tilde{t}_0 \in \mathbb R$ and $y_0>0$. First, we observe from \eqref{L2bCV.1} and \eqref{L2bCV.2} that
\begin{displaymath}
\int u^2(x,t_n+\tilde{t}_0)\psi_M(x_1-\rho_1(\tilde{t}_0+t_n)+y_0)dx \underset{n \to +\infty}{\longrightarrow} 
\int \tilde{u}^2(x,\tilde{t}_0) \psi_M(x_1-\tilde{\rho}_1(\tilde{t}_0)+y_0)dx \, .
\end{displaymath}
Thus, if we denote $\tilde{m}_0=\int \tilde{u}_0^2(x)dx$, there exists $n_0=n_0(y_0) \in \mathbb N$ such that
\begin{equation} \label{leftexpdecay.3}
\begin{split}
\int u^2(x,t_n&+\tilde{t}_0)\psi_M(x_1-\rho_1(\tilde{t}_0+t_n)+y_0)dx \\&\le \int \tilde{u}^2(x,\tilde{t}_0) \psi_M(x_1-\tilde{\rho}_1(\tilde{t}_0)+y_0)dx+e^{-y_0/M}\\ & 
=\tilde{m}_0-\int \tilde{u}^2(x,\tilde{t}_0) \big(1-\psi_M(x_1-\tilde{\rho}_1(\tilde{t}_0)+y_0)\big)dx+e^{-y_0/M} \, ,
\end{split}
\end{equation}
for all $n \ge n_0$. Note also that we used the fact that the $L^2$-norm of $\tilde{u}$ is conserved in time, since $\tilde{u}$ is a solution of \eqref{ZK}. 

Now, we use the monotonicity property on $u$ for times moving forward. Let $n' \ge n$ be such that $t_{n'} \ge t_n+\tilde{t}_0$. It follows from \eqref{tildemonotonicity.2} that
\begin{displaymath} 
\begin{split}
\int u^2(x,t_{n'}) &\psi_M\big(x_1-\rho_1(t_{n'})+\frac12(t_{n'}-(t_n+\tilde{t}_0))+y_0\big)dx 
\\ & \lesssim \int u^2(x,t_n+\tilde{t}_0) \psi_M\big(x_1-\rho_1(t_n+\tilde{t}_0)+y_0\big)dx +e^{-y_0/M} \, .
\end{split}
\end{displaymath}
This implies together with \eqref{leftexpdecay.3} that
\begin{equation} \label{leftexpdecay.4}
\begin{split}
\int u^2(x+&\rho(t_{n'}),t_{n'}) \psi_M\big(x_1+\frac12(t_{n'}-(t_n+\tilde{t}_0))+y_0\big)dx  \\ & 
\le \tilde{m}_0-\int \tilde{u}^2(x,\tilde{t}_0) \big(1-\psi_M(x_1-\tilde{\rho}_1(\tilde{t}_0)+y_0)\big)dx+Ke^{-y_0/M} \, ,
\end{split}
\end{equation}
as soon as $n' \ge n \ge n_0$ satisfies $t_{n'} \ge t_n +\tilde{t}_0$.

On the other hand, it follows from \eqref{L2CV.1} that 
\begin{displaymath}
\int_{x_1 >-A}u^2(x+\rho(t_{n'}),t_{n'})\psi_M\big(x_1+\frac12(t_{n'}-(t_n+\tilde{t}_0))+y_0\big)dx
\underset{n' \to +\infty}{\longrightarrow} \int_{x_1 > -A}\tilde{u}_0^2(x)dx \, ,
\end{displaymath}
for any $A>0$. We fix $A>0$ such that $\int_{x_1>-A}\tilde{u}_0^2(x)dx \ge \tilde{m}_0-\frac12e^{-y_0/M}$. Then, there exists $n'_1=n'_1(y_0,n) \in \mathbb N$ such that
\begin{equation} \label{leftexpdecay.5}
\begin{split}
\int u^2(x+&\rho(t_{n'}),t_{n'}) \psi_M\big(x_1+\frac12(t_{n'}-(t_n+\tilde{t}_0))+y_0\big)dx  
\ge \tilde{m}_0-e^{-y_0/M} \, ,
\end{split}
\end{equation}
for all $n' \ge n'_1$.

We conclude the proof of \eqref{leftexpdecay.1} gathering \eqref{leftexpdecay.4} and \eqref{leftexpdecay.5}.

We turn now to the proof of \eqref{leftexpdecay.2}. Fix some $y_0>0$. Since $\big(1-\psi_M(x_1+y_0)\big) \ge \frac12$ if $x_1 \le -y_0$, it follows from \eqref{leftexpdecay.1} that 
\begin{equation} \label{leftexpdecay.6}
\sup_{t \in \mathbb R}\int_{x_1 \le -y_0} \tilde{u}^2(x+\tilde{\rho}(t),t)dx \lesssim e^{-y_0/M} \, ,
\end{equation}
where the implicit constant does not depend on $y_0 >0$. 

Let us recall the following Sobolev inequality for functions $w=w(x_1)$ of one variable. It holds that
\begin{equation} \label{leftexpdecay.7}
\|w\|^2_{L^{\infty}(x_1\le-y_0)} \le 2\|w\|_{L^2(x_1\le-y_0)}\|\partial_{x_1}w\|_{L^2(x_1\le-y_0)} \, ,
\end{equation}
for any $w=w(x_1) \in H^1(\mathbb R)$ and $y_0 >0$.

We apply \eqref{leftexpdecay.7} to the function $w(x_1)=\Big(\int_{x_2}\tilde{u}^2(x_1+\tilde{\rho}_1(t),x_2+\tilde{\rho}_2(t),t)dx_2\Big)^{\frac12}$. Thus, we can bound $\|w\|^2_{L^{\infty}(x_1\le-y_0)} $ by
\begin{displaymath}
\Big(\int\limits_{ x_1\le-y_0}\tilde{u}^2(x+\tilde{\rho}(t),t)dx\Big)^{\frac12}
\Big(\int\limits_{x_1\le-y_0}\frac{\Big( \int_{x_2}\tilde{u}\partial_{x_1}\tilde{u}(x+\tilde{\rho}(t),t)dx_2\Big)^2}{\int_{x_2}\tilde{u}^2(x+\tilde{\rho}(t),t)dx_2} dx_1 \Big)^{\frac12} \, .
\end{displaymath}
It follows then from \eqref{leftexpdecay.6}, the Cauchy-Schwarz inequality and the global $H^1$ bound in $\tilde{u}$ that 
\begin{equation} \label{leftexpdecay.8}
\sup_{t\in \mathbb R, \, x_1 \le -y_0}\int_{x_2}\tilde{u}^2(x_1+\tilde{\rho}_1(t),x_2+\tilde{\rho}_2(t),t)dx_2 
\lesssim e^{-y_0/M} \, .
\end{equation}
This implies \eqref{leftexpdecay.2} since the implicit constant in \eqref{leftexpdecay.8} does not depend on $y_0$.
\end{proof}

Finally we give the proof of Proposition \ref{prop.AsymptStab}. 
\begin{proof}[Proof of Proposition \ref{prop.AsymptStab}]
The $H^1$ convergence on the right is given by Lemma \ref{H1CV}, while the exponential decay in the $x_1$ direction is obtained gathering \eqref{rightexpdecay.1} and \eqref{leftexpdecay.2}. Note that we also have that $\tilde{\rho}(0)=0$ from \eqref{L2bCV.2}
\end{proof}

\section{Stability of the sum of $N$-solitons} \label{sectionNSoliton}
 
\medskip
 
In this section we prove Theorem \ref{NSoliton}. 

\subsection{Reduction to a well-prepared case} First of all, after relabeling the set of scalings $(c_j^0)$ and the corresponding initial positions $(\rho^{j,0})$, we can assume that 
\be\label{NewCondition0}
0<c_1^0 < c_2^0< \cdots < c_N^0.
\ee
In what follows we will prove that there is a time $T_\#>0$, a constant $A_\#>0$ and another constant $\ga_0>0$, depending only on the parameters $(\rho^{j,0})$ and $(c_j^0)$, such that for some $\rho^{j,\#}= \rho^j(T_\#) \in \R^2$ one has
\be\label{NewCondition1}
\|u(T_\#) - \sum_{j=1}^N Q_{c_j^0}(x-\rho^{j,\#}) \|_{H^1} <A_\#(\ve +e^{-\ga_0 L}),
\ee
and now
\be\label{NewCondition2}
\rho^{1,\#}_1 < \rho^{2,\#}_1<\cdots < \rho^{N,\#}_1,
\ee
and \eqref{Condition} is also satisfied, in the sense that 
\be\label{NewCondition3}
\min \{ |\rho^{j,\#} -\rho^{k,\#} | \ : \ j\neq k \} >L.
\ee 
Let us define $T_\#$ as follows. Fix $L_0>0$ large and $L>L_0$ in \eqref{Condition}. We \emph{fix} $T_\#\geq 0$ such that  \eqref{NewCondition2} and \eqref{NewCondition3} are satisfied, where
\[
\rho^{j,\#}_1 := \rho^{j,0}_1 + c_j^0 \, T_\#.
\]
In other words,
\[
 \rho^{1,0}_1 + c_1^0 T_\# < \rho^{2,0}_1 + c_2^0 T_\# < \cdots < \rho^{N,0}_1 + c_N^0 T_\#.
\]
Consider the \emph{multi-soliton}
\[
R_0(x,t):= \sum_{j=1}^N Q_{c_j^0}(x_1-c_j^0 t - \rho^{j,0}_1,x_2-\rho^{j,0}_2).
\]
Then we have
\bee
S[R_0] & := & (R_0)_t +(\Delta R_0 +R_0^2)_{x_1} \\
&  =& \Big(R_0^2 -\sum_{j=1}^N Q_{c_j^0}^2(\cdot -c_j^0 t - \rho^{j,0}_1,\cdot -\rho^{j,0}_2) \Big)_{x_1} \\
& =& \Big(\sum_{i\neq j} Q_{c_i^0}^2(\cdot -c_j^0 t - \rho^{j,0}_1,\cdot -\rho^{j,0}_2)  Q_{c_j^0}^2(\cdot -c_j^0 t - \rho^{j,0}_1,\cdot -\rho^{j,0}_2) \Big)_{x_1}.
\eee
Under the assumption \eqref{Condition}, we have for all $t\geq 0$,
\[
\|S[R_0](t)\|_{H^1} \lesssim e^{-\ga_0 L},
\]
for some fixed constant $\ga_0>0$ only depending on the scalings $(c_j^0)$. The error function
\[
z_0(t) := u(t) -R_0(t) \in H^1
\]
satisfies (cf. \eqref{InitialCondition})
\[
\|z_0(0)\|_{H^1} <\ve,
\]
and the equation
\be\label{z0}
(z_0)_t +(\Delta z_0 + 2R_0 z_0 + z_0^2)_{x_1} + S[R_0]=0.
\ee
Now we establish some energy estimates. We have
\[
\frac{d}{dt}\Big( \frac 12 \int z_0^2 dx \Big) + \int (R_0)_{x_1} z_0^2 dx + \int S[R_0] z_0 dx =0,
\]
so that 
\[
\|z_0(t)\|_{L^2}^2 \lesssim \ve^2 + e^{-2\ga_0 L} + \int_0^t \|z_0(s)\|_{L^2}^2ds.
\]
We have for $t\in [0,T_\#]$,
\be\label{1Estimate}
\|z_0(t)\|_{L^2}^2 \lesssim e^{T_\#}(\ve + e^{-\ga_0 L}).
\ee
In order to obtain an estimate for the derivative of $z_0$, we have from \eqref{z0}
\[
(z_1)_t +(\Delta z_1 + 2R_0 z_1 + 2(R_0)_{x_1} z_0 + 2z_0 z_1)_{x_1} + (S[R_0])_{x_1}=0,
\]
where $z_1 := (z_0)_{x_1}$. This time we will have
\[
\frac{d}{dt}\Big( \frac 12 \int z_1^2 dx \Big) + \int (R_0)_{x_1} z_1^2 dx+ \int (z_0)_{x_1} z_1^2 dx + \int (S[R_0])_{x_1} z_1 dx =0.
\]
Using the Gronwall's inequality, we obtain once again, for all $t\in [0,T_\#]$,
\be\label{2Estimate}
\|z_1(t)\|_{L^2}^2 \lesssim e^{T_\#}(\ve + e^{-\ga_0 L}).
\ee
A similar estimate holds for $(z_0)_{x_2}$. From \eqref{1Estimate} and \eqref{2Estimate} we conclude (choose $A_\# \sim e^{T_\#}$).

\subsection{Proof in the well-prepared case}

Assume  \eqref{NewCondition0}, \eqref{NewCondition1}, \eqref{NewCondition2} and \eqref{NewCondition3}. We follow the Martel-Merle-Tsai paper \cite{MMT}, with some minor modifications. For technical reasons we need the following quantities
\be\label{ga0}
 \al_0:= A_\#(\ve +e^{-\ga_0 L}), \quad \ga_1\in(0,\ga_0).
\ee
The parameter $\ga_1$ is small but fixed, independent of $\ve$. Finally we define, for $A_0>1$ large to be fixed later and $\al>0$ small ($\al<\al_0$), the tubular neighborhood
\bee
V_{L}(\al, A_0) & :=& \Big\{ v \in H^1 \, : \,   \hbox{ there are } \ (\rho^j)_{j=1,\ldots,N} \in \R^{2N} \quad \hbox{such that } \\
 & & \qquad \| v - \sum_{j=1}^N  Q_{c_j^0}(\cdot - \rho^j) \|_{H^1} \leq A_0 (\al + e^{-\gamma_1L} )   \Big\}.
\eee
Since \eqref{ZK} is invariant by time translations, we can assume $T_\# =0$ in \eqref{NewCondition1}. Then we have $u(0)\in V_L(\al, 1) \subset V_L(\al, A_0)$. Moreover, by continuity of the $H^1$-flow map, we have that $u(t)\in V_L(\al, A_0)$ for all $t\in[0,T]$, for some $T=T(A_0)>0$.  The idea is to prove that for all $A_0$ large enough we can take $T=+\infty.$ (Recall that $A_0 (\al + e^{-\gamma_1L} )\leq \tilde A_0(\ve + e^{-\ga_1 L}) $, which leads to estimate \eqref{StabilityN}.)

\medskip

Let us assume that $T(A_0)<+\infty$. Then, by taking $\al_0$ smaller and $L_0$ larger if necessary, we have that there are parameters $(c_j(t),\rho^j(t)) \in \R_+\times \R^2$, $j=1,\ldots,N$, defined on $[0,T]$, and such that if
\[
R(x,t) := \sum_{j=1}^N \tilde Q_j(x,t), \quad \tilde Q_j(x,t):= Q_{c_j(t)} (x -\rho^j(t)),
\]
and
\[
z(x,t) := u(x,t) - R(x,t),
\]
then for all $t\in [0,T]$ and all $j=1,\ldots, N$,
\be\label{OrthoConds}
\int z \tilde Q_{j} dx= \int z \partial_{x_1}\tilde Q_{j} dx= \int z \partial_{x_2} \tilde Q_{j} dx=0.
\ee
The proof of this result is obtained by an standard application of the Implicit Function Theorem. An additional byproduct of this result is the estimate
\be\label{FirstEstimate}
\sum_{j=1}^N |c_j(t) -c_j(0)| + \|z(t)\|_{H^1} \lesssim A_0 (\al + e^{-\gamma_1L} ),
\ee
with a constant independent of time. Now we compute some energy estimates. Consider the energy $H(u)$ defined in \eqref{H}. It is not difficult to check that for some constant $\ga_1>0$ depending on $L$,
\bea\label{HH}
& & \Big| H(R+z)(t) -  H(R)(t)  - \frac12\int |\nabla z|^2(t)dx + \int R z^2(t) dx \Big|  \\
& & \nonu \qquad    \lesssim \|z(t)\|_{H^1}^3 + \|z(t)\|_{H^1} e^{-2\ga_1 t}.
\eea
Moreover, if $H_0:=\int (\frac12|\nabla Q|^2-\frac13Q^3) dx$, we have
\be\label{HH1}
\Big| H(R)(t)  - H_0\sum_{j=1}^N c_j^{2}(t) \Big|  \lesssim e^{-2\ga_1 t}.
\ee
On the other hand, consider the parameters 
\[
\sigma_j:= \frac12(c_j^0+c_{j-1}^0), \quad j=2,\ldots,N,
\]
and the perturbed mass
\be\label{masss}
M_j(t) := \frac12 \int u^2(x,t) \varphi_j(x,t)dx
\ee
where $\varphi_j(x,t):= \psi_A(x_1- \sigma_j t )$ and $\psi_A$ is defined in \eqref{psi0} for $A$ large but independent of $L$ and $\al$.  Note that thanks to \eqref{OrthoConds} the modified mass $M_j$ satisfies the identity
\be\label{Mj}
M_j(t) = M_0d_j(t)  + \frac12\int z^2 \varphi_j (t)dx +O( e^{-2\gamma_1t}), \quad M_0 :=\frac12\int Q^2,
\ee
where
\be\label{dj}
d_j(t) : =\sum_{k=j}^N c_k(t).
\ee
On the other hand, following the proof of Lemma \ref{nlL2monotonicity}, we have for $A>0$ large and $L>L_0$ large (depending on $A$) the monotonicity estimate
\be\label{MonotonicityMj}
M_j(t) -M_j(0) \lesssim  e^{-2\gamma_1 L},
\ee
with constants independent of $\al$ and $L$.  Let us define, for any $t\in [0,T]$, the quantity
\be\label{Dcj}
\hat\Delta c_j(t) := c_j(t)-c_j(0),
\ee
and more generally, for any time-dependent function $f(t),$
\[
\hat\Delta f(t) := f(t)-f(0).
\]
Then, using \eqref{Mj} and \eqref{MonotonicityMj} we have
\[
\hat\Delta d_j(t) \lesssim \|z(0)\|_{H^1}^2 +e^{-2\gamma_1t},
\]
or
\be\label{Deltadj}
|\hat\Delta d_j(t)| +\hat\Delta d_j(t)  \lesssim  \|z(0)\|_{H^1}^2 +e^{-2\gamma_1t}.
\ee
Now we estimate the difference $\hat\Delta c_j(t)$. First of all, note that for each $j$,
\[
| \hat\Delta [c_j^{2}](t) - 2c_j(0)\hat\Delta c_j(t) | = |\hat\Delta c_j(t)|^2.
\]
Therefore, using \eqref{HH1} and the previous identity,
\bee
\hat\Delta H(R)(t)  & =&   H_0\sum_{j=1}^N \Delta[ c_j^{2}](t) +O(e^{-2\ga_1 t})\\
& =&  2 H_0 \sum_{j=1}^N c_j(0)\hat\Delta c_j(t) + O\Big(\sum_{j=1}^N  |\hat\Delta c_j(t)|^2 + e^{-2\ga_1 t} \Big).
\eee
Next, we have
\bee
\sum_{j=1}^N c_j(0)\hat\Delta c_j(t) & = & \sum_{j=1}^{N-1} c_j(0)\hat\Delta ( d_j(t) -d_{j+1}(t)) +c_N(0)\hat\Delta  d_N(t) \\
& =&  \sum_{j=1}^{N-1} c_j(0)\hat\Delta d_j(t) - \sum_{j=2}^{N} c_{j-1}(0)\hat\Delta d_{j}(t)  +c_N(0)\hat\Delta  d_N(t) \\
& =&  \sum_{j=2}^{N} (c_j(0)- c_{j-1}(0))\hat\Delta d_{j}(t)     + c_1(0)\hat\Delta d_1(t)  .
\eee
Therefore, we use the identity 
\be\label{GroundState}
\frac{H_0}{M_0} =-\frac 12,
\ee
(see Appendix \ref{B} for a proof) to obtain
\bea
\hat\Delta H(R)(t)  & =&  -M_0 \sum_{j=1}^N c_j(0)\hat\Delta c_j(t) + O\Big(\sum_{j=1}^N  |\hat\Delta c_j(t)|^2 + e^{-2\ga_1 t} \Big) \nonu\\
& = &  -M_0   \sum_{j=2}^{N} (c_j(0)- c_{j-1}(0))\hat\Delta d_{j}(t)   -  M_0 c_1(0)\hat\Delta d_1(t)  \label{keyEst} \\
& & + O\Big(\sum_{j=1}^N  |\hat\Delta c_j(t)|^2 + e^{-2\ga_1 t} \Big). \nonu
\eea
Now we replace \eqref{Deltadj} and use the fact that $c_j(0)- c_{j-1}(0) >c_0>0$ for all $j$. We get
\[
\hat\Delta H(R)(t)  \geq  c_0  \sum_{j=1}^{N}|\hat\Delta d_j(t)| - C\Big( \|z(0)\|_{H^1}^2 +\sum_{j=1}^N  |\hat\Delta c_j(t)|^2 + e^{-2\ga_1 t} \Big)
\]
which, after using \eqref{HH}, implies that $|\hat\Delta d_j(t)|$ and $|\hat\Delta c_j(t)|$ have quadratic variation, for all $j=1,\ldots,N$. More precisely,
\[
|\hat\Delta c_j(t)| \leq \sum_{k=1}^{N}|\hat\Delta d_k(t)|  \lesssim  \|z(t)\|_{H^1}^2 +\sum_{k=1}^N  |\hat\Delta c_k(t)|^2 +e^{-2\ga_1 t},
\]
which implies that 
\be\label{quadraticVar}
|\hat\Delta c_j(t)| \lesssim  \|z(t)\|_{H^1}^2 +e^{-2\ga_1 t}.
\ee
Finally, using \eqref{keyEst}, \eqref{quadraticVar}, \eqref{Mj} and \eqref{MonotonicityMj},
\[
-\hat \Delta M_j(t) +\frac12\hat\Delta \int z^2 \varphi_j (t)dx +O( e^{-2\gamma_1L}) = -M_0 \hat \Delta d_j(t) , 
\]
and
\bee
\hat\Delta H(R)(t)  & = &    \sum_{j=2}^{N} (c_j(0)- c_{j-1}(0)) \Big[\frac12\hat\Delta \int z^2 \varphi_j (t)dx   -\hat \Delta M_j(t)  \Big]  \\
& &  + c_1(0) \Big[  \frac12\hat\Delta \int z^2 \varphi_1 (t)dx  -\hat \Delta M_1(t)\Big] \\
& & + O\Big(  \|z(t)\|_{H^1}^4 + e^{-2\ga_1 L} \Big). \\
& \geq &  \frac12 c_1(0)  \int z^2 \varphi_1 (t)dx +  \sum_{j=2}^{N} \frac12 (c_j(0)- c_{j-1}(0))  \int z^2 \varphi_j (t)dx \\
& & - C\Big(  \|z(0)\|_{H^1}^2 +  \|z(t)\|_{H^1}^4 + e^{-2\ga_1 L} \Big)\\
& \geq &    \sum_{j=1}^{N-1} \frac12 c_j(0)  \int z^2 (\varphi_j -\varphi_{j+1})(t)dx +\frac12 c_N(0)  \int z^2 \varphi_N(t)dx \\
& & - C\Big(  \|z(0)\|_{H^1}^2 +  \|z(t)\|_{H^1}^4 + e^{-2\ga_1 L} \Big),
\eee
which implies that 
\bee
& &  \frac12\int |\nabla z(t)|^2dx -\int R z^2(t)dx  \\
& & \qquad  + \sum_{j=1}^{N-1} \frac12 c_j(0)  \int z^2 (\varphi_j -\varphi_{j+1})(t)dx +\frac12 c_N(0)  \int z^2 \varphi_N(t)dx  \\
& & \qquad \lesssim \|z(0)\|_{H^1}^2 +  \|z(t)\|_{H^1}^3 + e^{-2\ga_1 L}.
\eee
A standard decomposition argument for $A>0$ large enough and \eqref{OrthoConds} allows to use the coercivity property associated to each soliton in the region $\sigma_j t \lesssim  x \lesssim \sigma_{j+1} t$ (see e.g. \cite[Lemma 4]{MMT}), and therefore we obtain
\[
\|z(t)\|_{H^1} \leq \frac12A_0 (\alpha + e^{-\ga_1 L} ).
\]
Finally, we use the decomposition
\bee
\|u(t) - \sum_{j=1}^N  Q_{c_j^0}(\cdot - \rho^j(t)) \|_{H^1} &  \leq &   \|z(t)\|_{H^1} + C\sum_{j=1}^N |\hat \Delta c_j(t)|\\
&  \leq & \frac34 A_0 (\alpha + e^{-\ga_1 L} ),
\eee
improving the original estimate, so that we have $u(t) \in V_L(\al,\frac34A_0)$. Therefore $T=+\infty$.

\appendix 
 
\section{Numerical Estimates for the Spectral Property}\label{A} 

Subject to Proposition \ref{SpectralProperty} on the
sign of an inner product, the gZK solitons are asymptotically stable.
Recall that the relevant quantity and its sign, \eqref{SpectralProperty1} are:
\begin{equation*}
\label{e:spec_quant}
(\mathcal{L}^{-1} \Lambda Q, \Lambda Q) < 0.
\end{equation*}
Having this condition yields coercivity of the bilinear form induced
by $\mathcal{L}$, on a subspace, which makes way for the proof of the
linear Liouville property.  

Conditions like \eqref{SpectralProperty1} have appeared in a variety of
works on soliton and blowup stability for gKdV, NLS, and other
equations.  While in dimension one, such
conditions can sometimes be proved analytically, due to our intimate
knowledge of the ${\sech}$ function, in dimensions two and higher, we
resort to computation.  This requires the computation of four
quantities, $Q$, $\Lambda Q$, $W$, and the inner product, where $W \in
H^1(\R^{d})$ is the solution of
\begin{equation}
\label{e:Wpde}
\mathcal{L} W = \Lambda Q
\end{equation}
Numerical computation of these quantities has been successfully
performed in several works on NLS, including
\cite{Asad:2011iz,Fibich:2006ea,Marzuola:2011td,Simpson:2011eh}.
Using the methods of
\cite{Asad:2011iz,Marzuola:2011td,Simpson:2011eh},  we will estimate
\eqref{e:spec_quant}, proving the desired property for certain values
of $p$ and $d$, including the quadratic case in dimension two.

\subsection{Computational Methods}

To solve \eqref{E}, \eqref{e:Wpde}, and compute \eqref{SpectralProperty1}, we
first remark that since $Q$ is radially symmetric, so is $\Lambda Q$.
Thus, $W$ is also radially symmetric, and we are reduced to solving
singular boundary value problems
\begin{gather}
\label{e:Qrad}
-Q'' - \frac{d-1}{r}Q' + Q - Q^p = 0, \quad Q'(0) = 0, \quad
\lim_{r\to\infty} Q(r) = 0,\\
\label{e:Wrad}
-W'' - \frac{d-1}{r}W' + W -p Q^{p-1} W =\frac{1}{p-1}Q + \frac12 r Q', \quad W'(0) = 0, \quad
\lim_{r\to \infty} W(r) =0,\\
\label{e:iprad}
(\mathcal{L}^{-1} \Lambda Q, \Lambda Q) = C_{d} \int_0^\infty (\Lambda
Q)(r) W(r) r^{d-1}dr.
\end{gather}
$C_d$ is the surface area of the ${d-1}$ dimensional sphere.
To make these problems computationally tractable, we truncate the
domain to $(0,\rmax)$, where $\rmax$ is taken sufficiently
large.  The asymptotic of $Q$ is well known, with 
\begin{equation}
\label{e:Qasympt}
Q\propto r^{-(d-1)/2} e^{-r}.
\end{equation}
Therefore,  for the truncated domain, we introduce the Robin boundary condition,
\begin{equation}
\label{e:Qrobin}
Q'(\rmax) + \frac{d-1 + 2\rmax}{2\rmax}Q(\rmax) = 0.
\end{equation}
For large values of $r$, a dominant balance
of  \eqref{e:Wrad} is
\begin{equation}
-W''  + W \approx r Q',
\end{equation}
from which we infer that
\begin{equation}
\label{e:Wasympt}
W\propto r^{(5-d)/2} e^{-r}.
\end{equation}
This motivate the Robin boundary condition
\begin{equation}
\label{e:Wrobin}
W'(\rmax) + \frac{d-5 + 2 \rmax}{2\rmax} W(\rmax) = 0.
\end{equation}

We thus solve the equations \eqref{e:Qrad} and \eqref{e:Wrad}, with
approximate boundary conditions \eqref{e:Qasympt} and
\eqref{e:Wasympt}.  To compute the inner product, we introduce the
function $\nu(r)$, solving the ODE
\begin{equation}
\label{e:nuODE}
\nu' = (\Lambda
Q)(r) W(r) r^{d-1}, \quad \nu(0).
\end{equation}
Then 
\begin{equation}
\nu(\rmax) = \int_0^{\rmax} (\Lambda
Q)(r) W(r) r dr \approx
C_d^{-1}(\mathcal{L}^{-1} \Lambda Q, \Lambda Q).  
\end{equation}
\eqref{e:nuODE}, though trivial, is introduce so that we can solve
this system, in concert, as a coupled first order system using {\sc
  Matlab}'s {\tt bvp4c}, a two point boundary value problem solver.
Since $C_d>0$, we omit it in our calculations.

\subsection{Numerical Results}

\subsubsection{Case of $p=2$ in Dimension $d=2$}

\begin{figure}
 \subfigure[$Q$ and $W$ Profiles]{\includegraphics[width=6.2cm]{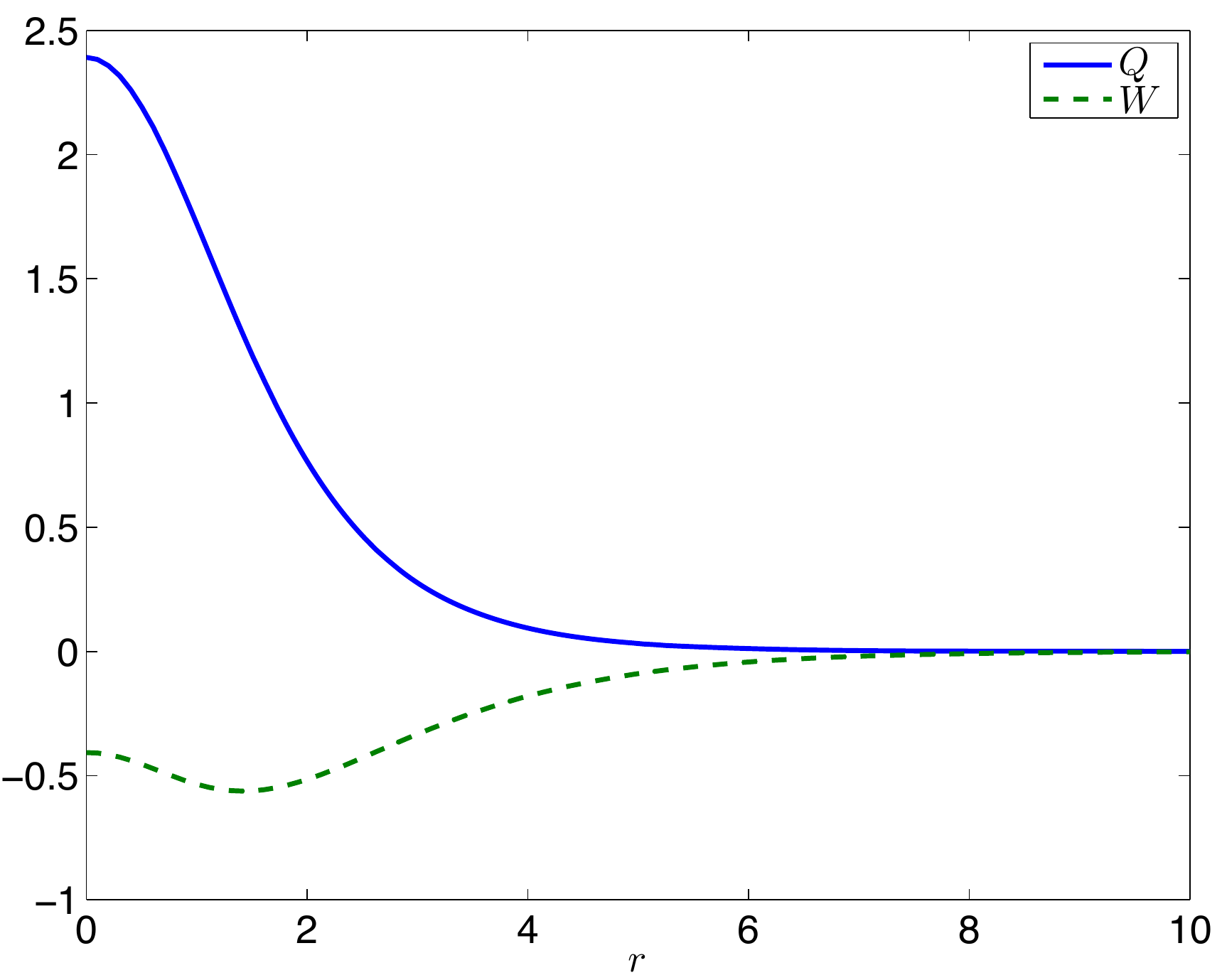}}
 \subfigure[$Q$ and $W$ Profiles ($\log$ Scale)]{\includegraphics[width=6.2cm]{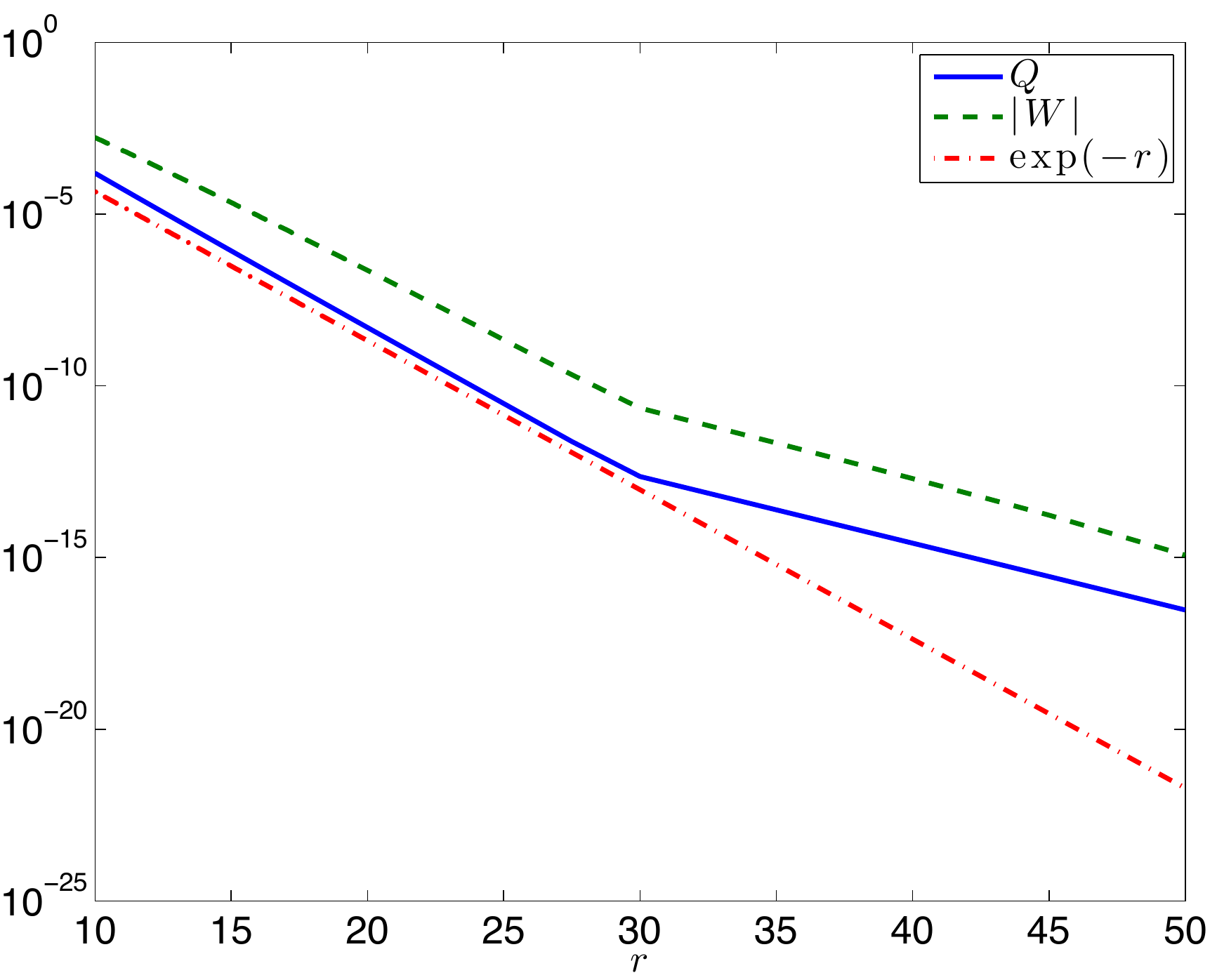}}

 \subfigure[$\nu$ Profile]{\includegraphics[width=6.2cm]{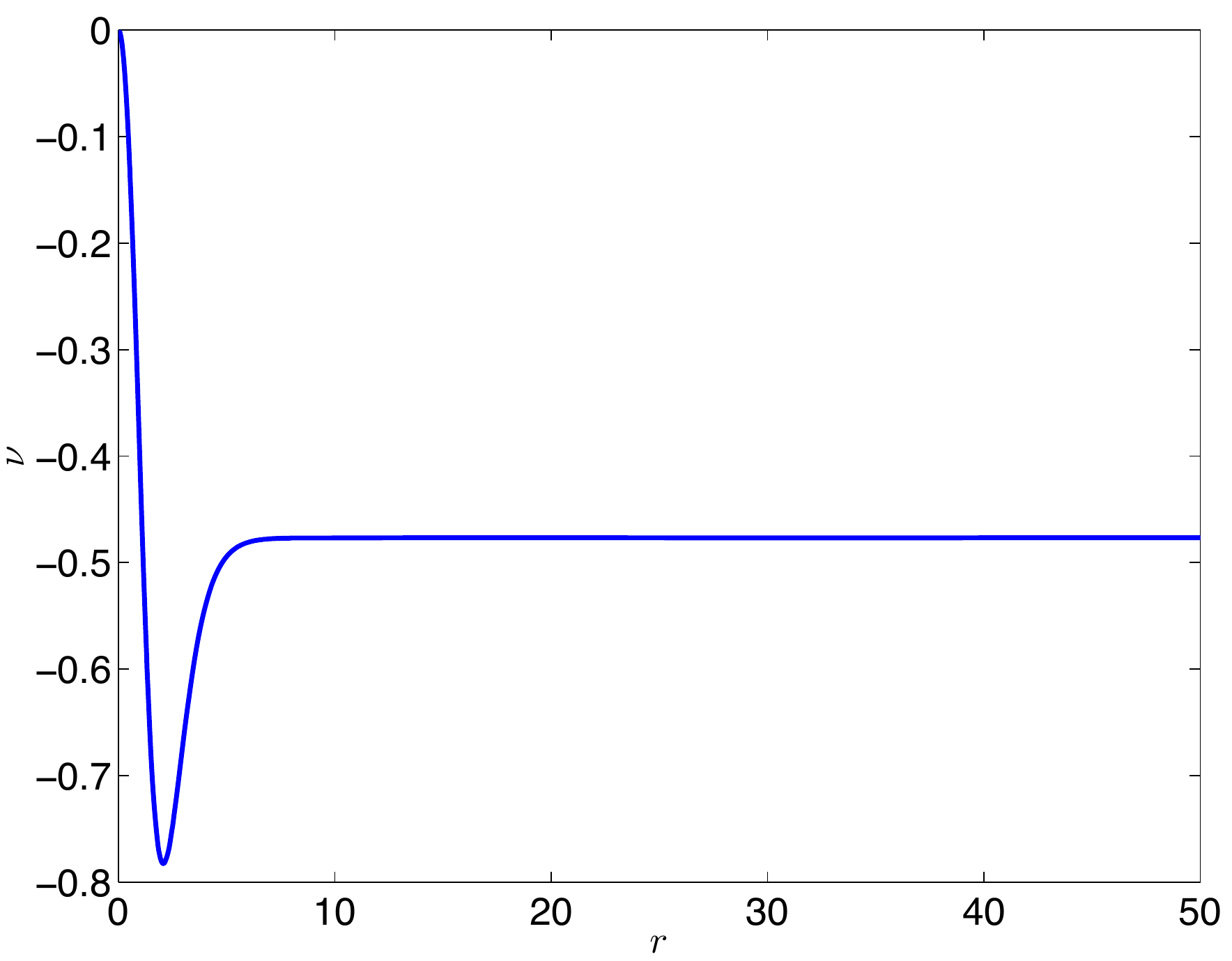}}
\caption{Computed $Q$ and $W$ profiles for $p=2$ in $d=2$, along with the $\nu(r)$
  function defined by \eqref{e:nuODE}.  $Q$ and $W$ are exponentially
  small and $\nu$ has plateaued at a fixed, negative, value.}
\label{f:QW}
\end{figure}

We apply the algorithm with $\rmax =50$ and absolute and
relative error tolerances of $10^{-8}$ and $10^{-10}$ to the $p=2$
case in $d=2$.  We find that \eqref{SpectralProperty1} is indeed negative.
In Figure \ref{f:QW}, we plot $Q$, $W$ and their asymptotics in (a) and (b), along with
the computed $\nu$ in (c).  $Q$ and $W$ are both vanishing exponentially,
and  $\nu$ has stabilized to a fixed, negative, value, $\nu(\rmax)
\approx -0.476741$.

\subsubsection{Other Cases}
\begin{figure}
  \subfigure[$d=1$]{\includegraphics[width=6.2cm]{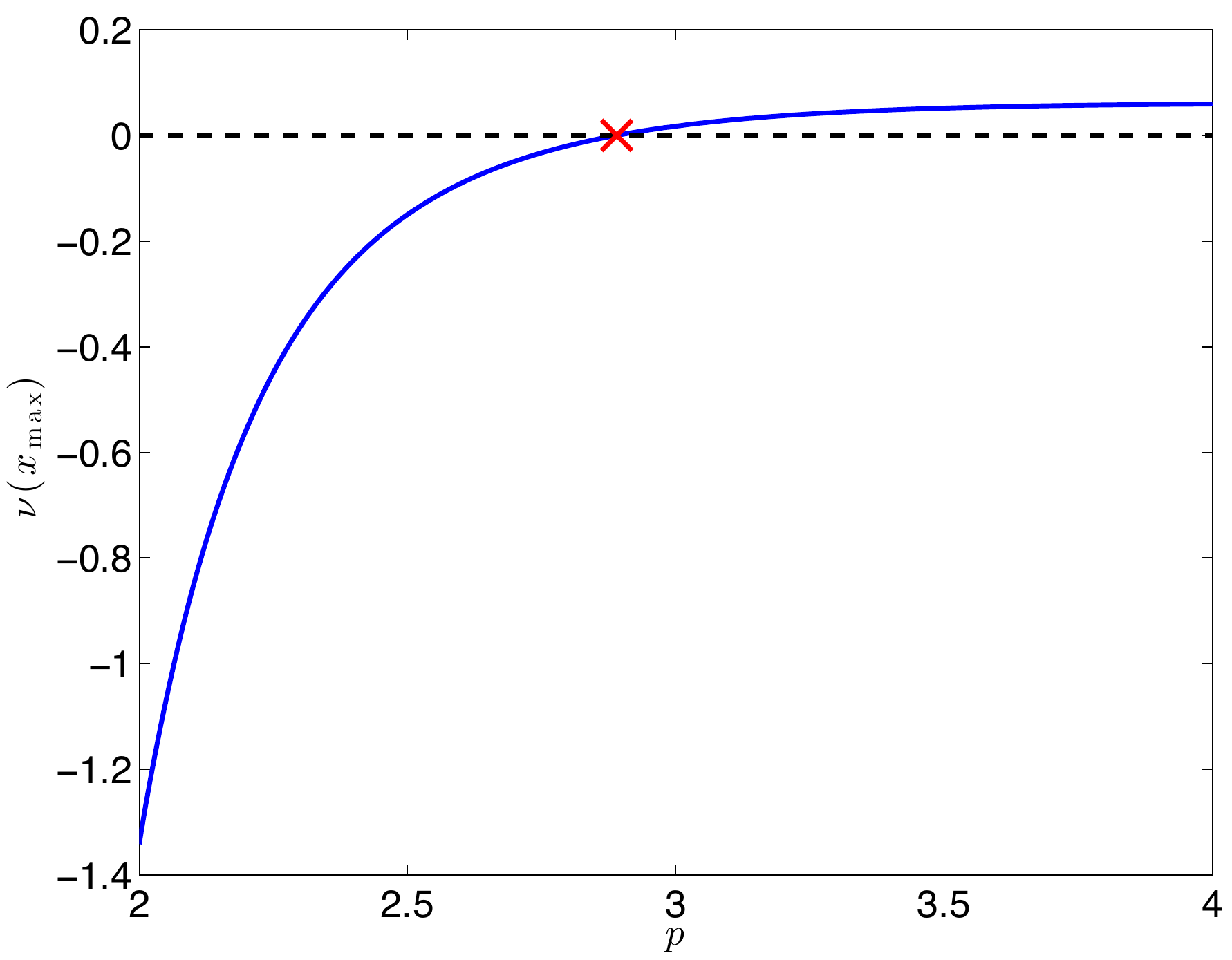}}
  \subfigure[$d=2$]{\includegraphics[width=6.2cm]{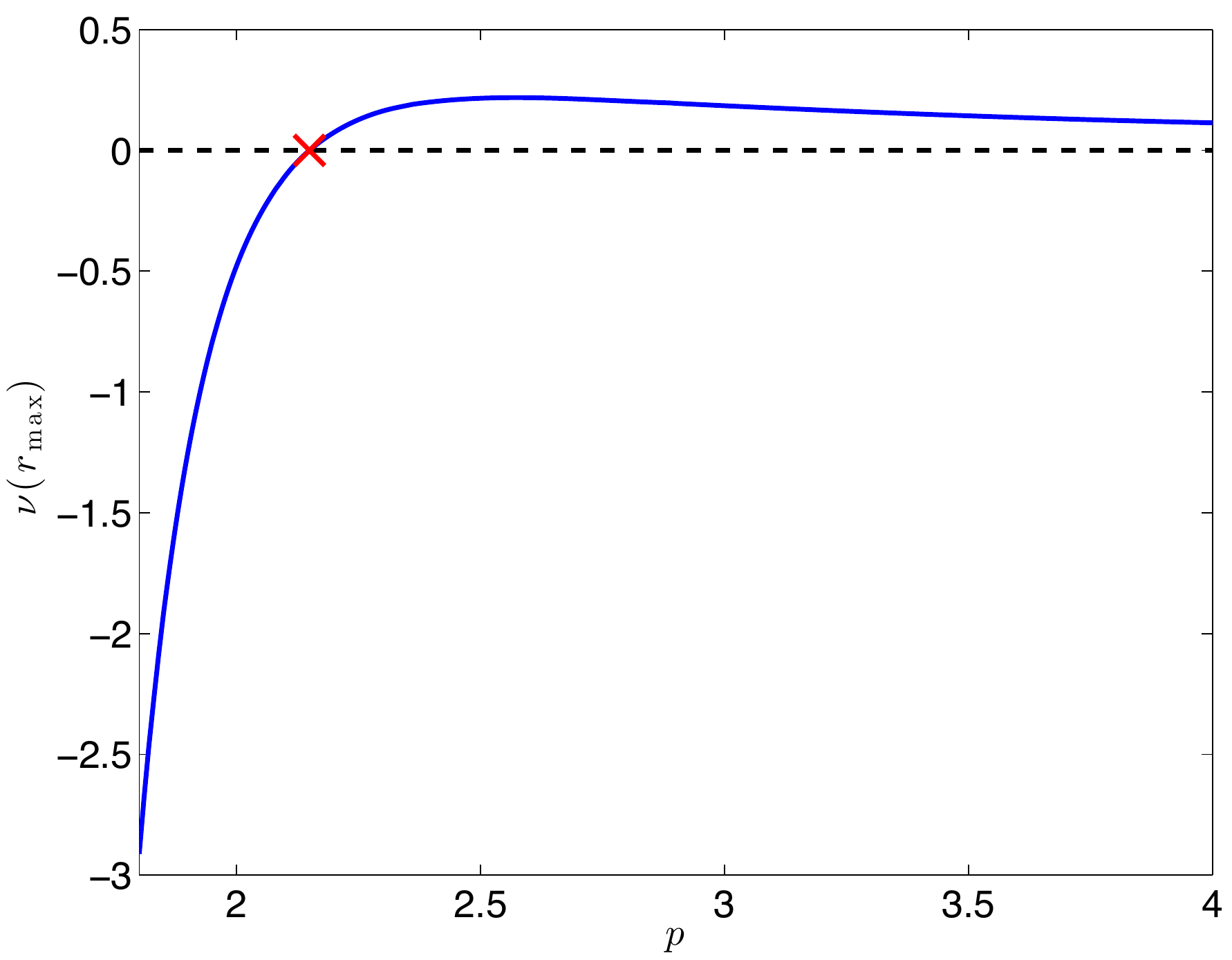}}

  \subfigure[$d=3$]{\includegraphics[width=6.2cm]{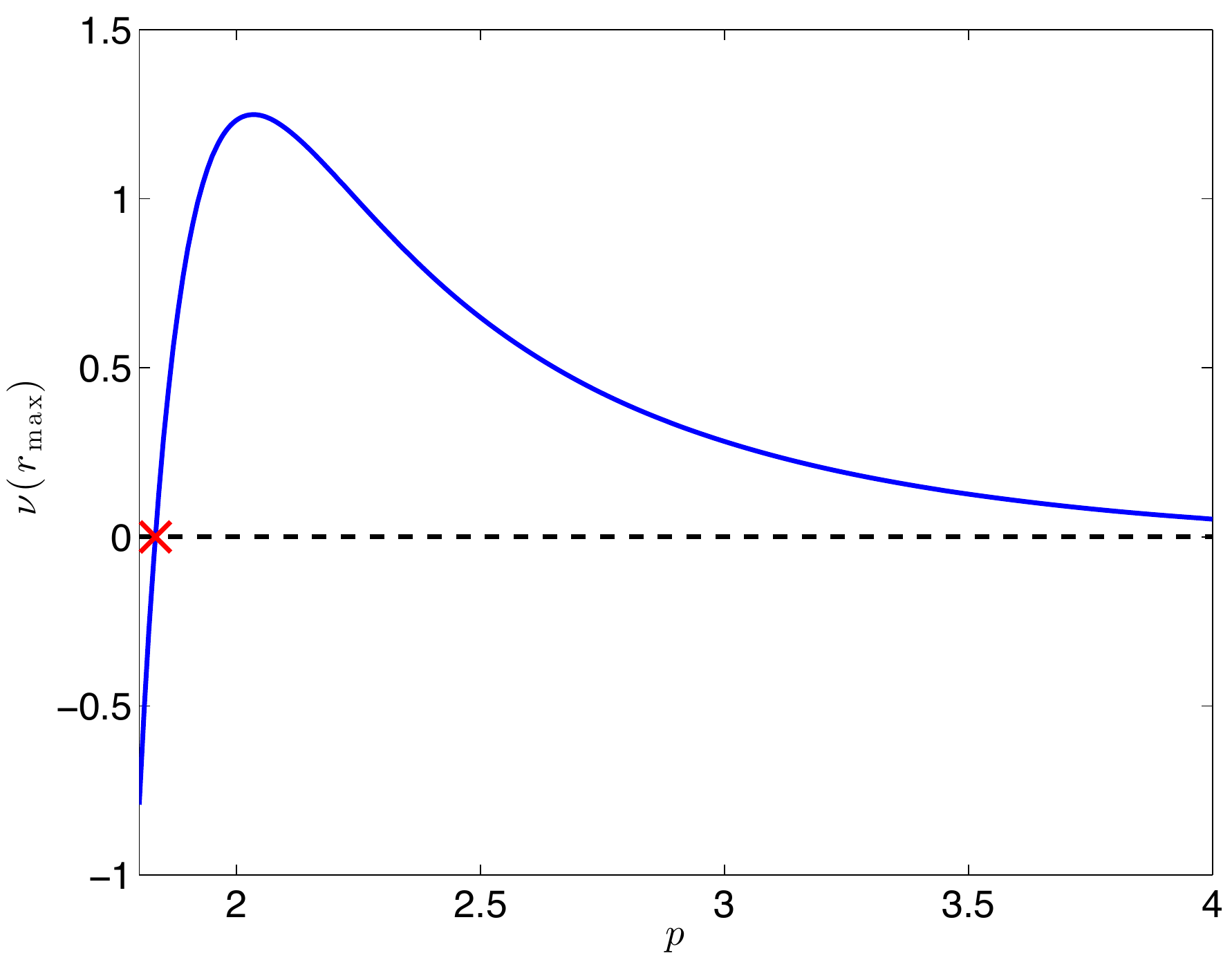}}
\caption{Estimated values of \eqref{e:spec_quant} in dimensions one, two,
  and three for a range of nonlinearities, $p$.  The zero crossings are
  located at 2.8899 ($d=1$), 2.1491 ($d=2$), and 1.8333 ($d=3$).}
\label{f:spec_vals}
\end{figure}

We can repeat this computation for other values of $p$ and in other
dimensions.  In Figure \ref{f:spec_vals}, we plot the $\nu(\rmax)$, the estimate
of \eqref{SpectralProperty1} for a range of $p$ in dimensions two and
three.  For a range of computed $p$, $p\in [1.8, 2.1491)$ in $d=2$ and
$p\in [1.8, 1.8333)$ in $d=3$, the sign is negative, as needed for stability.  For higher, more supercritical, of $p$,
\eqref{SpectralProperty1} is positive, and the result is inconclusive.
That the calculation is inconclusive is affirmed by the corresponding
computation made in dimension $d=1$, also shown in the figure.  For
this case, the zero crossing is at $p=2.8899$, excluding the cubic
nonlinearity.  But it is known $p=3$ in $d=1$ is asymptotically stable
in the sense discussed here, \cite{MM1,MM5}.

\section{Proof of \eqref{GroundState}}\label{B}

In this short appendix we prove \eqref{GroundState}. Recall that $Q$ is the unique radial solution of \eqref{E} in $H^1(\R^2)$. We multiply \eqref{E} by $Q$ and integrate to get
\be\label{E11}
-\int |\nabla Q|^2 -\int Q^2 +\int Q^3 =0.
\ee
Now we prove two different identities dealing with the gradient term. First of all, multiply \eqref{E} by $\partial_{x_1} Q$, we obtain
\[
\partial_{x_1} \Big[ \frac12 (\partial_{x_1}Q)^2  -\frac12 Q^2 +\frac13 Q^3 \Big] + \partial_{x_1} Q\partial_{x_2}^2 Q =0,
\]
which implies that the term $\partial_{x_1} Q\partial_{x_2}^2 Q$ has finite integral on each subinterval of $\R_{x_1}$. With a slight abuse of notation, we have
\[
\frac12 (\partial_{x_1}Q)^2  -\frac12 Q^2 +\frac13 Q^3 + \int_{-\infty}^{x_1}   \partial_{x_1} Q\partial_{x_2}^2 Q =0.
\]
Note that the first three terms above are integrable. Therefore $ \int_{-\infty}^{x_1}   \partial_{x_1} Q\partial_{x_2}^2 Q $ is integrable on $\R^2$ and we obtain
\[
\int_{x_1,x_2} \frac12 (\partial_{x_1}Q)^2  -\frac12 Q^2 +\frac13 Q^3 + \int_{x_1,x_2}\int_{-\infty}^{x_1}   \partial_{x_1} Q\partial_{x_2}^2 Q =0.
\]
Now we use the Fubini's Theorem to compute the last term above. We have
\bee
 \int_{x_1,x_2}\int_{-\infty}^{x_1}   \partial_{x_1} Q\partial_{x_2}^2 Q & =&  \int_{x_1}\int_{-\infty}^{x_1}  \int_{x_2} \partial_{x_1} Q\partial_{x_2}^2 Q\\
 & =& -\int_{x_1}\int_{-\infty}^{x_1}  \int_{x_2} \partial_{x_1,x_2} Q\partial_{x_2} Q \\
 & =& -\frac12   \int_{x_1,x_2}\int_{-\infty}^{x_1}  \partial_{x_1} \Big[(\partial_{x_2} Q)^2 \Big] \\
 & =& -\frac12 \int_{x_1,x_2} (\partial_{x_2} Q)^2.
\eee
We finally obtain
\[
\int  (\partial_{x_1}Q)^2  -(\partial_{x_2}Q)^2   - Q^2 +\frac23 Q^3  =0.
\]
Now we interchange the roles of  $x_1$ and $x_2$ to get a second estimate:
\[
\int  (\partial_{x_2}Q)^2  -(\partial_{x_1}Q)^2   - Q^2 +\frac23 Q^3  =0,
\]
which implies that 
\[
\int Q^2 - \frac23 Q^3 =0,
\]
and
\[
\int  (\partial_{x_1}Q)^2 =\int  (\partial_{x_2}Q)^2,
\]
as expected since $Q$ is radially symmetric. Replacing in \eqref{E11},  we get
\[
\int (\partial_{x_1}Q)^2 = \frac12 \Big[ \frac32 \int Q^2 -\int Q^2\Big] = \frac14 \int Q^2.
\]
Finally, we compute $H_0$. Using the previous identities we have
\bee
H_ 0 & =& \frac12 \int |\nabla Q|^2 -\frac13 \int  Q^3 \\
& =&  \frac12 \big(2 \cdot \frac14 \int Q^2 \big)-\frac13\cdot \frac32 \int Q^2 = -\frac14 \int Q^2 = -\frac12 M_0,
\eee
as desired.

\section{Linear waves versus Asymptotic stability in the energy space}\label{CC}

In this small section we prove that Remarks \ref{1R} and \ref{2R_} are formally sharp, by using linear waves at infinity. Indeed, consider the linear dynamics
\[
u_t + \partial_{x_1}\Delta u =0,
\]
and take $u=\exp( i(k_1 x_1 + k_2 x_2 -w t))$, the standard front wave. Then we compute $w$ in terms of $k_1$ and $k_2$. The result is
\[
w(k_1,k_2)= -(  k_1^3 + k_1 k_2^2 ).
\]
Now we compute the velocity group, which is the vector $\nabla w$. The result is
\[
\nabla w =-(3k_1^2 + k_2^2 , 2 k_1k_2 )^T,
\]
which is a vector with negative $x_1$-coordinate, but the $x_2$-coordinate depends on the sign of $k_1$ and $k_2$.  Without loss of generality, we can assume $k_1>0$, $k_2<0$. Now we compute the minimal angle $\theta$ for which
\[
-2 k_1k_2 = R \cos \theta, \quad \hbox{and } \quad 3k_1^2 + k_2^2 =R \sin \theta.
\]
It turns out that the angle is given by
\[
\min \frac{3k_1^2 + k_2^2}{ 2 |k_1| |k_2| } = \tan \theta,
\]
but
\[
\frac{3k_1^2 + k_2^2}{ 2 |k_1| |k_2| }  = \frac12 (3 r + \frac 1r),  \quad  r:=  \frac{|k_1|}{|k_2|}>0.
\]
Note that we always have
\[
3 r + \frac 1r \geq  2\sqrt{3},
\]
so
\[
\min \frac{3k_1^2 + k_2^2}{ 2 |k_1| |k_2| } \geq \sqrt{3} = \tan \frac \pi 3.
\]
as expected. A similar result holds for the three dimensional case. We thank Didier Smets for this interesting remark.

\bigskip

\section{Proof of Theorem \ref{NonExistence}}\label{D}

\medskip

Assume that $v \in H^1(\R^d)$ is a nontrivial solitary wave satisfying \eqref{SW} satisfying $\partial_{x_1}^{-1}\partial_{x_j} v\in L^2(\R^d)$ for all $j \in \mathbb Z_+ \cap [2,d]$ and $c_j \neq 0$ for some $ j \in \mathbb Z_+ \cap [2,d]$. Since \eqref{SW} is invariant by rotation in the $d-1$ variables $(x_2,\cdots,x_{d-1})$, we can always assume that $c_2\neq 0$ and $c_3=\cdots=c_d=0$.

Multiplying \eqref{SW} by $x_2 \partial_{x_1} v$ and integrating by parts (this process can be made rigorous by using a suitable cut-off approximation), we find
\[
\int \partial_{x_1}v \partial_{x_2}v = -\frac12 c_2 \int v^2.
\] 
On the other hand, we multiply \eqref{SW} by $x_1 \partial_{x_2} v$ to obtain
\[
\int \partial_{x_1}v \partial_{x_2}v = \frac12 c_2 \int (\partial_{x_1}^{-1} \partial_{x_2} v)^2,
\] 
which is a contradiction, unless $c_2 =0$ or $v\equiv 0.$

%\clearpage

\merci{The authors would like to thank the University of Chicago, the
  Ecole Polytechnique and the Instituto de Matem\'atica Pura e
  Aplicada (IMPA) for the kind hospitality during the elaboration and
  the redaction of this work. Moreover, D.P. gratefully acknowledges
  support of the European Council Advanced Grant no 291214,
  BLOWDISOL.  G.S. was supported by the United States National Science
  Foundation Grant DMS-1409018. C.M.  was partially supported by the project ERC 291214 BLOWDISOL, and by Chilean research grants FONDECYT 1150202, Fondo Basal CMM-Chile, and Millennium Nucleus Center for Analysis of PDE NC130017}

\bibliographystyle{amsplain}

\end{document}